\documentclass[10pt,reqno]{amsart}
\usepackage{amsmath,amsfonts, amssymb, amsthm}
  \usepackage{paralist}
  \usepackage{graphics} 
  \usepackage{epsfig} 
\usepackage{graphicx}  \usepackage{epstopdf}
 \usepackage[colorlinks=true]{hyperref}
\hypersetup{urlcolor=blue, citecolor=red}

\newtheorem{theorem}{Theorem}[section]

\newtheorem{lemma}[theorem]{Lemma}

\newtheorem{conjecture}{Conjecture}[section]

\theoremstyle{definition}

\newtheorem{remark}{Remark}[section]

\def\pmod #1{\ ({\rm{mod}}\ #1)}
\def\Z{\Bbb Z}
\def\N{\Bbb N}

\def\Q{\Bbb Q}

\def\l{\left}
\def\r{\right}
\def\bg{\bigg}
\def\({\bg(}
\def\){\bg)}
\def\t{\text}
\def\f{\frac}

\def\ls{\leq}
\def\gs{\geq}

\def\bi{\binom}

\def\ve{\varepsilon}

\def\eq{\equiv}

\def\da{\delta}

\def\Proof{\noindent{\it Proof}}

\begin{document}
\hbox{Electron. Res. Arch. 28 (2020), no.\,3, 1273--1342.}
\medskip

\title[New series for powers of $\pi$ and related congruences]
      {New series for powers of $\pi$\\ and related congruences}
\author[Zhi-Wei Sun]{Zhi-Wei Sun}


\address{Department of Mathematics, Nanjing
University, Nanjing 210093, People's Republic of China}
\email{zwsun@nju.edu.cn}

\keywords{Congruences, binomial coefficients,
Ramanujan-type series for $1/\pi$, binary quadratic forms, symbolic computation.
\newline \indent Supported by the Natural Science Foundation of China (grant no. 11971222).}

\begin{abstract}
Via symbolic computation we deduce 97 new type series for powers of $\pi$ related to Ramanujan-type series.
Here are three typical examples:
 $$\sum_{k=0}^\infty\frac{P(k)\binom{2k}k\binom{3k}k\binom{6k}{3k}}{(k+1)(2k-1)(6k-1)(-640320)^{3k}}
=\frac{18\times557403^3\sqrt{10005}}{5\pi}$$
with
\begin{align*}P(k)=&637379600041024803108 k^2  + 657229991696087780968 k
\\&+ 19850391655004126179,\end{align*}
$$\sum_{k=1}^\infty\f{(3k+1)16^k}{(2k+1)^2k^3\bi{2k}k^3}=\f{\pi^2-8}2,$$
and $$\sum_{n=0}^\infty\frac{3n+1}{(-100)^n}\sum_{k=0}^n{n\choose k}^2T_k(1,25)T_{n-k}(1,25)=\frac{25}{8\pi},$$
where the generalized central trinomial coefficient $T_k(b,c)$ denotes the coefficient of $x^k$ in the expansion of $(x^2+bx+c)^k$.
 We also formulate a general characterization of rational Ramanujan-type series for $1/\pi$
 via congruences, and pose 117 new conjectural series for powers of $\pi$ via looking for corresponding congruences. For example, we conjecture that
  $$\sum_{k=0}^\infty\frac{39480k+7321}{(-29700)^k}T_k(14,1)T_k(11,-11)^2=\frac{6795\sqrt5}{\pi}.$$
  Eighteen of the new series in this paper involve
 some imaginary quadratic fields with class number $8$.
\end{abstract}
\maketitle

\section{Introduction and our main results}

The classical rational Ramanujan-type series for $\pi^{-1}$ (cf. \cite{BB,Be,CC,R}
and a nice introduction by S. Cooper \cite[Chapter 14]{Co17}) have the form
\[\sum_{k=0}^\infty\f{bk+c}{m^k}a(k)=\f{\lambda\sqrt d}{\pi},\tag{$*$}\]
where $b,c,m$ are integers with $bm\not=0$, $d$ is a positive squarefree number,
 $\lambda$ is a nonzero rational number, and $a(k)$ is one of the products
$$\bi{2k}k^3,\ \bi{2k}k^2\bi{3k}k,\ \bi{2k}k^2\bi{4k}{2k},\ \bi{2k}k\bi{3k}k\bi{6k}{3k}.$$
In 1997 Van Hamme \cite{vH} conjectured that such a series $(*)$ has a $p$-adic analogue of the form
 $$\sum_{k=0}^{p-1}\f{bk+c}{m^k}a(k)\eq cp\l(\f{\ve_d d}p\r)\pmod{p^3},$$
 where $p$ is any odd prime with $p\nmid dm$ and $\lambda\in\Z_p$,
 $\ve_1\in\{\pm1\}$ and $\ve_d=-1$ if $d>1$. (As usual, $\Z_p$ denotes the ring of all $p$-adic integers, and $(\f{\cdot}p)$ stands for the Legendre symbol.) W. Zudilin \cite{Zu} followed Van Hamme's idea to provide more concrete examples. Sun \cite{S11e} realized that many Ramanujan-type congruences
 are related to Bernoulli numbers or Euler numbers.
 In 2016 the author \cite{S-u} thought that all Ramanujan-type congruences have their extensions
 like
$$\f{\sum_{k=0}^{pn-1}(21k+8)\bi{2k}k^3-p\sum_{k=0}^{n-1}(21k+8)\bi{2k}k^3}{(pn)^3\bi{2n}n^3}
\in\Z_p,$$
where $p$ is an odd prime, and $n\in\Z^+=\{1,2,3,\ldots\}$. See Sun \cite[Conjectures 21-24]{S19}
for more such examples and further refinements involving Bernoulli or Euler numbers.

During the period 2002--2010, some new Ramanujan-type series of the form $(*)$ with $a(k)$ not a product of three nontrivial parts were found (cf. \cite{CCL,ChCo,Co,Rogers}). For example, H. H. Chan, S. H. Chan and Z. Liu \cite{CCL}
proved that
$$\sum_{n=0}^\infty\f{5n+1}{64^n}D_n=\f 8{\sqrt3\,\pi},$$
where $D_n$ denotes the Domb number $\sum_{k=0}^n\bi nk^2\bi{2k}k\bi{2(n-k)}{n-k}$;
Zudilin \cite{Zu} conjectured its $p$-adic analogue:
$$\sum_{k=0}^{p-1}\f{5k+1}{64^k}D_k\eq p\l(\f p3\r)\pmod{p^3}\quad \t{for any prime}\ p>3.$$
The author \cite[Conjecture 77]{S19} conjectured further that
$$\f1{(pn)^3}\(\sum_{k=0}^{pn-1}\f{5k+1}{64^k}D_k-\l(\f p3\r)p\sum_{k=0}^{n-1}\f{5k+1}{64^r}D_k\)\in\Z_p$$
for each odd prime $p$ and positive integer $n$.

Let $b,c\in\Z$. For each $n\in\N=\{0,1,2,\ldots\}$,
we denote the coefficient of $x^n$ in the expansion of $(x^2+bx+c)^n$ by $T_n(b,c)$, and call it
a {\it generalized central trinomial coefficient}.
In view of the multinomial theorm, we have
$$T_n(b,c)=\sum_{k=0}^{\lfloor n/2\rfloor}\bi n{2k}\bi{2k}kb^{n-2k}c^k
=\sum_{k=0}^{\lfloor n/2\rfloor}\bi nk\bi{n-k}k b^{n-2k}c^k.$$
Note also that
$$T_0(b,c)=1,\ \ T_1(b,c)=b,$$
and
$$(n+1)T_{n+1}(b,c)=(2n+1)bT_n(b,c)-n(b^2-4c)T_{n-1}(b,c)$$
for all $n\in\Z^+.$
Clearly, $T_n(2,1)=\bi{2n}n$ for all $n\in\N$. Those $T_n:=T_n(1,1)$ with $n\in\N$
are the usual central trinomial coefficients, and they play important roles in enumerative combinatorics.
We view $T_n(b,c)$ as a natural generalization of central binomial and central trinomial coefficients.

For $n\in\N$ the Legendre polynomial of degree $n$
is defined by
$$P_n(x):=\sum_{k=0}^n\bi nk\bi{n+k}k\l(\f{x-1}2\r)^k.$$
It is well-known that if $b,c\in\Z$ and $b^2-4c\not=0$ then
$$T_n(b,c)=(\sqrt{b^2-4c})^nP_n\l(\f{b}{\sqrt{b^2-4c}}\r)\quad\t{for all}\ n\in\N.$$
Via the Laplace-Heine asymptotic formula for Legendre
polynomials, for any positive real numbers $b$
and $c$ we have
$$T_n(b,c)\sim\f{(b+2\sqrt{c})^{n+1/2}}{2\root{4}\of c\sqrt{n\pi}}\quad\t{as}\ n\to+\infty$$
(cf. \cite{S14c}).
For any real numbers $b$ and $c<0$, S. Wagner \cite{Wag} confirmed the author's conjecture that  $$\lim_{n\to\infty}\root n\of{|T_n(b,c)|}=\sqrt{b^2-4c}.$$

In 2011, the author posed over 60 conjectural series for $1/\pi$ of
the following new types with $a,b,c,d,m$ integers and $mbcd(b^2-4c)$ nonzero (cf. Sun \cite{S-11,S14c}).

\ \ {\tt Type I}. $\sum_{k=0}^\infty\f{a+dk}{m^k}\bi{2k}k^2T_k(b,c)$.

\ \ {\tt Type II}.
$\sum_{k=0}^\infty\f{a+dk}{m^k}\bi{2k}k\bi{3k}kT_k(b,c)$.

\ \ {\tt Type III}.
$\sum_{k=0}^\infty\f{a+dk}{m^k}\bi{4k}{2k}\bi{2k}kT_k(b,c)$.

\ \ {\tt Type IV}.
$\sum_{k=0}^\infty\f{a+dk}{m^k}\bi{2k}{k}^2T_{2k}(b,c)$.

\ \ {\tt Type V}.
$\sum_{k=0}^\infty\f{a+dk}{m^k}\bi{2k}{k}\bi{3k}kT_{3k}(b,c)$.

\ \ {\tt Type VI}.
$\sum_{k=0}^\infty\f{a+dk}{m^k}T_{k}(b,c)^3,$

\ \ {\tt Type VII}.
$\sum_{k=0}^\infty\f{a+dk}{m^k}\bi{2k}kT_{k}(b,c)^2,$
\medskip

In general, the corresponding $p$-adic congruences of these seven-type series involve linear combinations of two Legendre symbols.
The author's conjectural series of types I-V and VII
were studied in \cite{ChWZ,WZ,Z14}. The author's three conjectural series of type VI and two series of type VII
remain open. For example, the author conjectured that
$$\sum_{k=0}^\infty\f{3990k+1147}{(-288)^{3k}}T_k(62,95^2)^3=\f{432}{95\pi}(94\sqrt2+195\sqrt{14})$$
as well as its $p$-adic analogue
$$\sum_{k=0}^{p-1}\f{3990k+1147}{(-288)^{3k}}T_k(62,95^2)^3\eq\f p{19}\l(4230\l(\f{-2}p\r)+17563\l(\f{-14}p\r)\r)\pmod{p^2},$$
where $p$ is any prime greater than $3$.

In 1905, J. W. L. Glaisher \cite{Gl} proved that
$$\sum_{k=0}^\infty\f{(4k-1)\bi{2k}k^4}{(2k-1)^4256^k}=-\f 8{\pi^2}.$$
This actually follows from the following finite identity observed by the author \cite{S14a}:
$$\sum_{k=0}^n\f{(4k-1)\bi{2k}k^4}{(2k-1)^4256^k}=-(8n^2+4n+1)\f{\bi{2n}n^4}{256^n}\quad\ \t{for all}\ n\in\N.$$
Motivated by Glaisher's identity and Ramanujan-type series for $1/\pi$, we obtain the following theorem.
\begin{theorem}\label{Th1.1} We have the following identities:
\begin{align}\sum_{k=0}^\infty\f{k(4k-1)\bi{2k}k^3}{(2k-1)^2(-64)^k}&=-\f1{\pi},
\\\sum_{k=0}^\infty\f{(4k-1)\bi{2k}k^3}{(2k-1)^3(-64)^k}&=\f2{\pi},
\\\sum_{k=0}^\infty\f{(12k^2-1)\bi{2k}k^3}{(2k-1)^2 256^k}&=-\f2{\pi},
\\\sum_{k=0}^\infty\f{k(6k-1)\bi{2k}k^3}{(2k-1)^3256^k}&=\f1{2\pi},
\\\sum_{k=0}^\infty\f{(28k^2-4k-1)\bi{2k}k^3}{(2k-1)^2(-512)^k}&=-\f{3\sqrt2}{\pi},
\\\sum_{k=0}^\infty\f{(30k^2+3k-2)\bi{2k}k^3}{(2k-1)^3(-512)^k}&=\f{27\sqrt2}{8\pi},
\\\sum_{k=0}^\infty\f{(28k^2-4k-1)\bi{2k}k^3}{(2k-1)^2 4096^k}&=-\f3{\pi},
\\\sum_{k=0}^\infty\f{(42k^2-3k-1)\bi{2k}k^3}{(2k-1)^3 4096^k}&=\f{27}{8\pi},
\\\sum_{k=0}^\infty\f{(34k^2-3k-1)\bi{2k}k^2\bi{3k}k}{(2k-1)(3k-1)(-192)^k}&=-\f{10}{\sqrt3\,\pi},
\\\sum_{k=0}^\infty\f{(64k^2-11k-7)\bi{2k}k^2\bi{3k}k}{(k+1)(2k-1)(3k-1)(-192)^k}&=-\f{125\sqrt3}{9\pi},
\\\sum_{k=0}^\infty\f{(14k^2+k-1)\bi{2k}k^2\bi{3k}k}{(2k-1)(3k-1)216^k}&=-\f{\sqrt3}{\pi},
\\\sum_{k=0}^\infty\f{(90k^2+7k+1)\bi{2k}k^2\bi{3k}k}{(k+1)(2k-1)(3k-1)216^k}&=\f{9\sqrt3}{2\pi},
\\\sum_{k=0}^\infty\f{(34k^2-3k-1)\bi{2k}k^2\bi{3k}k}{(2k-1)(3k-1)(-12)^{3k}}&=-\f{2\sqrt3}{\pi},\\
\sum_{k=0}^\infty\f{(17k+5)\bi{2k}k^2\bi{3k}k}{(k+1)(2k-1)(3k-1)(-12)^{3k}}&=\f{9\sqrt3}{\pi},
\\\sum_{k=0}^\infty\f{(111k^2-7k-4)\bi{2k}k^2\bi{3k}k}{(2k-1)(3k-1)1458^k}&=-\f{45}{4\pi},
\end{align}
\begin{align}\sum_{k=0}^\infty\f{(1524k^2+899k+263)\bi{2k}k^2\bi{3k}k}{(k+1)(2k-1)(3k-1)1458^k}&=\f{3375}{4\pi},
\\
\sum_{k=0}^\infty\f{(522k^2-55k-13)\bi{2k}k^2\bi{3k}k}{(2k-1)(3k-1)(-8640)^k}&=-\f{54\sqrt{15}}{5\pi},
\\\sum_{k=0}^\infty\f{(1836k^2+2725k+541)\bi{2k}k^2\bi{3k}k}{(k+1)(2k-1)(3k-1)(-8640)^k}&=\f{2187\sqrt{15}}{5\pi},
\\\sum_{k=0}^\infty\f{(529k^2-45k-16)\bi{2k}k^2\bi{3k}k}{(2k-1)(3k-1)15^{3k}}&=-\f{55\sqrt3}{2\pi},
\\\sum_{k=0}^\infty\f{(77571k^2+68545k+16366)\bi{2k}k^2\bi{3k}k}{(k+1)(2k-1)(3k-1)15^{3k}}&=\f{59895\sqrt3}{2\pi},\\
\sum_{k=0}^\infty\f{(574k^2-73k-11)\bi{2k}k^2\bi{3k}k}{(2k-1)(3k-1)(-48)^{3k}}&=-20\f{\sqrt3}{\pi},
\\\sum_{k=0}^\infty\f{(8118k^2+9443k+1241)\bi{2k}k^2\bi{3k}k}{(k+1)(2k-1)(3k-1)(-48)^{3k}}&=\f{2250\sqrt3}{\pi},
\\\sum_{k=0}^\infty\f{(978k^2-131k-17)\bi{2k}k^2\bi{3k}k}{(2k-1)(3k-1)(-326592)^k}&=-\f{990\sqrt7}{49\pi},
\\\sum_{k=0}^\infty\f{(592212k^2+671387k^2+77219)\bi{2k}k^2\bi{3k}k}{(k+1)(2k-1)(3k-1)(-326592)^k}&=\f{4492125\sqrt7}{49\pi},
\\\sum_{k=0}^\infty\f{(116234k^2-17695k-1461)\bi{2k}k^2\bi{3k}k}{(2k-1)(3k-1)(-300)^{3k}}&=-2650\f{\sqrt3}{\pi},
\\
\sum_{k=0}^\infty\f{(223664832k^2+242140765k+18468097)\bi{2k}k^2\bi{3k}k}{(k+1)(2k-1)(3k-1)(-300)^{3k}}&=33497325\f{\sqrt3}{\pi},
\\\sum_{k=0}^\infty\f{(122k^2+3k-5)\bi{2k}k^2\bi{4k}{2k}}{(2k-1)(4k-1)648^k}&=-\f{21}{2\pi},
\\\sum_{k=0}^\infty\f{(1903k^2+114k+41)\bi{2k}k^2\bi{4k}{2k}}{(k+1)(2k-1)(4k-1)648^k}&=\f{343}{2\pi},
\\\sum_{k=0}^\infty\f{(40k^2-2k-1)\bi{2k}k^2\bi{4k}{2k}}{(2k-1)(4k-1)(-1024)^k}&=-\f4{\pi},
\\\sum_{k=0}^\infty\f{(8k^2-2k-1)\bi{2k}k^2\bi{4k}{2k}}{(k+1)(2k-1)(4k-1)(-1024)^k}&=-\f{16}{5\pi},\\
\sum_{k=0}^\infty\f{(176k^2-6k-5)\bi{2k}k^2\bi{4k}{2k}}{(2k-1)(4k-1)48^{2k}}&=-8\f{\sqrt3}{\pi},
\end{align}
\begin{align}\sum_{k=0}^\infty\f{(208k^2+66k+23)\bi{2k}k^2\bi{4k}{2k}}{(k+1)(2k-1)(4k-1)48^{2k}}&=\f{128}{\sqrt3\,\pi},
\\\sum_{k=0}^\infty\f{(6722k^2-411k-152)\bi{2k}k^2\bi{4k}{2k}}{(2k-1)(4k-1)(-63^2)^k}&=-195\f{\sqrt7}{\pi},
\\
\sum_{k=0}^\infty\f{(281591k^2-757041k-231992)\bi{2k}k^2\bi{4k}{2k}}{(k+1)(2k-1)(4k-1)(-63^2)^k}&=-274625\f{\sqrt7}{\pi},
\\\sum_{k=0}^\infty\f{(560k^2-42k-11)\bi{2k}k^2\bi{4k}{2k}}{(2k-1)(4k-1)12^{4k}}&=-24\f{\sqrt2}{\pi},
\\\sum_{k=0}^\infty\f{(112k^2+114k+23)\bi{2k}k^2\bi{4k}{2k}}{(k+1)(2k-1)(4k-1)12^{4k}}&=\f{256\sqrt2}{5\pi},
\\\sum_{k=0}^\infty\f{(248k^2-18k-5)\bi{2k}k^2\bi{4k}{2k}}{(2k-1)(4k-1)(-3\times2^{12})^k}&=-\f{28}{\sqrt3\,\pi},
\\\sum_{k=0}^\infty\f{(680k^2+1482k+337)\bi{2k}k^2\bi{4k}{2k}}{(k+1)(2k-1)(4k-1)(-3\times2^{12})^k}&=\f{5488\sqrt3}{9\pi},
\\\sum_{k=0}^\infty\f{(1144k^2-102k-19)\bi{2k}k^2\bi{4k}{2k}}{(2k-1)(4k-1)(-2^{10}3^4)^k}&=-\f{60}{\pi},
\\\sum_{k=0}^\infty\f{(3224k^2+4026k+637)\bi{2k}k^2\bi{4k}{2k}}{(k+1)(2k-1)(4k-1)(-2^{10}3^4)^k}&=\f{2000}{\pi},
\\\sum_{k=0}^\infty\f{(7408k^2-754k-103)\bi{2k}k^2\bi{4k}{2k}}{(2k-1)(4k-1)28^{4k}}&=-\f{560\sqrt3}{3\pi},
\\\sum_{k=0}^\infty\f{(3641424k^2+4114526k+493937)\bi{2k}k^2\bi{4k}{2k}}{(k+1)(2k-1)(4k-1)28^{4k}}&=896000\f{\sqrt3}{\pi},
\\\sum_{k=0}^\infty\f{(4744k^2-534k-55)\bi{2k}k^2\bi{4k}{2k}}{(2k-1)(4k-1)(-2^{14}3^45)^k}&=-\f{1932\sqrt5}{25\pi},
\\\sum_{k=0}^\infty\f{(18446264k^2+20356230k+1901071)\bi{2k}k^2\bi{4k}{2k}}{(k+1)(2k-1)(4k-1)(-2^{14}3^45)^k}&=66772496\f{\sqrt5}{25\pi},
\\\sum_{k=0}^\infty\f{(413512k^2-50826k-3877)\bi{2k}k^2\bi{4k}{2k}}{(2k-1)(4k-1)(-2^{10}21^4)^k}&=-\f{12180}{\pi},
\\\sum_{k=0}^\infty\f{(1424799848k^2+1533506502k+108685699)\bi{2k}k^2\bi{4k}{2k}}{(k+1)(2k-1)(4k-1)(-2^{10}21^4)^k}&=\f{341446000}{\pi},
\\\sum_{k=0}^\infty\f{(71312k^2-7746k-887)\bi{2k}k^2\bi{4k}{2k}}{(2k-1)(4k-1)1584^{2k}}&=-840\f{\sqrt{11}}{\pi},
\end{align}
\begin{align}
\sum_{k=0}^\infty\f{(50678512k^2+56405238k+5793581)\bi{2k}k^2\bi{4k}{2k}}{(k+1)(2k-1)(4k-1)1584^{2k}}&=5488000\f{\sqrt{11}}{\pi},
\\\sum_{k=0}^\infty\f{(7329808k^2-969294k-54073)\bi{2k}k^2\bi{4k}{2k}}{(2k-1)(4k-1)396^{4k}}&=-120120\f{\sqrt2}{\pi},
\end{align}
\begin{equation}
\begin{aligned}&\sum_{k=0}^\infty\f{(2140459883152k^2+2259867244398k+119407598201)\bi{2k}k^2\bi{4k}{2k}}{(k+1)(2k-1)(4k-1)396^{4k}}
\\&\qquad\qquad\qquad=44\times1820^3\f{\sqrt2}{\pi},
\end{aligned}\end{equation}
\begin{align}
\sum_{k=0}^\infty\f{(164k^2-k-3)\bi{2k}k\bi{3k}k\bi{6k}{3k}}{(2k-1)(6k-1)20^{3k}}&=-\f{7\sqrt5}{2\pi},
\\\sum_{k=0}^\infty\f{(2696k^2+206k+93)\bi{2k}k\bi{3k}k\bi{6k}{3k}}{(k+1)(2k-1)(6k-1)20^{3k}}
&=\f{686}{\sqrt5\,\pi},
\\\sum_{k=0}^\infty\f{(220k^2-8k-3)\bi{2k}k\bi{3k}k\bi{6k}{3k}}{(2k-1)(6k-1)(-2^{15})^k}
&=-\f{7\sqrt2}{\pi},
\\\sum_{k=0}^\infty\f{(836k^2-1048k-309)\bi{2k}k\bi{3k}k\bi{6k}{3k}}{(k+1)(2k-1)(6k-1)(-2^{15})^k}
&=-\f{686\sqrt2}{\pi},
\\\sum_{k=0}^\infty\f{(504k^2-11k-8)\bi{2k}k\bi{3k}k\bi{6k}{3k}}{(2k-1)(6k-1)(-15)^{3k}}
&=-\f{9\sqrt{15}}{\pi},
\\\sum_{k=0}^\infty\f{(189k^2-11k-8)\bi{2k}k\bi{3k}k\bi{6k}{3k}}{(k+1)(2k-1)(6k-1)(-15)^{3k}}&=-\f{243\sqrt{15}}{35\pi},
\\\sum_{k=0}^\infty\f{(516k^2-19k-7)\bi{2k}k\bi{3k}k\bi{6k}{3k}}{(2k-1)(6k-1)(2\times30^3)^k}
&=-\f{11\sqrt{15}}{2\pi},
\\\sum_{k=0}^\infty\f{(3237k^2+1922k+491)\bi{2k}k\bi{3k}k\bi{6k}{3k}}{(k+1)(2k-1)(6k-1)(2\times30^3)^k}
&=\f{3993\sqrt{15}}{10\pi},
\\\sum_{k=0}^\infty\f{(684k^2-40k-7)\bi{2k}k\bi{3k}k\bi{6k}{3k}}{(2k-1)(6k-1)(-96)^{3k}}&=-\f{9\sqrt6}{\pi},
\\\sum_{k=0}^\infty\f{(2052k^2+2536k+379)\bi{2k}k\bi{3k}k\bi{6k}{3k}}{(k+1)(2k-1)(6k-1)(-96)^{3k}}&=\f{486\sqrt6}{\pi},
\\\sum_{k=0}^\infty\f{(2556k^2-131k-29)\bi{2k}k\bi{3k}k\bi{6k}{3k}}{(2k-1)(6k-1)66^{3k}}&=-\f{63\sqrt{33}}{4\pi},
\\\sum_{k=0}^\infty\f{(203985k^2+212248k+38083)\bi{2k}k\bi{3k}k\bi{6k}{3k}}{(k+1)(2k-1)(6k-1)66^{3k}}&=\f{83349\sqrt{33}}{4\pi},
\\\sum_{k=0}^\infty\f{(5812k^2-408k-49)\bi{2k}k\bi{3k}k\bi{6k}{3k}}{(2k-1)(6k-1)(-3\times160^3)^k}&=-\f{253\sqrt{30}}{9\pi},
\end{align}
\begin{align}
\sum_{k=0}^\infty\f{(3471628k^2+3900088k+418289)\bi{2k}k\bi{3k}k\bi{6k}{3k}}{(k+1)(2k-1)(6k-1)(-3\times160^3)^k}&=\f{32388554\sqrt{30}}{135\pi},
\\\sum_{k=0}^\infty\f{(35604k^2-2936k-233)\bi{2k}k\bi{3k}k\bi{6k}{3k}}{(2k-1)(6k-1)(-960)^{3k}}&=-189\f{\sqrt{15}}{\pi},
\\\sum_{k=0}^\infty\f{(13983084k^2+15093304k+1109737)\bi{2k}k\bi{3k}k\bi{6k}{3k}}{(k+1)(2k-1)(6k-1)(-960)^{3k}}&=\f{4500846\sqrt{15}}{5\pi},
\\\sum_{k=0}^\infty\f{(157752k^2-11243k-1304)\bi{2k}k\bi{3k}k\bi{6k}{3k}}{(2k-1)(6k-1)255^{3k}}&=-\f{513\sqrt{255}}{2\pi},
\\\sum_{k=0}^\infty\f{(28240947k^2+31448587k+3267736)\bi{2k}k\bi{3k}k\bi{6k}{3k}}{(k+1)(2k-1)(6k-1)255^{3k}}&=\f{45001899\sqrt{255}}{70\pi},
\\\sum_{k=0}^\infty\f{(2187684k^2-200056k-11293)\bi{2k}k\bi{3k}k\bi{6k}{3k}}{(2k-1)(6k-1)(-5280)^{3k}}
&=-1953\f{\sqrt{330}}{\pi},\end{align}
\begin{equation}
\begin{aligned}
&\sum_{k=0}^\infty\f{(101740699836k^2+107483900696k+5743181813)\bi{2k}k\bi{3k}k\bi{6k}{3k}}{(k+1)(2k-1)(6k-1)(-5280)^{3k}}
\\&\qquad\qquad\qquad=\f{4966100118\sqrt{330}}{5\pi},
\end{aligned}
\end{equation}
\begin{equation}
\begin{aligned}
&\sum_{k=0}^\infty\f{(16444841148k^2-1709536232k-53241371)\bi{2k}k\bi{3k}k\bi{6k}{3k}}{(2k-1)(6k-1)(-640320)^{3k}}
\\&\qquad\qquad\qquad=-1672209\f{\sqrt{10005}}{\pi},
\end{aligned}\end{equation}
and
\begin{equation}\begin{aligned}
\sum_{k=0}^\infty\f{P(k)\bi{2k}k\bi{3k}k\bi{6k}{3k}}{(k+1)(2k-1)(6k-1)(-640320)^{3k}}
=\f{18\times557403^3\sqrt{10005}}{5\pi},
\end{aligned}\end{equation}
where
\begin{align*}P(k):=&637379600041024803108 k^2  + 657229991696087780968 k
\\&+ 19850391655004126179.
\end{align*}
\end{theorem}

 Recall that the Catalan numbers are given by
$$C_n:=\f{\bi{2n}n}{n+1}=\bi{2n}n-\bi{2n}{n+1}\ \ (n\in\N).$$
For $k\in\N$ it is easy to see that
$$\f{\bi{2k}k}{2k-1}=\begin{cases}-1&\t{if}\ k=0,\\2C_{k-1}&\t{if}\ k>0.\end{cases}$$
Thus, for any $a,b,c,m\in\Z$ with $|m|\gs64$, we have
\begin{align*}\sum_{k=0}^\infty\f{(ak^2+bk+c)\bi{2k}k^3}{(2k-1)^3m^k}
=&-c+\sum_{k=1}^\infty\f{(ak^2+bk+c)(2C_{k-1})^3}{m^k}
\\=&-c+\f 8m\sum_{k=0}^\infty\f{a(k+1)^2+b(k+1)+c}{m^k}C_k^3.
\end{align*}
For example, (1.2) has the equivalent form
\[\sum_{k=0}^\infty\f{4k+3}{(-64)^k}C_k^3=8-\f{16}{\pi}.\tag{1.2$'$}\]
For any odd prime $p$, the congruence (1.4) of V.J.W. Guo and J.-C. Liu \cite{GL} has the equivalent form
$$\sum_{k=0}^{(p+1)/2}\f{(4k-1)\bi{2k}k^3}{(2k-1)^3(-64)^k}\eq p\l(\f{-1}p\r)+p^3(E_{p-3}-2)\pmod{p^4}$$
(where $E_0,E_1,\ldots$ are the Euler numbers), and we note that this is also equivalent to the congruence
$$\sum_{k=0}^{(p-1)/2}\f{4k+3}{(-64)^k}C_k^3\eq8\l(1-p\l(\f{-1}p\r)-p^3(E_{p-3}-2)\r)\pmod{p^4}.$$
Recently, C. Wang \cite{Wang} proved that for any prime $p>3$ we have
$$\sum_{k=0}^{(p+1)/2}\f{(3k-1)\bi{2k}k^3}{(2k-1)^216^k}\eq p+2p^3\l(\f{-1}p\r)(E_{p-3}-3)\pmod{p^4}$$
and
$$\sum_{k=0}^{p-1}\f{(3k-1)\bi{2k}k^3}{(2k-1)^216^k}\eq p-2p^3\pmod{p^4}.$$
(Actually, Wang stated his results only in the language of hypergeometric series.) These two congruences extend a conjecture of Guo and M. J. Schlosser \cite{GS}.

We are also able to prove some other variants of Ramanujan-type series such as
$$\sum_{k=0}^\infty\f{(56k^2+118k+61)\bi{2k}k^3}{(k+1)^24096^k}=\f{192}{\pi}$$
and $$\sum_{k=0}^\infty\f{(420k^2+992k+551)\bi{2k}k^3}{(k+1)^2(2k-1)4096^k}=-\f{1728}{\pi}.$$

Now we state our second theorem.

\begin{theorem}\label{Th1.2} We have the identities
\begin{align}
\sum_{k=1}^\infty\f{28k^2+31k+8}{(2k+1)^2k^3\bi{2k}k^3}&=\f{\pi^2-8}2,
\\\sum_{k=1}^\infty\f{42k^2+39k+8}{(2k+1)^3k^3\bi{2k}k^3}&=\f{9\pi^2-88}2,
\\\sum_{k=1}^\infty\f{(8k^2+5k+1)(-8)^k}{(2k+1)^2k^3\bi{2k}k^3}&=4-6G,
\\\sum_{k=1}^\infty\f{(30k^2+33k+7)(-8)^k}{(2k+1)^3k^3\bi{2k}k^3}&=54G-52,
\\\sum_{k=1}^\infty\f{(3k+1)16^k}{(2k+1)^2k^3\bi{2k}k^3}&=\f{\pi^2-8}2,
\end{align}
\begin{align}
\sum_{k=1}^\infty\f{(4k+1)(-64)^k}{(2k+1)^2k^2\bi{2k}k^3}&=4-8G,
\\\sum_{k=1}^\infty\f{(4k+1)(-64)^k}{(2k+1)^3k^3\bi{2k}k^3}&=16G-16,
\\\sum_{k=1}^\infty\f{(2k^2-11k-3)8^k}{(2k+1)(3k+1)k^3\bi{2k}k^2\bi{3k}k}&=\f{48-5\pi^2}2,
\\\sum_{k=2}^\infty\f{(178k^2-103k-39)8^k}{(k-1)(2k+1)(3k+1)k^3\bi{2k}k^2\bi{3k}k}&=\f{1125\pi^2-11096}{36},
\\\sum_{k=1}^\infty\f{(5k+1)(-27)^k}{(2k+1)(3k+1)k^2\bi{2k}k^2\bi{3k}k}&=6-9K,
\\\sum_{k=2}^\infty\f{(45k^2+5k-2)(-27)^{k-1}}{(k-1)(2k+1)(3k+1)k^3\bi{2k}k^2\bi{3k}k}&=\f{37-48K}{16},
\\\sum_{k=1}^\infty\f{(98k^2-21k-8)81^k}{(2k+1)(4k+1)k^3\bi{2k}k^2\bi{4k}{2k}}&=216-20\pi^2,
\\\sum_{k=2}^\infty\f{(1967k^2-183k-104)81^k}{(k-1)(2k+1)(4k+1)k^3\bi{2k}k^2\bi{4k}{2k}}&=\f{20000\pi^2-190269}{120},
\end{align}
\begin{align}
\sum_{k=1}^\infty\f{(46k^2+3k-1)(-144)^k}{(2k+1)(4k+1)k^3\bi{2k}k^2\bi{4k}{2k}}&=72-\f{225}2K,
\\\sum_{k=2}^\infty\f{(343k^2+18k-16)(-144)^k}{(k-1)(2k+1)(4k+1)k^3\bi{2k}k^2\bi{4k}{2k}}&=\f{9375K-7048}{10},
\end{align}
where
$$G:=\sum_{k=0}^\infty\f{(-1)^k}{(2k+1)^2}\ \ \t{and}\ \ K:=\sum_{k=0}^\infty\f{(\f k3)}{k^2}.$$
\end{theorem}

For $k=j+1\in\Z^+$, it is easy to see that
$$(k-1)k\bi{2k}k=2(2j+1)j\bi{2j}j.$$
Thus, for any $a,b,c,m\in\Z$ with $0<|m|\ls 64$, we have
$$\sum_{j=1}^\infty\f{(aj^2+bj+c)m^j}{(2j+1)^3j^3\bi{2j}j^3}
=\f 8m\sum_{k=2}^\infty\f{(a(k-1)^2+b(k-1)+c)m^k}{(k-1)^3k^3\bi{2k}k^3}.$$
For example, (1.77) has the following equivalent form
\[\sum_{k=2}^\infty\f{(2k-1)(3k-2)16^k}{(k-1)^3k^3\bi{2k}k^3}=\pi^2-8.\tag{1.77$'$}\]

In contrast with the Domb numbers, we introduce a new kind of numbers
$$S_n:=\sum_{k=0}^n\bi nk^2T_kT_{n-k}\ \ (n=0,1,2,\ldots).$$
The values of $S_n\ (n=0,\ldots,10)$ are
$$1,\, 2,\, 10,\, 68,\, 586,\, 5252,\, 49204,\, 475400,\, 4723786,\, 47937812,\, 494786260$$
respectively. We may extend the numbers $S_n\ (n\in\N)$ further. For $b,c\in\Z$, we define
$$S_n(b,c):=\sum_{k=0}^n\bi nk^2T_k(b,c)T_{n-k}(b,c)\ \ (n=0,1,2,\ldots).$$
Note that $S_n(1,1)=S_n$ and $S_n(2,1)=D_n$ for all $n\in\N$.

Now we state our third theorem.

\begin{theorem}\label{Th1.3} We have
\begin{align}\label{S1}
\sum_{k=0}^\infty\f{7k+3}{24^k}S_k(1,-6)&=\f{15}{\sqrt2\,\pi},
\\ \label{S2}\sum_{k=0}^\infty\f{12k+5}{(-28)^k}S_k(1,7)&=\f{6\sqrt7}{\pi},
\\\label{S3}\sum_{k=0}^\infty\f{84k+29}{80^k}S_k(1,-20)&=\f{24\sqrt{15}}{\pi},
\\\label{S4}\sum_{k=0}^\infty\f{3k+1}{(-100)^k}S_k(1,25)&=\f{25}{8\pi},
\\\label{S5}\sum_{k=0}^\infty\f{228k+67}{224^k}S_k(1,-56)&=\f{80\sqrt7}{\pi},
\\\label{S6}\sum_{k=0}^\infty\f{399k+101}{(-676)^k}S_k(1,169)&=\f{2535}{8\pi},\\
\label{S7}\sum_{k=0}^\infty\f{2604k+563}{2600^k}S_k(1,-650)&=\f{850\sqrt{39}}{3\pi},
\\\label{S8}\sum_{k=0}^\infty\f{39468k+7817}{(-6076)^k}S_k(1,1519)&=\f{4410\sqrt{31}}{\pi},
\\\label{S9}\sum_{k=0}^\infty\f{41667k+7879}{9800^k}S_k(1,-2450)&=\f{40425\sqrt6}{4\pi},
\\\label{S10}\sum_{k=0}^\infty\f{74613k+10711}{(-530^2)^k}S_k(1,265^2)&=\f{1615175}{48\pi}.
\end{align}
\end{theorem}
\begin{remark}\label{Rem1.3}\rm The author found the 10 series in Theorem \ref{Th1.3}
in Nov. 2019.
\end{remark}

We shall prove Theorems 1.1-1.3 in the next section.
In Sections 3-10, we propose 117 new conjectural series for powers of $\pi$ involving generalized central trinomial coefficients. In particular, we will present in Section 3 four conjectural series for $1/\pi$ of the following new type:
\smallskip

\ \ {\tt Type VIII}.
$\sum_{k=0}^\infty\f{a+dk}{m^k}T_k(b,c)T_{k}(b_*,c_*)^2,$
\smallskip
\newline
where $a,b,b_*,c,c_*,d,m$ are integers with $mbb_*cc_*d(b^2-4c)(b_*^2-4c_*)(b^2c_*-b_*^2c)\not=0$.

Unlike Ramanujan-type series given by others, all our series for $1/\pi$ of types I-VIII
have the general term involving a {\it product} of three generalized central trinomial coefficients.

Motivated by the author's effective way to find new series for $1/\pi$ (cf. Sun \cite{S13d}),
we formulate the following general characterization of rational Ramanujan-type series for $1/\pi$
 via congruences.

 \begin{conjecture}[General Criterion for Rational Ramanujan-type Series for $1/\pi$]\label{general}
  Suppose that the series $\sum_{k=0}^\infty\f{bk+c}{m^k}a_k$ converges,
 where $(a_k)_{k\gs0}$ is an integer sequence and $b,c,m$ are integers with $bm\not=0$.
 Let $r\in\{1,2,3\}$ and let $d_1,\ldots,d_r$ be positive integers with $\sqrt{d_i/d_j}$ irrational for all distinct $i,j\in\{1,\ldots,r\}$. Then
 \begin{equation}\label{series}\sum_{k=0}^\infty\f{bk+c}{m^k}a_k
 =\f{\sum_{i=1}^r\lambda_i\sqrt{d_i}}{\pi}
 \end{equation}
 for some nonzero rational numbers $\lambda_1,\ldots,\lambda_r$ if and only if there are
 positive integers $d_j\ (r<j\ls 3)$ and rational numbers $c_1,c_2,c_3$ with $\prod_{i=1}^rc_i\not=0$, such that for any prime $p>3$
 with $p\nmid m\prod_{i=1}^rd_i$ and $c_1,c_2,c_3\in\Z_p$ we have
 \begin{equation}\label{cong}\sum_{k=0}^{p-1}\f{bk+c}{m^k}a_k\eq p\(\sum_{i=1}^rc_i\l(\f{\ve_id_i}p\r)+\sum_{r<j\ls 3}c_j\l(\f{d_j}p\r)\)
 \pmod{p^2},\end{equation}
 where $\ve_i\in\{\pm1\}$, $\ve_i=-1$ if $d_i$ is not an integer square,
 and $c_2=c_3=0$ if $r=1$ and $\ve_1=1$.
 \end{conjecture}

For a Ramanujan-type series of the form \eqref{series}, we call $r$ its {\it rank}.
We believe that there are some Ramanujan-type series of rank three  but we have not yet found such a series.

\begin{conjecture}\label{Conj-pn} Let $(a_k)_{k\gs0}$ be any integer sequence,
 and let $b,c,m,d_1,d_2,d_3\in\Z$ with $bm\not=0$.
 Assume that $\lim_{n\to+\infty}\root n\of{|a_n|}=r<|m|$,
and $\pi\sum_{k=0}^\infty\f{bk+c}{m^k}a_k$
is an algebraic number. Suppose that $c_1,c_2,c_3$ are rational numbers
with $c_1+c_2+c_3=a_0c$,  and
\begin{equation}\label{linear-p}\sum_{k=0}^{p-1}\f{bk+c}{m^k}a_k\eq p\l(c_1\l(\f{d_1}p\r)+c_2\l(\f{d_2}p\r)+c_3\l(\f{d_3}p\r)\r)\pmod{p^2}
\end{equation}
for all primes $p>3$ with $p\nmid d_1d_2d_3m$ and $c_1,c_2,c_3\in\Z_p$.
Then, for any prime $p>3$
with $p\nmid m$, $c_1,c_2,c_3\in\Z_p$ and $(\f{d_1}p)=(\f{d_2}p)=(\f{d_3}p)=\da\in\{\pm1\}$,
 we have
\begin{equation*}\label{pn}
\f1{(pn)^2}\(\sum_{k=0}^{pn-1}\f{bk+c}{m^k}a_k-p\da\sum_{k=0}^{n-1}\f{bk+c}{m^k}a_k\)\in\Z_p
\ \ \t{for all}\ n\in\Z^+.
\end{equation*}
\end{conjecture}

Joint with the author's PhD student Chen Wang, we pose the following conjecture.

\begin{conjecture}[Chen Wang and Z.-W. Sun]\label{Conj-WS} Let $(a_k)_{k\gs0}$ be an integer sequence with $a_0=1$. Let $b,c,m,d_1,d_2,d_3\in\Z$ with $bm\not=0$, and let
$c_1,c_2,c_3$ be rational numbers. If $\pi\sum_{k=0}^\infty\f{bk+c}{m^k}a_k$
is an algebraic number, and the congruence \eqref{linear-p}
holds
for all primes $p>3$ with $p\nmid d_1d_2d_3m$ and $c_1,c_2,c_3\in\Z_p$, then we must have
$c_1+c_2+c_3=c.$
\end{conjecture}
\begin{remark} The author \cite[Conjecture 1.1(i)]{S14b} conjectured that
$$\sum_{k=0}^{p-1}(8k+5)T_k^2\eq 3p\l(\f{-3}p\r)\pmod{p^2}$$ for any prime $p>3$,
which was confirmed by Y.-P. Mu and Z.-W. Sun \cite{MS}. This is not a counterexample
to Conjecture \ref{Conj-WS} since $\sum_{k=0}^\infty(8k+5)T_k^2$ diverges.
\end{remark}

All the new series and related congruences in Sections 3-9 support Conjectures 1.1-1.3.
We discover the conjectural series for $1/\pi$ in Sections 3-9
based on the author's previous {\tt Philosophy about Series for $1/\pi$} stated in \cite{S13d},
the PSLQ algorithm to discover integer relations (cf. \cite{FBA}),
and the following {\tt Duality Principle} based on the author's experience and intuition.

\begin{conjecture} [Duality Principle]\label{Conj-Dual} Let $(a_k)_{k\gs0}$ be an integer sequence such that
\begin{equation}\label{ak}a_{k}\eq\l(\f dp\r)D^ka_{p-1-k}\pmod p
\end{equation}
for any prime $p\nmid 6dD$ and $k\in\{0,\ldots,p-1\}$, where $d$ and $D$ are fixed nonzero integers.
If $m$ is a nonzero integer such that $$\sum_{k=0}^\infty\f{bk+c}{m^k}a_k=\f{\lambda_1\sqrt{d_1}+\lambda_2\sqrt{d_2}+\lambda_3\sqrt{d_3}}{\pi}$$
for some $b,d_1,d_2,d_3\in\Z^+$, $c\in\Z$ and $\lambda_1,\lambda_2,\lambda_3\in\Q$, then $m$ divides $D$, and \begin{equation}\label{dual}\sum_{k=0}^{p-1}\f{a_k}{m^k}\eq\l(\f dp\r)\sum_{k=0}^{p-1}\f{a_k}{(D/m)^k}\pmod{p^2}
\end{equation}
for any prime $p>3$ with $p\nmid dD$.
\end{conjecture}
\begin{remark}\label{Rem-Dual} (i) For any prime $p>3$ with $p\nmid dDm$,
the congruence \eqref{dual} holds modulo $p$ by \eqref{ak} and Fermat's little theorem.
We call $\sum_{k=0}^{p-1}a_k/(\f Dm)^k$ the {\it dual} of the sum $\sum_{k=0}^{p-1}a_k/m^k$.

(ii) For any $b,c\in\Z$ and odd prime $p\nmid b^2-4c$, it is known (see, e.g., \cite[Lemma 2.2]{S14b})
that
\begin{equation} \label{T-dual}T_k(b,c)\eq\l(\f{b^2-4c}p\r)(b^2-4c)^kT_{p-1-k}(b,c)\pmod p
\end{equation}
for all $k=0,1,\ldots,p-1$.
\end{remark}

For a  series $\sum_{k=0}^\infty a_k$ with $a_0,a_1,\ldots$ real numbers, if $\lim_{k\to+\infty}a_{k+1}/a_k=r\in(-1,1)$ then we say that the series converge at a geometric rate with ratio $r$. Except for $(7.1)$, all other conjectural series in Sections 3-9 converge
at geometric rates and thus one can easily check them numerically via a computer.

In Section 10, we pose two curious conjectural series for $\pi$ involving the central trinomial coefficients.

\section{Proofs of Theorems 1.1-1.3}
\setcounter{equation}{0}

\begin{lemma}\label{Lem2.1} Let $m\not=0$ and $n\gs0$ be integers. Then
\begin{align}\sum_{k=0}^n\f{((64-m)k^3-32k^2-16k+8)\bi{2k}k^3}{(2k-1)^2m^k}=&\f{8(2n+1)}{m^n}\bi{2n}n^3,
\\\sum_{k=0}^n\f{((64-m)k^3-96k^2+48k-8)\bi{2k}k^3}{(2k-1)^3m^k}=&\f 8{m^n}\bi{2n}n^3,
\\\sum_{k=0}^n\f{((108-m)k^3-54k^2-12k+6)\bi{2k}k^2\bi{3k}k}{(2k-1)(3k-1)m^k}=&\f{6(3n+1)}{m^n}\bi{2n}n^2\bi{3n}n,
\end{align}
\begin{equation}\begin{aligned}
&\sum_{k=0}^n\f{((108-m)k^3-(54+m)k^2-12k+6)\bi{2k}k^2\bi{3k}k}{(k+1)(2k-1)(3k-1)m^k}
\\&\qquad\qquad\qquad=\f{6(3n+1)}{(n+1)m^n}\bi{2n}n^2\bi{3n}n,
\end{aligned}\end{equation}
\begin{equation}\begin{aligned}
&\sum_{k=0}^n\f{((256-m)k^3-128k^2-16k+8)\bi{2k}k^2\bi{4k}{2k}}{(2k-1)(4k-1)m^k}
\\&\qquad\qquad\qquad=\f{8(4n+1)}{m^n}\bi{2n}n^2\bi{4n}{2n},
\end{aligned}\end{equation}
\begin{equation}\begin{aligned}
&\sum_{k=0}^n\f{((256-m)k^3-(128+m)k^2-16k+8)\bi{2k}k^2\bi{4k}{2k}}{(k+1)(2k-1)(4k-1)m^k}
\\&\qquad\qquad\qquad=\f{8(4n+1)}{(n+1)m^n}\bi{2n}n^2\bi{4n}{2n},
\end{aligned}\end{equation}
\begin{equation}\begin{aligned}
&\sum_{k=0}^n\f{((1728-m)k^3-864k^2-48k+24)\bi{2k}k\bi{3k}k\bi{6k}{3k}}{(2k-1)(6k-1)m^k}
\\&\qquad\qquad\qquad=\f{24(6n+1)}{m^n}\bi{2n}n\bi{3n}n\bi{6n}{3n},
\end{aligned}\end{equation}
\begin{equation}\begin{aligned}
&\sum_{k=0}^n\f{((1728-m)k^3-(864+m)k^2-48k+24)\bi{2k}k\bi{3k}k\bi{6k}{3k}}{(k+1)(2k-1)(6k-1)m^k}
\\&\qquad\qquad\qquad=\f{24(6n+1)}{(n+1)m^n}\bi{2n}n\bi{3n}n\bi{6n}{3n}.
\end{aligned}\end{equation}
\end{lemma}
\begin{remark} The eight identities in Lemma \ref{Lem2.1} can be easily proved by induction on $n$.
In light of Stirling's formula, $n!\sim\sqrt{2\pi n}(n/e)^n$ as $n\to+\infty$,
we have
 \begin{gather}\label{asy}\bi{2n}n\sim\f{4^n}{\sqrt{n\pi}},\ \ \bi{2n}n\bi{3n}n\sim\f{\sqrt3\, 27^n}{2n\pi},
 \\ \bi{2n}n\bi{4n}{2n}\sim\f{64^n}{\sqrt2\, n\pi},\ \ \bi{3n}n\bi{6n}n\sim\f{432^n}{2n\pi}.
 \end{gather}
\end{remark}

\medskip
\noindent{\it Proof of Theorem 1.1}. Just apply Lemma \ref{Lem2.1} and the 36 known rational
Ramanujan-type series listed in \cite{G}. Let us illustrate the proofs
 by showing (1.1), (1.2), (1.71) and (1.72) in details.

 By (2.1) with $m=-64$, we have
$$\sum_{k=0}^\infty\f{(16k^3-4k^2-2k+1)\bi{2k}k^3}{(2k-1)^2(-64)^k}=\lim_{n\to+\infty}\f{2n+1}{(-64)^n}\bi{2n}n^3=0.$$
Note that
$$16k^3-4k^2-2k+1=(4k+1)(2k-1)^2+2k(4k-1)$$
and recall Bauer's series
$$\sum_{k=0}^\infty(4k+1)\f{\bi{2k}k^3}{(-64)^k}=\f2{\pi}.$$
So, we get
$$\sum_{k=0}^\infty\f{k(4k-1)\bi{2k}k^3}{(2k-1)^2(-64)^k}
=-\f12\sum_{k=0}^\infty(4k+1)\f{\bi{2k}k^3}{(-64)^k}=-\f1{\pi}.$$
This proves (1.1). By (2.2) with $m=-64$, we have
$$\sum_{k=0}^n\f{(4k-1)(4k^2-2k+1)\bi{2k}k^3}{(2k-1)^3(-64)^k}=\f{\bi{2n}n^3}{(-64)^n}$$
and hence
$$\sum_{k=0}^\infty\f{(2k(2k-1)(4k-1)+4k-1)\bi{2k}k^3}{(2k-1)^3(-64)^k}=\lim_{n\to+\infty}\f{\bi{2n}n^3}{(-64)^n}=0.$$
Combining this with $(1.1)$ we immediately get $(1.2)$.

In view of (2.7) with $m=-640320^3$, we have
\begin{align*}&\sum_{k=0}^n\f{(10939058860032072k^3-36k^2-2k+1)\bi{2k}k\bi{3k}k\bi{6k}{3k}}
{(2k-1)(6k-1)(-640320)^{3k}}
\\=&\f{6n+1}{(-640320)^{3n}}\bi{2n}n\bi{3n}n\bi{6n}{3n}.
\end{align*}
and hence
$$\sum_{k=0}^\infty\f{(10939058860032072k^3-36k^2-2k+1)\bi{2k}k\bi{3k}k\bi{6k}{3k}}
{(2k-1)(6k-1)(-640320)^{3k}}=0.$$
In 1987, D. V. Chudnovsky and G. V. Chudnovsky \cite{CC}
got the formula
$$\sum_{k=0}^\infty\f{545140134k+13591409}{(-640320)^{3k}}\bi{2k}k\bi{3k}k\bi{6k}{3k}=\f{3\times53360^2}{2\pi\sqrt{10005}},$$
which enabled them to hold the world record for the calculation of $\pi$ during 1989--1994.
Note that
\begin{align*}&10939058860032072k^3-36k^2-2k+1
\\=&1672209(2k-1)(6k-1)(545140134k+13591409)
\\&+426880 (16444841148 k^2 - 1709536232 k-53241371 )
\end{align*}
and hence
\begin{align*}&\sum_{k=0}^\infty\f{(16444841148 k^2 - 1709536232 k-53241371 )\bi{2k}k\bi{3k}k\bi{6k}{3k}}{(2k-1)(6k-1)(-640320)^{3k}}
\\=&-\f{1672209}{426880}\times\f{3\times53360^2}{2\pi\sqrt{10005}}=-1672209\f{\sqrt{10005}}{\pi}.
\end{align*}
This proves $(1.71)$.

By (2.8) with $m=-640320^3$, we have
\begin{align*}&\sum_{k=0}^n\f{(10939058860032072k^3 +10939058860031964k^2-2k+1)\bi{2k}k\bi{3k}k\bi{6k}{3k}}
{(k+1)(2k-1)(6k-1)(-640320)^{3k}}
\\&\qquad\qquad=\f{6n+1}{(n+1)(-640320)^{3n}}\bi{2n}n\bi{3n}n\bi{6n}{3n}
\end{align*}
and hence
$$\sum_{k=0}^\infty\f{(10939058860032072k^3 +10939058860031964k^2-2k+1)\bi{2k}k\bi{3k}k\bi{6k}{3k}}
{(k+1)(2k-1)(6k-1)(-640320)^{3k}}=0.$$
Note that
\begin{align*}&2802461(10939058860032072k^3 +10939058860031964k^2-2k+1)
\\=&1864188626454(k+1)(16444841148 k^2 - 1709536232 k-53241371)+5P(k).
\end{align*}
Therefore, with the help of $(1.71)$ we get
\begin{align*}&\sum_{k=0}^\infty\f{P(k)\bi{2k}k\bi{3k}k\bi{6k}{3k}}{(k+1)(2k-1)(6k-1)(-640320)^{3k}}
\\=&-\f{1864188626454}5\times(-1672209)\f{\sqrt{10005}}{\pi}=18\times557403^3\f{\sqrt{10005}}{5\pi}.
\end{align*}
This proves $(1.72)$.

The identities (1.3)--(1.70) can be proved similarly. \qed

\begin{lemma}\label{Lem2.2} Let $m$ and $n>0$ be integers. Then
\begin{align}\sum_{k=1}^n\f{m^k((m-64)k^3-32k^2+16k+8)}{(2k+1)^2k^3\bi{2k}k^3}=&\f{m^{n+1}}{(2n+1)^2\bi{2n}n^3}-m,
\\\sum_{k=1}^n\f{m^k((m-64)k^3-96k^2-48k-8)}{(2k+1)^3k^3\bi{2k}k^3}=&\f{m^{n+1}}{(2n+1)^3\bi{2n}n^3}-m,
\\\sum_{k=1}^n\f{m^k((m-108)k^3-54k^2+12k+6)}{(2k+1)(3k+1)k^3\bi{2k}k^2\bi{3k}k}
=&\f{m^{n+1}}{(2n+1)(3n+1)\bi{2n}n^2\bi{3n}n}-m,
\end{align}
\begin{equation}\begin{aligned}
&\sum_{1<k\ls n}\f{m^k((m-108)k^3-(54+m)k^2+12k+6)}{(k-1)(2k+1)(3k+1)k^3\bi{2k}k^2\bi{3k}k}
\\&\qquad\qquad=\f{m^{n+1}}{n(2n+1)(3n+1)\bi{2n}n^2\bi{3n}n}-\f{m^2}{144},
\end{aligned}\end{equation}
\begin{equation}\begin{aligned}
&\sum_{k=1}^n\f{m^k((m-256)k^3-128k^2+16k+8)}{(2k+1)(4k+1)k^3\bi{2k}k^2\bi{4k}{2k}}
\\&\qquad\qquad=\f{m^{n+1}}{(2n+1)(4n+1)\bi{2n}n^2\bi{4n}{2n}}-m,
\end{aligned}\end{equation}
\begin{equation}\begin{aligned}
&\sum_{1<k\ls n}\f{m^k((m-256)k^3-(128+m)k^2+16k+8)}{(k-1)(2k+1)(4k+1)k^3\bi{2k}k^2\bi{4k}{2k}}
\\&\qquad\qquad\quad=\f{m^{n+1}}{n(2n+1)(4n+1)\bi{2n}n^2\bi{4n}{2n}}-\f{m^2}{360}.
\end{aligned}\end{equation}
\end{lemma}
\begin{remark}\rm This can be easily proved by induction on $n$.
\end{remark}

\medskip
\noindent{\it Proof of Theorem 1.2}. We just apply Lemma \ref{Lem2.2} and use the known identities:
\begin{gather*}\sum_{k=1}^\infty\f{21k-8}{k^3\bi{2k}k^3}=\f{\pi^2}6,
\ \ \sum_{k=1}^\infty\f{(4k-1)(-64)^k}{k^3\bi{2k}k^3}=-16G,
\\\sum_{k=1}^\infty\f{(3k-1)(-8)^k}{k^3\bi{2k}k^3}=-2G,
\ \ \sum_{k=1}^\infty\f{(3k-1)16^k}{k^3\bi{2k}k^3}=\f{\pi^2}2,
\\\sum_{k=1}^\infty\f{(15k-4)(-27)^{k-1}}{k^3\bi{2k}k^2\bi{3k}k}=K,\ \
 \sum_{k=1}^\infty\f{(5k-1)(-144)^k}{k^3\bi{2k}k^2\bi{4k}{2k}}=-\f{45}2K,
\\\sum_{k=1}^\infty\f{(11k-3)64^k}{k^2\bi{2k}k^2\bi{3k}k}=8\pi^2,\
\sum_{k=1}^\infty\f{(10k-3)8^k}{k^3\bi{2k}k^2\bi{3k}k}=\f{\pi^2}2,\
\ \sum_{k=1}^\infty\f{(35k-8)81^k}{k^3\bi{2k}k^2\bi{4k}{2k}}=12\pi^2.
 \end{gather*}
 Here, the first identity was found and proved by D. Zeilberger \cite{Z} in 1993.
 The second, third and fourth identities were obtained by J. Guillera \cite{G08} in 2008.
The fifth identity on $K$ was conjectured by Sun \cite{S11e} and later confirmed
by  K. Hessami Pilehrood and T. Hessami Pilehrood \cite{HP} in 2012.
The last four identities were also conjectured by Sun \cite{S11e}, and they were later
proved in the paper \cite[Theorem 3]{GR} by Guillera and M. Rogers.

Let us illustrate our proofs by proving (1.77)-(1.79) and (1.82)-(1.83) in details.

In view of (2.11) with $m=16$, we have
$$\sum_{k=1}^n\f{16^k(-48k^3-32k^2+16k+8)}{(2k+1)^2k^3\bi{2k}k^3}=\f{16^{n+1}}{(2n+1)^2\bi{2n}n^3}-16$$
for all $n\in\Z^+$, and hence
$$\sum_{k=1}^\infty\f{16^k(6k^3+4k^2-2k-1)}{(2k+1)^2k^3\bi{2k}k^3}
=\lim_{n\to+\infty}\l(\f{-2\times16^n}{(2n+1)^2\bi{2n}n^3}+2\r)=2.$$
Notice that
$$2(6k^3+4k^2-2k-1)=(2k+1)^2(3k-1)-(3k+1).$$
So we have
$$-\sum_{k=1}^\infty\f{(3k+1)16^k}{(2k+1)^2k^3\bi{2k}k^3}=2\times2-\sum_{k=1}^\infty\f{(3k-1)16^k}{k^3\bi{2k}k^3}
=4-\f{\pi^2}2$$
and hence (1.77) holds.

 By (2.11) with $m=-64$, we have
 $$\sum_{k=1}^n\f{(-64)^k(-128k^3-32k^2+16k+8)}{(2k+1)^2k^3\bi{2k}k^3}
 =\f{(-64)^{n+1}}{(2n+1)^2\bi{2n}n^3}+64$$
 for all $n\in\Z^+$, and hence
 $$\sum_{k=1}^\infty\f{(-64)^k(16k^3+4k^2-2k-1)}{(2k+1)^2k^3\bi{2k}k^3}
 =-8+\lim_{n\to+\infty}\f{8(-64)^n}{(2n+1)^2\bi{2n}n^3}=-8.$$
 Since $16k^3+4k^2-2k-1=(4k-1)(2k+1)^2-2k(4k+1)$ and $$\sum_{k=1}^\infty\f{(4k-1)(-64)^k}{k^3\bi{2k}k^3}=-16G,$$
we see that
$$-16G-2\sum_{k=1}^\infty\f{(4k+1)(-64)^k}{(2k+1)^2k^2\bi{2k}k^3}=-8$$
and hence (1.78) holds. In light of (2.12) with $m=-64$, we have
 $$\sum_{k=1}^n\f{(-64)^k(-128k^3-96k^2-48k-8)}{(2k+1)^3k^3\bi{2k}k^3}=\f{(-64)^{n+1}}{(2n+1)^3\bi{2n}n^3}+64$$
 for all $n\in\Z^+$, and hence
 $$\sum_{k=1}^\infty\f{(-64)^k(16k^3+12k^2+6k+1)}{(2k+1)^3k^3\bi{2k}k^3}
 =-8+\lim_{n\to+\infty}\f{8(-64)^n}{(2n+1)^3\bi{2n}n^3}=-8.$$
 Since $16k^3+12k^2+6k+1=2k(2k+1)(4k+1)+(4k+1)$, with the aid of (1.78) we obtain
$$\sum_{k=1}^\infty\f{(4k+1)(-64)^k}{(2k+1)^3k^3\bi{2k}k^3}=-8-2(4-8G)=16G-16.$$
This proves (1.79).

 By (2.13) with $m=-27$, we have
$$\sum_{k=1}^\infty\f{(45k^3+18k^2-4k-2)(-27)^k}{(2k+1)(3k+1)k^3\bi{2k}k^2\bi{3k}k}=-9.$$
As
$$2(45k^3+18k^2-4k-2)=(15k-4)(2k+1)(3k+1)-3k(5k+1)$$ and $$\sum_{k=1}^\infty\f{(15k-4)(-27)^k}{k^3\bi{2k}k^2\bi{3k}k}=-27K,$$
we see that (1.82) follows. By (2.14) with $m=-27$, we have
$$-3\sum_{k=2}^\infty\f{(-27)^k(45k^3+9k^2-4k-2)}{(k-1)(2k+1)(3k+1)k^3\bi{2k}k^2\bi{3k}k}=-\f{(-27)^2}{144}$$
and hence
$$\sum_{k=2}^\infty\f{(-27)^k(45k^3+9k^2-4k-2)}{(k-1)(2k+1)(3k+1)k^3\bi{2k}k^2\bi{3k}k}=\f{27}{16}.$$
As
$$45k^3+9k^2-4k-2=9(k-1)k(5k+1)+(45k^2+5k-2),$$
with the aid of (1.82) we get
\begin{align*}&\sum_{k=2}^\infty\f{(-27)^k(45k^2+5k-2)}{(k-1)(2k+1)(3k+1)k^3\bi{2k}k^2\bi{3k}k}
\\=&\f{27}{16}-9\l(6-9K-\f{6(-27)}{12^2}\r)=\f{27}{16}(48K-37)
\end{align*}
and hence $(1.83)$ follows.

Other identities in Theorem 1.2 can be proved similarly. \qed

\medskip
For integers $n\gs k\gs0$, we define
\begin{equation}\label{snk}s_{n,k}:=\f1{\bi nk}\sum_{i=0}^{k}\bi n{2i}\bi n{2(k-i)}\bi{2i}i\bi{2(k-i)}{k-i}.
\end{equation}
For $n\in\N$ we set
\begin{equation}t_n:=\sum_{0<k\ls n}\bi{n-1}{k-1}(-1)^k4^{n-k}s_{n+k,k}.
\end{equation}

\begin{lemma}\label{Lem2.3} For any $n\in\N$, we have
\begin{equation}\label{sf}\sum_{k=0}^n\bi nk(-1)^k4^{n-k}s_{n+k,k}=f_n
\end{equation}
and
\begin{equation}\label{tf}(2n+1)t_{n+1}+8nt_n=(2n+1)f_{n+1}-4(n+1)f_n,
\end{equation}
where $f_n$ denotes the Franel number $\sum_{k=0}^n\bi nk^3$.
\end{lemma}
\begin{proof} For $n,i,k\in\N$ with $i\ls k$, we set
\[
F(n,i,k) = {n \choose k} \f{(-1)^k 4^{n-k}}{\bi{n+k}k} {n+k \choose 2i}{2i \choose i} {n+k \choose 2(k-i)}{2(k-i) \choose k-i}.
\]
By the telescoping method for double summation \cite{CHM},
for
\[{\mathcal F}(n,i,k):= F(n,i,k) + \frac{7n^2+21n+16}{8(n+1)^2} F(n+1,i,k) - \frac{(n+2)^2}{8(n+1)^2} F(n+2,i,k)\]
with $0\ls i\ls k$,
we find that
\[
{\mathcal F}(n,i,k)  = (G_1(n,i+1,k)-G_1(n,i,k)) + (G_2(n,i,k+1)-G_2(n,i,k)),
\]
where
\[G_1(n,i,k):=\frac{i^2(-k+i-1)(-1)^{k+1} 4^{n-k} n!^2 (n+k)!  p(n,i,k)}{(2n+3)(n-k+2)!(n+k+2-2i)!(n-k+2i)!(i!(k-i+1)!)^2}\]
and
\[G_2(n,i,k):=\frac{ 2(k-i)(-1)^{k} 4^{n-k} n!^2 (n+k)! q(n,i,k)} {(2n+3) (n-k+2)! (n+k-2i+1)!  (n-k+2i+2)! (i!(k-i)!)^2},\]
with $(-1)!,(-2)!,\ldots$ regarded as $+\infty$, and $p(n,i,k)$ and $q(n,i,k)$ given by
\begin{align*}
&-10n^4+(i-10k-68)n^3+(-24i^2+(32k+31)i+2k^2-67k-172)n^2\\
&+(36i^3+(-68k-124)i^2+(39k^2+149k+104)i+2k^3-8k^2-145k-192)n\\
&+60i^3+(-114k-140)i^2+(66k^2+160k+92)i+3k^3-19k^2-102k-80
\end{align*}
and
\begin{align*}
&10(i-k)n^4+(-20i^2+(46k+47)i-6k^2-47k)n^3\\
+&(72i^3+(-60k-38)i^2+(22k^2+145k+90)i+4k^3-11k^2-90k)n^2\\
+&(72k+156)i^3n+(-72k^2-60k-10)i^2n+(18k^3+4k^2+165k+85)in
\\+&(22k^3-5k^2-85k)n
+(120k+60)i^3+(-120k^2+68k-4)i^2
\\+&(30k^3-56k^2+86k+32)i+26k^3-6k^2-32k
\end{align*}
respectively.
Therefore
\begin{align*}
  &\sum_{k=0}^{n+2} \sum_{i=0}^k {\mathcal F}(n,i,k)
  \\=& \sum_{k=0}^{n+2} (G_1(n,k+1,k)-G_1(n,0,k)) + \sum_{i=0}^{n+2} (G_2(n,i,n+3)-G_2(n,i,i))
  \\=&\sum_{k=0}^{n+2}(0-0)+\sum_{i=0}^{n+2}(0-0)= 0,
\end{align*}
and hence
\[
u(n):= \sum_{k=0}^n {n \choose k} (-1)^k 4^{n-k} s_{n+k,k}
\]
satisfies the recurrence relation
\[
8(n+1)^2 u(n) + (7n^2+21n+16) u(n+1) - (n+2)^2 u(n+2) = 0.
\]
As pointed out by J. Franel \cite{Fr}, the Franel numbers
satisfy the same recurrence. Note also that
$u(0)=f_0 = 1$ and $u(1)=f_1 = 2$.
So we always have $u(n)=f_n$. This proves \eqref{sf}.

The identity \eqref{tf} can be proved similarly.
In fact, if we use $v(n)$ denote the left-hand side or the right-hand side of \eqref{tf}, then
we have the recurrence
\begin{align*}&8(n+1)(n+2)(18n^3+117n^2+249n+172)v(n)
\\&+(126n^5+1197n^4+4452n^3+8131n^2+7350n+2656)v(n+1)
\\=&(n+3)^2(18n^3+63n^2+69n+22)v(n+2).
\end{align*}

In view of the above, we have completed the proof of Lemma \ref{Lem2.3}.
\end{proof}

\begin{lemma}\label{Lem2.4}For any $c\in\Z$ and $n\in\N$, we have
\begin{equation}\label{S_n}S_n(4,c)=\sum_{k=0}^{\lfloor n/2\rfloor}\bi{n-k}k\bi{2(n-k)}{n-k}c^k4^{n-2k}s_{n,k}.
\end{equation}
\end{lemma}
\begin{proof} For each $k=0,\ldots,n$, we have
\begin{align*}T_k(4,c)T_{n-k}(4,c)=&\sum_{i=0}^{\lfloor k/2\rfloor}\bi {k}{2i}\bi{2i}i4^{k-2i}c^i
\sum_{j=0}^{\lfloor(n-k)/2\rfloor}\bi{n-k}{2j}\bi{2j}j4^{n-k-2j}c^j
\\=&\sum_{r=0}^{\lfloor n/2\rfloor}c^r4^{n-2r}\sum_{i,j\in\N\atop i+j=r}\bi k{2i}\bi{n-k}{2j}\bi{2i}i\bi{2j}j.
\end{align*}
If $i,j\in\N$ and $i+j=r\ls n/2$, then
\begin{align*}\sum_{k=0}^n\bi nk^2\bi{k}{2i}\bi{n-k}{2j}
=&\bi n{2i}\bi n{2j}\sum_{k=2i}^{n-2j}\bi{n-2i}{k-2i}\bi{n-2j}{n-k-2j}
\\=&\bi n{2i}\bi n{2j}\bi{2n-2(i+j)}{n-2(i+j)}
=\bi{2n-2r}n\bi n{2i}\bi n{2j}
\end{align*}
with the aid of the Chu-Vandermonde identity. Therefore
\begin{align*}S_n(4,c)=&\sum_{k=0}^{\lfloor n/2\rfloor}c^k4^{n-2k}\bi{2n-2k}{n}\bi nks_{n,k}
\\=&\sum_{k=0}^{\lfloor n/2\rfloor}c^k4^{n-2k}\bi{2n-2k}{n-k}\bi{n-k}ks_{n,k}.
\end{align*}
This proves \eqref{S_n}.
\end{proof}

\begin{lemma}\label{Lem2.5} For $k\in\N$ and $l\in\Z^+$, we have
\begin{equation}\label{skl}s_{k+l,k}\ls (2k+1)4^kl\bi{k+l}l.
\end{equation}
\end{lemma}
\proof Let $n=k+l$. Then
\begin{align*}\bi nks_{n,k}\ls&\sum_{i,j\in\N\atop i+j=k}\bi n{2i}\bi n{2j}\sum_{i,j\in\N\atop i+j=k}\bi{2i}i\bi{2j}j
\\\ls&\sum_{s,t\in\N\atop s+t=2k}\bi ns\bi nt\sum_{i,j\in\N\atop i+j=k}4^i4^j=\bi{2n}{2k}(k+1)4^k
\end{align*}
and
\begin{align*}\f{\bi{2n}{2k}}{\bi nk}=&\f{\bi{2n}{2l}}{\bi nl}=\prod_{j=0}^{l-1}\f{2(j+k)+1}{2j+1}
\\\ls&(2k+1)\prod_{0<j<l}\f{2(j+k)}{2j}
\\=&(2k+1)\bi{k+l-1}{l-1}.
\end{align*}
Hence
$$s_{k+l,k}\ls (k+1)4^k(2k+1)\f{l}{k+l}\bi{k+l}l\ls (2k+1)4^kl\bi{k+l}l.$$
This proves \eqref{skl}. \qed
\medskip

To prove Theorem \ref{Th1.3}, we need an auxiliary theorem.
\begin{theorem}\label{Th2.1} Let $a$ and $b$ be real numbers. For any integer $m$ with $|m|\gs94$, we have
\begin{equation}\label{ab}\sum_{n=0}^\infty(an+b)\f{S_n(4,-m)}{m^n}
=\f1{m+16}\sum_{n=0}^\infty(2a(m+4)n-8a+b(m+16))\f{\bi{2n}nf_n}{m^n}.
\end{equation}
\end{theorem}
\begin{proof} Let $N\in\N$. In view of \eqref{S_n},
\begin{align*}\sum_{n=0}^N\f{S_n(4,-m)}{m^n}
=&\sum_{n=0}^N\f1{m^n}\sum_{k=0}^{\lfloor n/2\rfloor}(-m)^k4^{n-2k}\bi{2n-2k}{n-k}\bi{n-k}ks_{n,k}
\end{align*}\begin{align*}=&\sum_{l=0}^N\f{\bi{2l}l}{m^l}\sum_{k=0}^{N-l}\bi lk(-1)^k4^{l-k}s_{l+k,k}
\\=&\sum_{l=0}^{\lfloor N/2\rfloor}\f{\bi{2l}l}{m^l}\sum_{k=0}^l\bi lk(-1)^k4^{l-k}s_{l+k,k}
\\&+\sum_{N/2<l\ls N}\f{\bi{2l}l}{m^l}\sum_{k=0}^{N-l}\bi lk(-1)^k4^{l-k}s_{l+k,k}
\end{align*}
and similarly
\begin{align*}\sum_{n=0}^N\f{nS_n(4,-m)}{m^n}=&\sum_{l=0}^N\f{\bi{2l}l}{m^l}\sum_{k=0}^{N-l}\bi lk(-1)^k4^{l-k}(k+l)s_{l+k,k}
\\=&\sum_{l=0}^N\f{l\bi{2l}l}{m^l}\sum_{k=0}^{N-l}\l(\bi lk+\bi{l-1}{k-1}\r)(-1)^k4^{l-k}s_{l+k,k}
\\=&\sum_{l=0}^{\lfloor N/2\rfloor}\f{l\bi{2l}l}{m^l}\sum_{k=0}^{l}\l(\bi lk+\bi{l-1}{k-1}\r)(-1)^k4^{l-k}s_{l+k,k}
\\&+\sum_{N/2<l\ls N}\f{l\bi{2l}l}{m^l}\sum_{k=0}^{N-l}\l(\bi lk+\bi{l-1}{k-1}\r)(-1)^k4^{l-k}s_{l+k,k},
\end{align*}
where we consider $\bi x{-1}$ as $0$.

If $l$ is an integer in the interval $(N/2,N]$, then by Lemma \ref{Lem2.5} we have
\begin{align*}&\bg|\sum_{k=0}^{N-l}\bi lk(-1)^k4^{l-k}s_{l+k,k}\bg|
\\\ls&\sum_{k=0}^l\bi lk4^{l-k}s_{l+k,k}
\ls\sum_{k=0}^l\bi lk4^{l-k}(2k+1)4^kl\bi{k+l}l
\\\ls& l(2l+1)4^l\sum_{k=0}^l\bi lk\bi{l+k}k=l(2l+1)4^lP_l(3),
\end{align*}
where $P_l(x)$ is the Legendre polynomial of degree $l$. Thus
\begin{align*}&\bg|\sum_{N/2<l\ls N}\f{\bi{2l}l}{m^l}\sum_{k=0}^{N-l}\bi lk(-1)^k4^{l-k}s_{l+k,k}\bg|
\\\ls&\sum_{N/2<l\ls N}l(2l+1)\l(\f{16}m\r)^lP_l(3)
\ls\sum_{l>N/2}l(2l+1)P_l(3)\l(\f{16}m\r)^l
\end{align*}
and
\begin{align*}&\bg|\sum_{N/2<l\ls N}\f{l\bi{2l}l}{m^l}\sum_{k=0}^{N-l}\l(\bi lk+\bi{l-1}{k-1}\r)(-1)^k4^{l-k}s_{l+k,k}\bg|
\\\ls&\sum_{N/2<l\ls N}\f{l4^l}{m^l}\sum_{k=0}^l 2\bi lk4^{l-k}s_{l+k,k}
\end{align*}\begin{align*}\ls&\sum_{N/2<l\ls N}2l^2(2l+1)\l(\f{16}m\r)^lP_l(3)
\\\ls&2\sum_{l>N/2}l^2(2l+1)P_l(3)\l(\f{16}m\r)^l.
\end{align*}

Recall that
$$P_l(3)=T_l(3,2)\sim\f{(3+2\sqrt2)^{l+1/2}}{2\root4\of{2}\sqrt{l\pi}}\ \ \t{as}\ l\to+\infty.$$
As $|m|\gs94$, we have $|m|>16(3+2\sqrt2)\approx 93.255$ and hence
$$\sum_{l=0}^\infty l^2(2l+1)P_l(3)\l(\f{16}m\r)^l$$ converges.
Thus
$$\lim_{N\to+\infty}\sum_{l>N/2}l(2l+1)P_l(3)\l(\f{16}m\r)^l=0
=\lim_{N\to+\infty}\sum_{l>N/2}l^2(2l+1)P_l(3)\l(\f{16}m\r)^l$$
and hence by the above we have
$$\sum_{n=0}^\infty\f{S_n(4,-m)}{m^n}
=\sum_{l=0}^{\infty}\f{\bi{2l}l}{m^l}\sum_{k=0}^l\bi lk(-1)^k4^{l-k}s_{l+k,k}$$
and
$$\sum_{n=0}^\infty\f{nS_n(4,-m)}{m^n}
=\sum_{l=0}^{\infty}\f{l\bi{2l}l}{m^l}\sum_{k=0}^{l}\l(\bi lk+\bi{l-1}{k-1}\r)(-1)^k4^{l-k}s_{l+k,k}.$$
Therefore, with the aid of \eqref{sf}, we obtain
\begin{equation}\label{fs}\sum_{n=0}^\infty\f{S_n(4,-m)}{m^n}=\sum_{n=0}^\infty\f{\bi{2n}n}{m^n}f_n
\end{equation}
and
\begin{equation}\label{ft}
\sum_{n=0}^\infty\f{nS_n(4,-m)}{m^n}=\sum_{n=0}^\infty\f{n\bi{2n}n}{m^n}(f_n+t_n).
\end{equation}

In view of \eqref{ft} and \eqref{tf},
\begin{align*}&(m+16)\sum_{n=0}^\infty\f{nS_n(4,-m)}{m^n}
\\=&\sum_{n=1}^\infty\f{n\bi{2n}n}{m^{n-1}}(f_n+t_n)+16\sum_{n=0}^\infty\f{n\bi{2n}n}{m^n}(f_n+t_n)
\\=&\sum_{n=0}^\infty \f{(n+1)\bi{2n+2}{n+1}(f_{n+1}+t_{n+1})+16n\bi{2n}n(f_n+t_n)}{m^n}
\\
=&2\sum_{n=0}^\infty\f{\bi{2n}n}{m^n}\l((2n+1)(f_{n+1}+t_{n+1})+8n(f_n+t_n)\r)
\end{align*}\begin{align*}=&2\sum_{n=0}^\infty\f{\bi{2n}n}{m^n}\l(2(2n+1)f_{n+1}+4(n-1)f_n\r)
\\=&2\sum_{n=0}^\infty\f{(n+1)\bi{2n+2}{n+1}f_{n+1}}{m^n}+8\sum_{n=0}^\infty\f{(n-1)\bi{2n}nf_n}{m^n}
\\=&2\sum_{n=0}^\infty\f{n\bi{2n}nf_n}{m^{n-1}}+8\sum_{n=0}^\infty\f{(n-1)\bi{2n}nf_n}{m^n}
=2\sum_{n=0}^\infty((m+4)n-4)\f{\bi{2n}nf_n}{m^n}.
\end{align*}
Combining this with \eqref{fs}, we immediately obtain the desired \eqref{ab}.
\end{proof}

\medskip
\noindent{\it Proof of Theorem \ref{Th1.3}}. Let $a,b,m\in\Z$ with $|m|\gs6$.
Since
$$4^nT_n(1,m)=\sum_{k=0}^{\lfloor n/2\rfloor}\bi n{2k}\bi{2k}k4^{n-2k}(16m)^k=T_n(4,16m)$$
for any $n\in\N$, we have $4^nS_n(1,m)=S_n(4,16m)$ for all $n\in\N$. Thus,
in light of Theorem \ref{Th2.1},
\begin{align*}&\sum_{n=0}^\infty(an+b)\f{S_n(1,m)}{(-4m)^n}
\\=&\sum_{n=0}^\infty(an+b)\f{S_n(4,16m)}{(-16m)^n}
\\=&\f1{16-16m}\sum_{n=0}^\infty(2a(4-16m)n-8a+(16-16m)b)\f{\bi{2n}nf_n}{(-16m)^n}
\\=&\f1{2(m-1)}\sum_{n=0}^\infty(a(4m-1)n+a+2b(m-1))\f{\bi{2n}nf_n}{(-16m)^n}.
\end{align*}
Therefore
\begin{align*}\sum_{k=0}^\infty\f{7k+3}{24^k}S_k(1,-6)=&\f52\sum_{k=0}^\infty\f{5k+1}{96^k}\bi{2k}kf_k,
\\\sum_{k=0}^\infty\f{12k+5}{(-28)^k}S_k(1,7)=&3\sum_{k=0}^\infty\f{9k+2}{(-112)^k}\bi{2k}kf_k,\\
\sum_{k=0}^\infty\f{84k+29}{80^k}S_k(1,-20)=&27\sum_{k=0}^\infty\f{6k+1}{320^k}\bi{2k}kf_k,
\\\sum_{k=0}^\infty\f{3k+1}{(-100)^k}S_k(1,25)=&\f1{16}\sum_{k=0}^\infty\f{99k+17}{(-400)^k}\bi{2k}kf_k,
\\\sum_{k=0}^\infty\f{228k+67}{224^k}S_k(1,-56)=&5\sum_{k=0}^\infty\f{90k+13}{896^k}\bi{2k}kf_k,
\\\sum_{k=0}^\infty\f{399k+101}{(-676)^k}S_k(1,169)=&\f{15}{16}\sum_{k=0}^\infty\f{855k+109}{(-2704)^k}\bi{2k}kf_k,
\\
\sum_{k=0}^\infty\f{2604k+563}{2600^k}S_k(1,-650)=&51\sum_{k=0}^\infty\f{102k+11}{10400^k}\bi{2k}kf_k,
\end{align*}
\begin{align*}\sum_{k=0}^\infty\f{39468k+7817}{(-6076)^k}S_k(1,1519)=&135\sum_{k=0}^\infty\f{585k+58}{(-24304)^k}\bi{2k}kf_k,
\\\sum_{k=0}^\infty\f{41667k+7879}{9800^k}S_k(1,-2450)=&\f{297}2\sum_{k=0}^\infty\f{561k+53}{39200^k}\bi{2k}kf_k,
\\\sum_{k=0}^\infty\f{74613k+10711}{(-530^2)^k}S_k(1,265^2)=&\f{23}{32}\sum_{k=0}^\infty\f{207621k+14903}{(-1060^2)^k}\bi{2k}kf_k.
\end{align*}
It is known (cf. \cite{CTYZ,ChCo})
that
\begin{gather*}\sum_{k=0}^\infty\f{5k+1}{96^k}\bi{2k}kf_k=\f{3\sqrt2}{\pi},
\ \ \sum_{k=0}^\infty\f{9k+2}{(-112)^k}\bi{2k}kf_k=\f{2\sqrt7}{\pi},
\\ \sum_{k=0}^\infty\f{6k+1}{320^k}\bi{2k}kf_k=\f{8\sqrt{15}}{9\pi},\ \ \sum_{k=0}^\infty\f{99k+17}{(-400)^k}\bi{2k}kf_k=\f{50}{\pi},
\\\sum_{k=0}^\infty\f{90k+13}{896^k}\bi{2k}kf_k=\f{16\sqrt7}{\pi},\ \ \sum_{k=0}^\infty\f{855k+109}{(-2704)^k}\bi{2k}kf_k=\f{338}{\pi},
\\\sum_{k=0}^\infty\f{102k+11}{10400^k}\bi{2k}kf_k=\f{50\sqrt{39}}{9\pi},\ \ \sum_{k=0}^\infty\f{585k+58}{(-24304)^k}\bi{2k}kf_k=\f{98\sqrt{31}}{3\pi},
\\\sum_{k=0}^\infty\f{561k+53}{39200^k}\bi{2k}kf_k=\f{1225\sqrt6}{18\pi},
\ \ \sum_{k=0}^\infty\f{207621k+14903}{(-1060^2)^k}\bi{2k}kf_k=\f{140450}{3\pi}.
\end{gather*}
So we get the identities \eqref{S1}-\eqref{S10} finally. \qed

\section{New series involving $T_n(b,c)$ for $1/\pi$ related to types I-VIII}
\setcounter{equation}{0}

Now we pose a conjecture related to the series (I1)-(I4) of Sun \cite{S-11,S14c}.

\begin{conjecture}\label{Conj-I} We have the following identities:
\[\sum_{k=0}^\infty\f{50k+1}{(-256)^k}\bi{2k}k\bi{2k}{k+1}T_k(1,16)=\f{8}{3\pi},\tag{I1$'$}\]
\[\sum_{k=0}^\infty\f{(100k^2-4k-7)\bi{2k}k^2T_k(1,16)}{(2k-1)^2(-256)^k}=-\f{24}{\pi},\tag{I1$''$}\]
\[\sum_{k=0}^\infty\f{30k+23}{(-1024)^k}\bi{2k}k\bi{2k}{k+1}T_k(34,1)=-\f{20}{3\pi},\tag{I2$'$}\]
\[\sum_{k=0}^\infty\f{(36k^2-12k+1)\bi{2k}k^2T_k(34,1)}{(2k-1)^2(-1024)^k}=-\f{6}{\pi},\tag{I2$''$}\]
\[\sum_{k=0}^\infty\f{110k+103}{4096^k}\bi{2k}k\bi{2k}{k+1}T_k(194,1)=\f{304}{\pi},\tag{I3$'$}\]
\[\sum_{k=0}^\infty\f{(20k^2+28k-11)\bi{2k}k^2T_k(194,1)}{(2k-1)^2 4096^k}=-\f{6}{\pi},\tag{I3$''$}\]
\[\sum_{k=0}^\infty\f{238k+263}{4096^k}\bi{2k}k\bi{2k}{k+1}T_k(62,1)=\f{112\sqrt3}{3\pi},\tag{I4$'$}\]
\[\sum_{k=0}^\infty\f{(44k^2+4k-5)\bi{2k}k^2T_k(62,1)}{(2k-1)^2 4096^k}=-\f{4\sqrt3}{\pi},\tag{I4$''$}\]
\[\sum_{k=0}^\infty\f{6k+1}{256^k}\bi{2k}k^2T_k(8,-2)=\f{2}{\pi}\sqrt{8+6\sqrt2},\tag{I5}\]
\[\sum_{k=0}^\infty\f{2k+3}{256^k}\bi{2k}k\bi{2k}{k+1}T_k(8,-2)=\f{6\sqrt{8+6\sqrt2}-16\root4\of{2}}{3\pi},\tag{I5$'$}\]
\[\sum_{k=0}^\infty\f{(4k^2+2k-1)\bi{2k}k^2T_k(8,-2)}{(2k-1)^2 256^k}=-\f{3\root4\of{2}}{4\pi}.\tag{I5$''$}\]
\end{conjecture}
\begin{remark}\label{Rem3.1} For each $k\in\N$, we have
$$((1+\lambda_0-\lambda_1)k+\lambda_0)C_k=(k+\lambda_0)\bi{2k}k-(k+\lambda_1)\bi{2k}{k+1}$$
since $\bi{2k}k=(k+1)C_k$ and $\bi{2k}{k+1}=kC_k$. Thus, for example, \cite[(I1)]{S14c} and (I1$'$) together imply that
$$\sum_{k=0}^\infty\f{26k+5}{(-256)^k}\bi{2k}kC_kT_k(1,16)=\f{16}{\pi},$$
and (I5) and (I5$'$) imply that
$$\sum_{k=0}^\infty\f{2k-1}{256^k}\bi{2k}kC_kT_k(8,-2)=\f4{\pi}\l(\sqrt{8+6\sqrt2}-4\root4\of2\r).$$
For the conjectural identities in Conjecture \ref{Conj-I}, we have conjectures for the corresponding $p$-adic congruences. For example, in contrast with (I2$'$), we conjecture that for any prime $p>3$ we have
the congruences
$$\sum_{k=0}^{p-1}\f{30k+23}{(-1024)^k}\bi{2k}k\bi{2k}{k+1}T_k(34,1)\eq\f p3\l(21\l(\f 2p\r)-10\l(\f{-1}p\r)-11\r)\pmod{p^2}$$
and
$$\sum_{k=0}^{p-1}\f{2k+1}{(-1024)^k}\bi{2k}kC_kT_k(34,1)\eq\f p3\l(2-3\l(\f 2p\r)+4\l(\f{-1}p\r)\r)\pmod{p^2}.$$
Concerning (I5) and (I5$''$), we conjecture that
$$\f1{2^{\lfloor n/2\rfloor+1}n\bi{2n}n}\sum_{k=0}^{n-1}(6k+1)\bi{2k}k^2T_k(8,-2)256^{n-1-k}\in\Z^+$$
and $$\f1{\bi{2n-2}{n-1}}\sum_{k=0}^{n-1}\f{(1-2k-4k^2)\bi{2k}k^2T_k(8,-2)}{(2k-1)^2 256^k}\in\Z^+$$
for each $n=2,3,\ldots$, and that for any prime $p\eq1\pmod4$ with $p=x^2+4y^2\ (x,y\in\Z)$ we have
$$\sum_{k=0}^{p-1}\f{\bi{2k}k^2T_k(8,-2)}{256^k}
\eq\begin{cases}(-1)^{y/2}(4x^2-2p)\pmod{p^2}&\t{if}\ p\eq1\pmod 8,
\\(-1)^{(xy-1)/2}8xy\pmod{p^2}&\t{if}\ p\eq5\pmod 8,
\end{cases}$$
and
$$\sum_{k=0}^{p-1}\f{(4k^2+2k-1)\bi{2k}k^2T_k(8,-2)}{(2k-1)^2256^k}\eq0\pmod{p^2}.$$
By \cite[Theorem 5.1]{S14c}, we have
$$ \sum_{k=0}^{p-1}\f{\bi{2k}k^2T_k(8,-2)}{256^k}\eq0\pmod{p^2}$$
for any prime $p\eq3\pmod4$.
The identities (I5), (I5$'$) and (I5$''$) were formulated by the author on Dec. 9, 2019.
\end{remark}

Next we pose a conjecture related to the series (II1)-(II7) and (II10)-(II12) of Sun \cite{S-11,S14c}.

\begin{conjecture}\label{Conj-II} We have the following identities:
\[\sum_{k=0}^\infty\f{3k+4}{972^k}\bi{2k}{k+1}\bi{3k}kT_k(18,6)=\f{63\sqrt3}{40\pi},\tag{II1$'$}\]
\[\sum_{k=0}^\infty\f{91k+107}{10^{3k}}\bi{2k}{k+1}\bi{3k}kT_k(10,1)=\f{275\sqrt3}{18\pi},\tag{II2$'$}\]
\[\sum_{k=0}^\infty\f{195k+83}{18^{3k}}\bi{2k}{k+1}\bi{3k}{k}T_k(198,1)=\f{9423\sqrt3}{10\pi},\tag{II3$'$}\]
\[\sum_{k=0}^\infty\f{483k-419}{30^{3k}}\bi{2k}{k+1}\bi{3k}kT_k(970,1)=\f{6550\sqrt3}{\pi},\tag{II4$'$}\]
\[\sum_{k=0}^\infty\f{666k+757}{30^{3k}}\bi{2k}{k+1}\bi{3k}{k}T_k(730,729)=\f{3475\sqrt3}{4\pi},\tag{II5$'$}\]
\[\sum_{k=0}^\infty\f{8427573k+8442107}{102^{3k}}\bi{2k}{k+1}\bi{3k}kT_k(102,1)=\f{125137\sqrt6}{20\pi},\tag{II6$'$}\]
\[\sum_{k=0}^\infty\f{959982231k+960422503}{198^{3k}}\bi{2k}{k+1}\bi{3k}kT_k(198,1)=\f{5335011\sqrt3}{20\pi},\tag{II7$'$}\]
\[\sum_{k=0}^\infty\f{99k+1}{24^{3k}}\bi{2k}{k+1}\bi{3k}kT_k(26,729)=\f{16(289\sqrt{15}-645\sqrt3)}{15\pi},\tag{II10$'$}\]
\[\sum_{k=0}^\infty\f{45k+1}{(-5400)^k}\bi{2k}{k+1}\bi{3k}kT_k(70,3645)=\f{345\sqrt3-157\sqrt{15}}{6\pi},\tag{II11$'$}\]
\[\sum_{k=0}^\infty\f{252k-1}{(-13500)^k}\bi{2k}{k+1}\bi{3k}kT_k(40,1458)=\f{25(1212\sqrt3-859\sqrt6)}{24\pi},\tag{II12$'$}\]
\[\sum_{k=0}^\infty\f{9k+2}{(-675)^k}\bi{2k}{k}\bi{3k}kT_k(15,-5)=\f{7\sqrt{15}}{8\pi},\tag{II13}\]
\[\sum_{k=0}^\infty\f{45k+31}{(-675)^k}\bi{2k}{k+1}\bi{3k}kT_k(15,-5)=-\f{19\sqrt{15}}{8\pi},\tag{II13$'$}\]
\[\sum_{k=0}^\infty\f{39k+7}{(-1944)^k}\bi{2k}{k}\bi{3k}kT_k(18,-3)=\f{9\sqrt{3}}{\pi},\tag{II14}\]
\[\sum_{k=0}^\infty\f{312k+263}{(-1944)^k}\bi{2k}{k+1}\bi{3k}kT_k(18,-3)=-\f{45\sqrt{3}}{2\pi}.\tag{II14$'$}\]
\end{conjecture}
\begin{remark}\rm We also have conjectures on related congruences. For example, concerning
(II14), for any prime $p>3$ we conjecture that
$$\sum_{k=0}^{p-1}\f{39k+7}{(-1944)^k}\bi{2k}k\bi{3k}kT_k(18,-3)\eq\f p2\l(13\l(\f p3\r)+1\r)
\pmod{p^2}$$
and that
\begin{align*}&\sum_{k=0}^{p-1}\f{\bi{2k}k\bi{3k}kT_k(18,-3)}{(-1944)^k}
\\\eq&\begin{cases}4x^2-2p\pmod{p^2}&\t{if}\ (\f{-1}p)=(\f p3)=(\f p7)=1\ \&\ p=x^2+21y^2,
\\2p-2x^2\pmod{p^2}&\t{if}\ (\f{-1}p)=(\f p3)=-1,\ (\f p7)=1\ \&\ 2p=x^2+21y^2,
\\12x^2-2p\pmod{p^2}&\t{if}\ (\f{-1}p)=(\f p7)=-1,\ (\f p3)=1\ \&\ p=3x^2+7y^2,
\\2p-6x^2\pmod{p^2}&\t{if}\ (\f{-1}p)=1,\ (\f p3)=(\f p7)=-1\ \&\ 2p=3x^2+7y^2,
\\0\pmod{p^2}&\t{if}\ (\f{-21}p)=-1,
\end{cases}
\end{align*}
where $x$ and $y$ are integers.
The identities (II13), (II13$'$), (II14) and (II14$'$) were found by the author
on Dec. 11, 2019.
\end{remark}

The following conjecture is related to the series (III1)-(III10) and (III12) of Sun \cite{S-11,S14c}.

\begin{conjecture}\label{Conj-III} We have the following identities:
\[\sum_{k=0}^\infty\f{17k+18}{66^{2k}}\bi{2k}{k+1}\bi{4k}{2k}T_k(52,1)=\f{77\sqrt{33}}{12\pi},\tag{III1$'$}\]
\[\sum_{k=0}^\infty\f{4k+3}{(-96^2)^k}\bi{2k}{k+1}\bi{4k}{2k}T_k(110,1)=-\f{\sqrt6}{3\pi},\tag{III2$'$}\]
\[\sum_{k=0}^\infty\f{8k+9}{112^{2k}}\bi{2k}{k+1}\bi{4k}{2k}T_k(98,1)=\f{154\sqrt{21}}{135\pi},\tag{III3$'$}\]
\[\sum_{k=0}^\infty\f{3568k+4027}{264^{2k}}\bi{2k}{k+1}\bi{4k}{2k}T_k(257,256)=\f{869\sqrt{66}}{10\pi},\tag{III4$'$}\]
\[\sum_{k=0}^\infty\f{144k+1}{(-168^2)^{k}}\bi{2k}{k+1}\bi{4k}{2k}T_k(7,4096)=\f{7(1745\sqrt{42}-778\sqrt{210})}{120\pi},\tag{III5$'$}\]
\[\sum_{k=0}^\infty\f{3496k+3709}{336^{2k}}\bi{2k}{k+1}\bi{4k}{2k}T_k(322,1)=\f{182\sqrt7}{\pi},\tag{III6$'$}\]
\[\sum_{k=0}^\infty\f{286k+229}{336^{2k}}\bi{2k}{k+1}\bi{4k}{2k}T_k(1442,1)=\f{1113\sqrt{210}}{20\pi},\tag{III7$'$}\]
\[\sum_{k=0}^\infty\f{8426k+8633}{912^{2k}}\bi{2k}{k+1}\bi{4k}{2k}T_k(898,1)=\f{703\sqrt{114}}{20\pi},\tag{III8$'$}\]
\[\sum_{k=0}^\infty\f{1608k+79}{912^{2k}}\bi{2k}{k+1}\bi{4k}{2k}T_k(12098,1)=\f{67849\sqrt{399}}{105\pi},\tag{III9$'$}\]
\[\sum_{k=0}^\infty\f{134328722k+134635283}{10416^{2k}}\bi{2k}{k+1}\bi{4k}{2k}T_k(10402,1)=\f{93961\sqrt{434}}{4\pi},\tag{III10$'$}\]
and
\begin{equation*}\begin{aligned}&\sum_{k=0}^\infty\f{39600310408k+39624469807}{39216^{2k}}\bi{2k}{k+1}\bi{4k}{2k}T_k(39202,1)
\\&\qquad\qquad=\f{1334161\sqrt{817}}{\pi}.\end{aligned}\tag{III12$'$}\end{equation*}
\end{conjecture}

The following conjecture is related to the series (IV1)-(IV21) of Sun \cite{S-11,S14c}.

\begin{conjecture}\label{Conj-IV} We have the following identities:
\[\sum_{k=0}^\infty\f{(356k^2+288k+7)\bi{2k}k^2T_{2k}(7,1)}{(k+1)(2k-1)(-48^2)^k}=-\f{304}{3\pi},\tag{IV1$'$}\]
\[\sum_{k=0}^\infty\f{(172k^2+141k-1)\bi{2k}k^2T_{2k}(62,1)}{(k+1)(2k-1)(-480^2)^k}=-\f{80}{3\pi},\tag{IV2$'$}\]
\[\sum_{k=0}^\infty\f{(782k^2+771k+19)\bi{2k}k^2T_{2k}(322,1)}{(k+1)(2k-1)(-5760^2)^k}=-\f{90}{\pi},\tag{IV3$'$}\]
\[\sum_{k=0}^\infty\f{(34k^2+45k+5)\bi{2k}k^2T_{2k}(10,1)}{(k+1)(2k-1)96^{2k}}=-\f{20\sqrt2}{3\pi},\tag{IV4$'$}\]
\[\sum_{k=0}^\infty\f{(106k^2+193k+27)\bi{2k}k^2T_{2k}(38,1)}{(k+1)(2k-1)240^{2k}}=-\f{10\sqrt6}{\pi},\tag{IV5$'$}\]
\[\sum_{k=0}^\infty\f{(214166k^2+221463k+7227)\bi{2k}k^2T_{2k}(198,1)}{(k+1)(2k-1)39200^{2k}} =-\f{9240\sqrt6}{\pi},\tag{IV6$'$}\]
\[\sum_{k=0}^\infty\f{(112k^2+126k+9)\bi{2k}k^2T_{2k}(18,1)}{(k+1)(2k-1)320^{2k}}=-\f{6\sqrt{15}}{\pi},\tag{IV7$'$}\]
\[\sum_{k=0}^\infty\f{(926k^2+995k+55)\bi{2k}k^2T_{2k}(30,1)}{(k+1)(2k-1)896^{2k}}=-\f{60\sqrt7}{\pi},\tag{IV8$'$}\]
\[\sum_{k=0}^\infty\f{(1136k^2+2962k+503)\bi{2k}k^2T_{2k}(110,1)}{(k+1)(2k-1)24^{4k}}=-\f{90\sqrt7}{\pi},\tag{IV9$'$}\]
\[\sum_{k=0}^\infty\f{(5488k^2+8414k+901)\bi{2k}k^2T_{2k}(322,1)}{(k+1)(2k-1)48^{4k}}=-\f{294\sqrt7}{\pi},\tag{IV10$'$}\]
\[\sum_{k=0}^\infty\f{(170k^2+193k+11)\bi{2k}k^2T_{2k}(198,1)}{(k+1)(2k-1)2800^{2k}}=-\f{6\sqrt{14}}{\pi},\tag{IV11$'$}\]
\[\sum_{k=0}^\infty\f{(104386k^2+108613k+4097)\bi{2k}k^2T_{2k}(102,1)}{(k+1)(2k-1)10400^{2k}} =-\f{2040\sqrt{39}}{\pi},\tag{IV12$'$}\]
\[\sum_{k=0}^\infty\f{(7880k^2+8217k+259)\bi{2k}k^2T_{2k}(1298,1)}{(k+1)(2k-1)46800^{2k}}=-\f{144\sqrt{26}}{\pi},\tag{IV13$'$}\]
\[\sum_{k=0}^\infty\f{(6152k^2+45391k+9989)\bi{2k}k^2T_{2k}(1298,1)}{(k+1)(2k-1)5616^{2k}}=-\f{663\sqrt3}{\pi},\tag{IV14$'$}\]
\[\sum_{k=0}^\infty\f{(147178k^2+2018049k+471431)\bi{2k}k^2T_{2k}(4898,1)}{(k+1)(2k-1)20400^{2k}}=-3740\f{\sqrt{51}}{\pi},\tag{IV15$'$}\]
\[\sum_{k=0}^\infty\f{(1979224k^2+5771627k+991993)\bi{2k}k^2T_{2k}(5778,1)}{(k+1)(2k-1)28880^{2k}}=-73872\f{\sqrt{10}}{\pi},\tag{IV16$'$}\]
\[\sum_{k=0}^\infty\f{(233656k^2+239993k+5827)\bi{2k}k^2T_{2k}(5778,1)}{(k+1)(2k-1)439280^{2k}}=-4080\f{\sqrt{19}}{\pi},\tag{IV17$'$}\]
\[\sum_{k=0}^\infty\f{(5890798k^2+32372979k+6727511)\bi{2k}k^2T_{2k}(54758,1)}{(k+1)(2k-1)243360^{2k}} =-600704\f{\sqrt{95}}{9\pi},\tag{IV18$'$}\]
\[\sum_{k=0}^\infty\f{(148k^2+272k+43)\bi{2k}k^2T_{2k}(10,-2)}{(k+1)(2k-1)4608^{k}}=-28\f{\sqrt{6}}{\pi},\tag{IV19$'$}\]
\[\sum_{k=0}^\infty\f{(3332k^2+17056k+3599)\bi{2k}k^2T_{2k}(238,-14)}{(k+1)(2k-1)1161216^{k}}=-744\f{\sqrt{2}}{\pi},\tag{IV20$'$}\]
\[\sum_{k=0}^\infty\f{(11511872k^2+10794676k+72929)\bi{2k}k^2T_{2k}(9918,-19)}{(k+1)(2k-1)(-16629048064)^{k}} =-390354\f{\sqrt{7}}{\pi}.\tag{IV21$'$}\]
\end{conjecture}

For the five open conjectural series (VI1), (VI2), (VI3), (VII2) and (VII7) of Sun \cite{S-11,S14c},
we make the following conjecture on related supercongruences.

\begin{conjecture}\label{Conj-T} Let $p$ be an odd prime and let $n\in\Z^+$.
If $(\f 3p)=1$, then
\begin{equation*}\label{VI1}
\sum_{k=0}^{pn-1}\f{66k+17}{(2^{11}3^3)^k}T_k(10,11^2)^3
-p\l(\f{-2}p\r)\sum_{k=0}^{n-1}\f{66k+17}{(2^{11}3^3)^k}T_k(10,11^2)^3
\end{equation*}
divided by $(pn)^2$ is a $p$-adic integer.
If $p\not=5$, then
\begin{equation*}\label{VI2}
\sum_{k=0}^{pn-1}\f{126k+31}{(-80)^{3k}}T_k(22,21^2)^3
-p\l(\f{-5}p\r)\sum_{k=0}^{n-1}\f{126k+31}{(-80)^{3k}}T_k(22,21^2)^3
\end{equation*}
divided by $(pn)^2$ is a $p$-adic integer.
If $(\f 7p)=1$ but $p\not=3$, then
\begin{equation*}\label{VI3}
\sum_{k=0}^{pn-1}\f{3990k+1147}{(-288)^{3k}}T_k(62,95^2)^3
-p\l(\f{-2}p\r)\sum_{k=0}^{n-1}\f{3990k+1147}{(-288)^{3k}}T_k(62,95^2)^3
\end{equation*}
divided by $(pn)^2$ is a $p$-adic integer.
If $p\eq\pm1\pmod8$ but $p\not=7$, then
\begin{equation*}\label{VII2}
\sum_{k=0}^{pn-1}\f{24k+5}{28^{2k}}\bi{2k}kT_k(4,9)^2
-p\l(\f p3\r)\sum_{k=0}^{n-1}\f{24k+5}{28^{2k}}\bi{2k}kT_k(4,9)^2
\end{equation*}
divided by $(pn)^2$ is a $p$-adic integer.
If $(\f{-6}p)=1$ but $p\not=7,31$, then
\begin{equation*}\label{VII2}\begin{aligned}
&\sum_{k=0}^{pn-1}\f{2800512k+435257}{434^{2k}}\bi{2k}kT_k(73,576)^2
\\&-p\sum_{k=0}^{n-1}\f{2800512k+435257}{434^{2k}}\bi{2k}kT_k(73,576)^2
\end{aligned}
\end{equation*}
divided by $(pn)^2$ is a $p$-adic integer.
\end{conjecture}

Now we pose four conjectural series for $1/\pi$ of type VIII.

\begin{conjecture}\label{VIII} We have
\[\sum_{k=0}^\infty\f{40k+13}{(-50)^k}T_k(4,1)T_k(1,-1)^2=\f{55\sqrt{15}}{9\pi},\tag{VIII1}\]
\[\sum_{k=0}^\infty\f{1435k+113}{3240^k}T_k(7,1)T_k(10,10)^2=\f{1452\sqrt{5}}{\pi},\tag{VIII2}\]
\[\sum_{k=0}^\infty\f{840k+197}{(-2430)^k}T_k(8,1)T_k(5,-5)^2=\f{189\sqrt{15}}{2\pi},\tag{VIII3}\]
\[\sum_{k=0}^\infty\f{39480k+7321}{(-29700)^k}T_k(14,1)T_k(11,-11)^2=\f{6795\sqrt{5}}{\pi}.\tag{VIII4}\]
\end{conjecture}
\begin{remark} The author found the identity (VIII1) on Nov. 3, 2019. The identities
(VIII2), (VIII3) and (VIII4) were formulated on Nov. 4, 2019.
\end{remark}

Below we present some conjectures on congruences related to Conjecture \ref{VIII}.

\begin{conjecture}\label{VIII-1} {\rm (i)} For each $n\in\Z^+$, we have
\begin{equation}\label{8-1-n}\f1n\sum_{k=0}^{n-1}(40k+13)(-1)^k50^{n-1-k}T_k(4,1)T_k(1,-1)^2\in\Z^+,
\end{equation}
and this number is odd if and only if $n$ is a power of two $($i.e., $n\in\{2^a:\ a\in\N\})$.

{\rm (ii)} Let $p\not=2,5$ be a prime. Then
\begin{equation}\label{8-1-p}\sum_{k=0}^{p-1}\f{40k+13}{(-50)^k}T_k(4,1)T_k(1,-1)^2\eq\f p3\l(12+5\l(\f3p\r)+22\l(\f {-15}p\r)\r)
\pmod{p^2}.\end{equation}
If $(\f 3p)=(\f{-5}p)=1$, then
\begin{equation}\label{8-1-pn}
\f1{(pn)^2}\(\sum_{k=0}^{pn-1}\f{40k+13}{(-50)^k}T_k(4,1)T_k(1,-1)^2-p\sum_{k=0}^{n-1}\f{40k+13}{(-50)^k}T_k(4,1)T_k(1,-1)^2\)
\in\Z_p
\end{equation}
for all $n\in\Z^+$.

{\rm (iii)} Let $p\not=2,5$ be a prime. Then
\begin{equation}\label{8-1-q}
\begin{aligned}&\sum_{k=0}^{p-1}\f{T_k(4,1)T_k(1,-1)^2}{(-50)^k}
\\\eq&\begin{cases}4x^2-2p\pmod{p^2}&\t{if}\ p\eq1,9\pmod{20}\ \&\ p=x^2+5y^2\ (x,y\in\Z),
\\2x^2-2p\pmod{p^2}&\t{if}\ p\eq3,7\pmod{20}\ \&\ 2p=x^2+5y^2\ (x,y\in\Z),
\\0\pmod{p^2}&\t{if}\ (\f{-5}p)=-1.\end{cases}
\end{aligned}
\end{equation}
\end{conjecture}
\begin{remark}\label{Rem-8-1}\rm The imaginary quadratic field $\Q(\sqrt{-5})$ has class number two.
\end{remark}

\begin{conjecture}
{\rm (i)} For any $n\in\Z^+$, we have
\begin{equation}\label{8-1-d-n}
\f1n\sum_{k=0}^{n-1}(40k+27)(-6)^{n-1-k}T_k(4,1)T_k(1,-1)^2\in\Z,
\end{equation}
and the number is odd if and only if $n$ is a power of two.

{\rm (ii)} Let $p>3$ be a prime. Then
\begin{equation}\label{8-1-d-p}
\sum_{k=0}^{p-1}\f{40k+27}{(-6)^k}T_k(4,1)T_k(1,-1)^2\eq\f p9\l(55\l(\f{-5}p\r)+198\l(\f 3p\r)-10\r)
\pmod{p^2}.
\end{equation}
If $(\f 3p)=(\f{-5}p)=1$, then
\begin{equation}\label{8-1-d-pn}
\f1{(pn)^2}\(\sum_{k=0}^{pn-1}\f{40k+27}{(-6)^k}T_k(4,1)T_k(1,-1)^2-p\sum_{k=0}^{n-1}\f{40k+27}{(-6)^k}T_k(4,1)T_k(1,-1)^2\)
\in\Z_p
\end{equation}
for all $n\in\Z^+$.

{\rm (iii)} Let $p>5$ be a prime. Then
\begin{equation}\label{8-1-d-q}
\begin{aligned}&\l(\f p3\r)\sum_{k=0}^{p-1}\f{T_k(4,1)T_k(1,-1)^2}{(-6)^k}
\\\eq&\begin{cases}4x^2-2p\pmod{p^2}&\t{if}\ p\eq1,9\pmod{20}\ \&\ p=x^2+5y^2\ (x,y\in\Z),
\\2p-2x^2\pmod{p^2}&\t{if}\ p\eq3,7\pmod{20}\ \&\ 2p=x^2+5y^2\ (x,y\in\Z),
\\0\pmod{p^2}&\t{if}\ (\f{-5}p)=-1.\end{cases}
\end{aligned}
\end{equation}
\end{conjecture}
\begin{remark}\rm \label{Rem-8-1-d} This conjecture can be viewed as the dual of Conjecture \ref{VIII-1}. Note that the series $\sum_{k=0}^\infty\f{(40k+27}{(-6)^k}T_k(4,1)T_k(1,-1)^2$ diverges.
\end{remark}

\begin{conjecture}\label{VIII-2} {\rm (i)} For each $n\in\Z^+$, we have
\begin{equation}\label{8-2-n}\f1{n10^{n-1}}\sum_{k=0}^{n-1}(1435k+113)
3240^{n-1-k}T_k(7,1)T_k(10,10)^2\in\Z^+.
\end{equation}

{\rm (ii)} Let $p>3$ be a prime. Then
\begin{equation}\label{8-2-p}\begin{aligned}&\sum_{k=0}^{p-1}\f{1435k+113}{3240^k}T_k(7,1)T_k(10,10)^2
\\\eq&\f p9\l(2420\l(\f{-5}p\r)+105\l(\f5p\r)-1508\r)
\pmod{p^2}.\end{aligned}\end{equation}
If $p\eq1,9\pmod{20}$, then
\begin{equation}\label{8-2-pn}
\sum_{k=0}^{pn-1}\f{1435k+113}{3240^k}T_k(7,1)T_k(10,10)^2
-p\sum_{k=0}^{n-1}\f{1435k+113}{3240^k}T_k(7,1)T_k(10,10)^2
\end{equation}
divided by $(pn)^2$ is a $p$-adic integer for each $n\in\Z^+$.

{\rm (iii)} Let $p>5$ be a prime. Then
\begin{equation}\label{8-2-q}
\begin{aligned}&\sum_{k=0}^{p-1}\f{T_k(7,1)T_k(10,10)^2}{3240^k}
\\\eq&\begin{cases}4x^2-2p\pmod{p^2}&\t{if}\ p\eq1,4\pmod{15}\ \&\ p=x^2+15y^2\ (x,y\in\Z),
\\12x^2-2p\pmod{p^2}&\t{if}\ p\eq2,8\pmod{15}\ \&\ p=3x^2+5y^2\ (x,y\in\Z),
\\0\pmod{p^2}&\t{if}\ (\f{-15}p)=-1.\end{cases}
\end{aligned}
\end{equation}
\end{conjecture}

\begin{remark}\label{Rem-8-2}\rm The imaginary quadratic field $\Q(\sqrt{-15})$ has class number two.
\end{remark}

\begin{conjecture}\label{VIII-2-d} {\rm (i)} For each $n\in\Z^+$, we have
\begin{equation}\label{8-2-d-n}\f3{2n10^{n-1}}\sum_{k=0}^{n-1}(1435k+1322)
50^{n-1-k}T_k(7,1)T_k(10,10)^2\in\Z^+.
\end{equation}

{\rm (ii)} Let $p>5$ be a prime. Then
\begin{equation}\label{8-2-d-p}
\begin{aligned}&\sum_{k=0}^{p-1}\f{1435k+1322}{50^k}T_k(7,1)T_k(10,10)^2
\\\eq&\f p3\l(3432\l(\f{5}p\r)+968\l(\f{-1}p\r)-434\r)
\pmod{p^2}.\end{aligned}\end{equation}
If $p\eq1,9\pmod{20}$, then
\begin{equation}\label{8-2-d-pn}
\begin{aligned}&\sum_{k=0}^{pn-1}\f{1435k+1322}{50^k}T_k(7,1)T_k(10,10)^2
-p\sum_{k=0}^{n-1}\f{1435k+1322}{50^k}T_k(7,1)T_k(10,10)^2
\end{aligned}
\end{equation}
divided by $(pn)^2$ is a $p$-adic integer for each $n\in\Z^+$.

{\rm (iii)} Let $p>5$ be a prime. Then
\begin{equation}\label{8-2-d-q}
\begin{aligned}&\sum_{k=0}^{p-1}\f{T_k(7,1)T_k(10,10)^2}{50^k}
\\\eq&\begin{cases}4x^2-2p\pmod{p^2}&\t{if}\ p\eq1,4\pmod{15}\ \&\ p=x^2+15y^2\ (x,y\in\Z),
\\2p-12x^2\pmod{p^2}&\t{if}\ p\eq2,8\pmod{15}\ \&\ p=3x^2+5y^2\ (x,y\in\Z),
\\0\pmod{p^2}&\t{if}\ (\f{-15}p)=-1.\end{cases}
\end{aligned}
\end{equation}
\end{conjecture}

\begin{remark}\label{Rem-8-2-d}\rm This conjecture can be viewed as the dual of Conjecture \ref{VIII-2}. Note that
the series $$\sum_{k=0}^\infty\f{1435k+1322}{50^k}T_k(7,1)T_k(10,10)^2$$ diverges.
\end{remark}

\begin{conjecture}\label{VIII-3} {\rm (i)} For each $n\in\Z^+$, we have
\begin{equation}\label{8-3-n}\f1{n5^{n-1}}\sum_{k=0}^{n-1}(840k+197)(-1)^k
2430^{n-1-k}T_k(8,1)T_k(5,-5)^2\in\Z^+.
\end{equation}

{\rm (ii)} Let $p>3$ be a prime. Then
\begin{equation}\label{8-3-p}\begin{aligned}&\sum_{k=0}^{p-1}\f{840k+197}{(-2430)^k}T_k(8,1)T_k(5,-5)^2
\eq p\l(140\l(\f{-15}p\r)+5\l(\f{15}p\r)+52\r)
\pmod{p^2}.\end{aligned}\end{equation}
If $(\f {-1}p)=(\f {15}p)=1$, then
\begin{equation}\label{8-3-pn}
\sum_{k=0}^{pn-1}\f{840k+197}{(-2430)^k}T_k(8,1)T_k(5,-5)^2
-p\sum_{k=0}^{n-1}\f{840k+197}{(-2430)^k}T_k(8,1)T_k(5,-5)^2
\end{equation}
divided by $(pn)^2$ is an $p$-adic integer for any $n\in\Z^+$.

{\rm (iii)} Let $p>7$ be a prime. Then
\begin{equation}\label{8-3-q}
\begin{aligned}&\sum_{k=0}^{p-1}\f{T_k(8,1)T_k(5,-5)^2}{(-2430)^k}
\\\eq&\begin{cases}4x^2-2p\pmod{p^2}&\t{if}\ (\f{-1}p)=(\f p3)=(\f p5)=(\f p7)=1,\ p=x^2+105y^2,
\\2x^2-2p\pmod{p^2}&\t{if}\ (\f{-1}p)=(\f p7)=1,\ (\f p3)=(\f p5)=-1,\ 2p=x^2+105y^2,
\\12x^2-2p\pmod{p^2}&\t{if}\ (\f{-1}p)=(\f p3)=(\f p5)=(\f p7)=-1,\ p=3x^2+35y^2,
\\6x^2-2p\pmod{p^2}&\t{if}\ (\f{-1}p)=(\f p7)=-1,\ (\f p3)=(\f p5)=1,\ 2p=3x^2+35y^2,
\\2p-20x^2\pmod{p^2}&\t{if}\ (\f{-1}p)=(\f p5)=1,\ (\f p5)=(\f p7)=-1,\ p=5x^2+21y^2,
\\2p-10x^2\pmod{p^2}&\t{if}\ (\f{-1}p)=(\f p3)=1,\ (\f p5)=(\f p7)=-1,\ 2p=5x^2+21y^2,
\\28x^2-2p\pmod{p^2}&\t{if}\ (\f{-1}p)=(\f p5)=-1,\ (\f p3)=(\f p7)=1,\ p=7x^2+15y^2,
\\14x^2-2p\pmod{p^2}&\t{if}\ (\f{-1}p)=(\f p3)=-1,\ (\f p5)=(\f p7)=1,\ 2p=7x^2+15y^2,
\\0\pmod{p^2}&\t{if}\ (\f{-105}p)=-1,\end{cases}
\end{aligned}
\end{equation}
where $x$ and $y$ are integers.
\end{conjecture}

\begin{remark}\label{Rem-8-3}\rm Note that the imaginary quadratic field $\Q(\sqrt{-105})$ has class number $8$.
\end{remark}

\begin{conjecture}\label{VIII-4} {\rm (i)} For each $n\in\Z^+$, we have
\begin{equation}\label{8-4-n}\f1{n}\sum_{k=0}^{n-1}(39480k+7321)(-1)^k
29700^{n-1-k}T_k(14,1)T_k(11,-11)^2\in\Z^+,
\end{equation}
and this number is odd if and only if $n$ is a power of two.

{\rm (ii)} Let $p>5$ be a prime. Then
\begin{equation}\label{8-4-p}
\begin{aligned}&\sum_{k=0}^{p-1}\f{39480k+7321}{(-29700)^k}T_k(14,1)T_k(11,-11)^2
\\\eq& p\l(5738\l(\f{-5}p\r)+70\l(\f3p\r)+1513\r)
\pmod{p^2}.\end{aligned}\end{equation}
If $(\f 3p)=(\f{-5}p)=1$, then
\begin{equation}\label{8-4-pn}\begin{aligned}
&\sum_{k=0}^{pn-1}\f{39480k+7321}{(-29700)^k}T_k(14,1)T_k(11,-11)^2
\\&-p\sum_{k=0}^{n-1}\f{39480k+7321}{(-29700)^k}T_k(14,1)T_k(11,-11)^2
\end{aligned}
\end{equation} divided by $(pn)^2$ is a $p$-adic integer for each $n\in\Z^+$.

{\rm (iii)} Let $p>5$ be a prime with $p\not=11$. Then
\begin{equation}\label{8-4-q}
\begin{aligned}&\sum_{k=0}^{p-1}\f{T_k(14,1)T_k(11,-11)^2}{(-29700)^k}
\\\eq&\begin{cases}4x^2-2p\pmod{p^2}&\t{if}\ (\f{-1}p)=(\f p3)=(\f p5)=(\f p{11})=1,\ p=x^2+165y^2,
\\2x^2-2p\pmod{p^2}&\t{if}\ (\f{-1}p)=(\f p3)=(\f p5)=(\f p{11})=-1,\ 2p=x^2+165y^2,
\\12x^2-2p\pmod{p^2}&\t{if}\ (\f{-1}p)=(\f p5)=-1,\ (\f p3)=(\f p{11})=1,\ p=3x^2+55y^2,
\\6x^2-2p\pmod{p^2}&\t{if}\ (\f{-1}p)=(\f p5)=1,\ (\f p3)=(\f p{11})=-1,\ 2p=3x^2+55y^2,
\\2p-20x^2\pmod{p^2}&\t{if}\ (\f{-1}p)=(\f p{11})=1,\ (\f p3)=(\f p5)=-1,\ p=5x^2+33y^2,
\\2p-10x^2\pmod{p^2}&\t{if}\ (\f{-1}p)=(\f p{11})=-1,\ (\f p3)=(\f p5)=1,\ 2p=5x^2+33y^2,
\\44x^2-2p\pmod{p^2}&\t{if}\ (\f{-1}p)=(\f p3)=-1,\ (\f p5)=(\f p{11})=1,\ p=11x^2+15y^2,
\\22x^2-2p\pmod{p^2}&\t{if}\ (\f{-1}p)=(\f p3)=1,\ (\f p5)=(\f p{11})=-1,\ 2p=11x^2+15y^2,
\\0\pmod{p^2}&\t{if}\ (\f{-165}p)=-1,\end{cases}
\end{aligned}
\end{equation}
where $x$ and $y$ are integers.
\end{conjecture}

\begin{remark}\label{Rem-8-4}\rm Note that the imaginary quadratic field $\Q(\sqrt{-165})$ has class number $8$.
\end{remark}

\section{Congruences related to Theorem \ref{Th1.3}}
 \setcounter{equation}{0}

Conjectures \ref{Conj-S1}--\ref{Conj-S10} below provide congruences related to \eqref{S1}--\eqref{S10}.

\begin{conjecture}\label{Conj-S1} {\rm (i)} For any $n\in\Z^+$, we have
\begin{equation}\label{S1-n}
\f1n\sum_{k=0}^{n-1}(7k+3)S_k(1,-6)24^{n-1-k}\in\Z^+.
\end{equation}

{\rm (ii)} Let $p>3$ be a prime. Then
\begin{equation}\label{S1-p}\sum_{k=0}^{p-1}\f{7k+3}{24^k}S_k(1,-6)\eq\f p2\l(5\l(\f{-2}p\r)+\l(\f 6p\r)\r)\pmod{p^2}.
\end{equation}
If $p\eq1\pmod3$, then
\begin{equation}\label{S1-pn}
\f1{(pn)^2}\(\sum_{k=0}^{pn-1}\f{7k+3}{24^k}S_k(1,-6)-p\l(\f{-2}p\r)\sum_{k=0}^{n-1}\f{7k+3}{24^k}S_k(1,-6)\)
\in\Z_p
\end{equation} for all $n\in\Z^+$.

{\rm (iii)} For any prime $p>3$, we have
\begin{equation}\label{S1-q}
\begin{aligned}&\sum_{k=0}^{p-1}\f{S_k(1,-6)}{24^k}
\\\eq&\begin{cases}(\f p3)(4x^2-2p)\pmod{p^2}&\t{if}\ p\eq1,3\pmod8\ \&\ p=x^2+2y^2\ (x,y\in\Z),
\\0\pmod{p^2}&\t{if}\ p\eq5,7\pmod 8.
\end{cases}\end{aligned}
\end{equation}
\end{conjecture}

\begin{conjecture}\label{Conj-S2} {\rm (i)} For any $n\in\Z^+$, we have
\begin{equation}\label{S2-n}
\f1n\sum_{k=0}^{n-1}(12k+5)S_k(1,7)(-1)^k28^{n-1-k}\in\Z^+,
\end{equation}
and this number is odd if and only if $n$ is a power of two.

{\rm (ii)} Let $p\not=7$ be an odd prime. Then
\begin{equation}\label{S2-p}\sum_{k=0}^{p-1}\f{12k+5}{(-28)^k}S_k(1,7)\eq5p\l(\f p7\r)\pmod{p^2},
\end{equation}
and moreover
\begin{equation}\label{S2-pn}
\f1{(pn)^2}\(\sum_{k=0}^{pn-1}\f{12k+5}{(-28)^k}S_k(1,7)-p\l(\f p7\r)\sum_{k=0}^{n-1}\f{12k+5}{(-28)^k}S_k(1,7)\)\in\Z_p
\end{equation}
for all $n\in\Z^+$.

{\rm (iii)} For any prime $p\not=2,7$, we have
\begin{equation}\label{S2-q}
\begin{aligned}&\sum_{k=0}^{p-1}\f{S_k(1,7)}{(-28)^k}
\\\eq&\begin{cases}4x^2-2p\pmod{p^2}&\t{if}\ (\f{-1}p)=(\f p3)=(\f p7)=1\ \&\ p=x^2+21y^2,
\\2x^2-2p\pmod{p^2}&\t{if}\ (\f{-1}p)=(\f p3)=-1,\ (\f p7)=1\ \&\ 2p=x^2+21y^2,
\\2p-12x^2\pmod{p^2}&\t{if}\ (\f{-1}p)=(\f p7)=-1,\ (\f p3)=1\ \&\ p=3x^2+7y^2,
\\2p-6x^2\pmod{p^2}&\t{if}\ (\f{-1}p)=1,\ (\f p3)=(\f p7)=-1\ \&\ 2p=3x^2+7y^2,
\\0\pmod{p^2}&\t{if}\ (\f{-21}p)=-1,
\end{cases}\end{aligned}
\end{equation}
where $x$ and $y$ are integers.

\end{conjecture}

\begin{conjecture}\label{Conj-S3} {\rm (i)} For any $n\in\Z^+$, we have
\begin{equation}\label{S3-n}
\f1n\sum_{k=0}^{n-1}(84k+29)S_k(1,-20)80^{n-1-k}\in\Z^+,
\end{equation}
and this number is odd if and only if $n$ is a power of two.

{\rm (ii)} Let $p$ be an odd prime with $p\not=5$. Then
\begin{equation}\label{S3-p}\sum_{k=0}^{p-1}\f{84k+29}{80^k}S_k(1,-20)\eq
p\l(2\l(\f 5p\r)+27\l(\f{-15}p\r)\r)\pmod{p^2}.
\end{equation}
If $p\eq1\pmod3$, then
\begin{equation}\label{S3-pn}\f1{(pn)^2}\(\sum_{k=0}^{pn-1}\f{84k+29}{80^k}S_k(1,-20)-
p\l(\f p5\r)\sum_{k=0}^{n-1}\f{84k+29}{80^k}S_k(1,-20)\)\in\Z_p
\end{equation}
for all $n\in\Z^+$.

{\rm (iii)} For any prime $p\not=2,5$, we have
\begin{equation}\label{S3-q}
\begin{aligned}&\sum_{k=0}^{p-1}\f{S_k(1,-20)}{80^k}
\\\eq&\begin{cases}4x^2-2p\pmod{p^2}&\t{if}\ (\f{2}p)=(\f p3)=(\f p5)=1\ \&\ p=x^2+30y^2,
\\8x^2-2p\pmod{p^2}&\t{if}\ (\f{2}p)=1,\ (\f p3)=(\f p5)=-1\ \&\ p=2x^2+15y^2,
\\2p-12x^2\pmod{p^2}&\t{if}\ (\f p3)=1,\ (\f 2p)=(\f p5)=-1\ \&\ p=3x^2+10y^2,
\\20x^2-2p\pmod{p^2}&\t{if}\ (\f{p}5)=1,\ (\f 2p)=(\f p3)=-1\ \&\ p=5x^2+6y^2,
\\0\pmod{p^2}&\t{if}\ (\f{-30}p)=-1,
\end{cases}\end{aligned}
\end{equation}
where $x$ and $y$ are integers.
\end{conjecture}

\begin{conjecture}\label{Conj-S6} {\rm (i)} For any $n\in\Z^+$, we have
\begin{equation}\label{S6-n}\f1n\sum_{k=0}^{n-1}(3k+1)(-1)^k100^{n-1-k}S_k(1,25)\in\Z^+.
\end{equation}

{\rm (ii)} Let $p\not=5$ be an odd prime. Then
\begin{equation}\label{S6-pn}\f1{(pn)^2}\(\sum_{k=0}^{pn-1}\f{3k+1}{(-100)^k}S_k(1,25)-
p\l(\f{-1}p\r)\sum_{k=0}^{n-1}\f{3k+1}{(-100)^k}S_k(1,25)\)\in\Z_p
\end{equation}
for all $n\in\Z^+$.

{\rm (iii)} For any prime $p>3$ with $p\not=11$, we have
\begin{equation}\label{S6-q}
\begin{aligned}&\sum_{k=0}^{p-1}\f{S_k(1,25)}{(-100)^k}
\\\eq&\begin{cases}4x^2-2p\pmod{p^2}&\t{if}\ (\f{-1}p)=(\f p3)=(\f p{11})=1\ \&\ p=x^2+33y^2,
\\2x^2-2p\pmod{p^2}&\t{if}\ (\f{-1}{p})=1,\ (\f p3)=(\f p{11})=-1\ \&\ 2p=x^2+33y^2,
\\2p-12x^2\pmod{p^2}&\t{if}\ (\f p{11})=1,\ (\f {-1}p)=(\f p3)=-1\ \&\ p=3x^2+11y^2,
\\2p-6x^2\pmod{p^2}&\t{if}\ (\f{p}3)=1,\ (\f {-1}p)=(\f p{11})=-1\ \&\ 2p=3x^2+11y^2,
\\0\pmod{p^2}&\t{if}\ (\f{-33}p)=-1,
\end{cases}\end{aligned}
\end{equation}
where $x$ and $y$ are integers.
\end{conjecture}

\begin{conjecture}\label{Conj-S4} {\rm (i)} For any $n\in\Z^+$, we have
\begin{equation}\label{S4-n}
\f1n\sum_{k=0}^{n-1}(228k+67)S_k(1,-56)224^{n-1-k}\in\Z^+,
\end{equation}
and this number is odd if and only if $n$ is a power of two.

{\rm (ii)} Let $p$ be an odd prime with $p\not=7$. Then
\begin{equation}\label{S4-p}\sum_{k=0}^{p-1}\f{228k+67}{224^k}S_k(1,-56)\eq
p\l(65\l(\f{-7}p\r)+2\l(\f {14}p\r)\r)\pmod{p^2}.
\end{equation}
If $p\eq1,3\pmod 8$, then
\begin{equation}\label{S4-pn}\f1{(pn)^2}\(\sum_{k=0}^{pn-1}\f{228k+67}{224^k}S_k(1,-56)-
p\l(\f p7\r)\sum_{k=0}^{n-1}\f{228k+67}{224^k}S_k(1,-56)\)\in\Z_p
\end{equation}
for all $n\in\Z^+$.

{\rm (iii)} For any prime $p\not=2,7$, we have
\begin{equation}\label{S4-q}
\begin{aligned}&\sum_{k=0}^{p-1}\f{S_k(1,-56)}{224^k}
\\\eq&\begin{cases}4x^2-2p\pmod{p^2}&\t{if}\ (\f{-2}p)=(\f p3)=(\f p7)=1\ \&\ p=x^2+42y^2,
\\8x^2-2p\pmod{p^2}&\t{if}\ (\f{p}7)=1,\ (\f {-2}p)=(\f p3)=-1\ \&\ p=2x^2+21y^2,
\\2p-12x^2\pmod{p^2}&\t{if}\ (\f {-2}p)=1,\ (\f p3)=(\f p7)=-1\ \&\ p=3x^2+14y^2,
\\2p-24x^2\pmod{p^2}&\t{if}\ (\f{p}5)=1,\ (\f 2p)=(\f p3)=-1\ \&\ p=6x^2+7y^2,
\\0\pmod{p^2}&\t{if}\ (\f{-42}p)=-1,
\end{cases}\end{aligned}
\end{equation}
where $x$ and $y$ are integers.
\end{conjecture}

\begin{conjecture}\label{Conj-S169} {\rm (i)} For any $n\in\Z^+$, we have
\begin{equation}\label{S169-n}\f1n\sum_{k=0}^{n-1}(399k+101)(-1)^k676^{n-1-k}S_k(1,169)\in\Z^+.
\end{equation}

{\rm (ii)} Let $p\not=13$ be an odd prime. Then
\begin{equation}\label{S169-pn}\sum_{k=0}^{pn-1}\f{399k+101}{(-676)^k}S_k(1,169)-
p\l(\f{-1}p\r)\sum_{k=0}^{n-1}\f{399k+101}{(-676)^k}S_k(1,169)
\end{equation}
divided by $(pn)^2$ is a $p$-adic integer
for any $n\in\Z^+$.

{\rm (iii)} For any prime $p>3$ with $p\not=19$, we have
\begin{equation}\label{S169-q}
\begin{aligned}&\sum_{k=0}^{p-1}\f{S_k(1,169)}{(-676)^k}
\\\eq&\begin{cases}4x^2-2p\pmod{p^2}&\t{if}\ (\f{-1}p)=(\f p3)=(\f p{19})=1\ \&\ p=x^2+57y^2,
\\2x^2-2p\pmod{p^2}&\t{if}\ (\f{-1}{p})=1,\ (\f p3)=(\f p{19})=-1\ \&\ 2p=x^2+57y^2,
\\2p-12x^2\pmod{p^2}&\t{if}\ (\f p{3})=1,\ (\f {-1}p)=(\f p{19})=-1\ \&\ p=3x^2+19y^2,
\\2p-6x^2\pmod{p^2}&\t{if}\ (\f{p}{19})=1,\ (\f {-1}p)=(\f p{3})=-1\ \&\ 2p=3x^2+19y^2,
\\0\pmod{p^2}&\t{if}\ (\f{-57}p)=-1,
\end{cases}\end{aligned}
\end{equation}
where $x$ and $y$ are integers.
\end{conjecture}

\begin{conjecture}\label{Conj-S5} {\rm (i)} For any $n\in\Z^+$, we have
\begin{equation}\label{S5-n}
\f1n\sum_{k=0}^{n-1}(2604k+563)S_k(1,-650)2600^{n-1-k}\in\Z^+,
\end{equation}
and this number is odd if and only if $n\in\{2^a:\ a\in\N\}$.

{\rm (ii)} Let $p$ be an odd prime with $p\not=5,13$. Then
\begin{equation}\label{S5-p}\sum_{k=0}^{p-1}\f{2604k+563}{2600^k}S_k(1,-650)\eq
p\l(561\l(\f{-39}p\r)+2\l(\f {26}p\r)\r)\pmod{p^2}.
\end{equation}
If $(\f{-6}p)=1$, then
\begin{equation}\label{S5-pn}\sum_{k=0}^{pn-1}\f{2604k+563}{2600^k}S_k(1,-650)-
p\l(\f{26}p\r)\sum_{k=0}^{n-1}\f{2604k+563}{2600^k}S_k(1,-650)
\end{equation}
divided by $(pn)^2$ is a $p$-adic integer
for any $n\in\Z^+$.

{\rm (iii)} For any odd prime $p\not=5,13$, we have
\begin{equation}\label{S5-q}
\begin{aligned}&\sum_{k=0}^{p-1}\f{S_k(1,-650)}{2600^k}
\\\eq&\begin{cases}4x^2-2p\pmod{p^2}&\t{if}\ (\f{2}p)=(\f p3)=(\f p{13})=1\ \&\ p=x^2+78y^2,
\\8x^2-2p\pmod{p^2}&\t{if}\ (\f2{p})=1,\ (\f p3)=(\f p{13})=-1\ \&\ p=2x^2+39y^2,
\\2p-12x^2\pmod{p^2}&\t{if}\ (\f p{13})=1,\ (\f 2p)=(\f p3)=-1\ \&\ p=3x^2+26y^2,
\\2p-24x^2\pmod{p^2}&\t{if}\ (\f{p}3)=1,\ (\f 2p)=(\f p{13})=-1\ \&\ p=6x^2+13y^2,
\\0\pmod{p^2}&\t{if}\ (\f{-78}p)=-1,
\end{cases}\end{aligned}
\end{equation}
where $x$ and $y$ are integers.
\end{conjecture}

\begin{conjecture}\label{Conj-S1519} {\rm (i)} For any $n\in\Z^+$, we have
\begin{equation}\label{S1519-n}\f1n\sum_{k=0}^{n-1}(39468k+7817)(-1)^k6076^{n-1-k}S_k(1,1519)\in\Z^+,
\end{equation}
and this number is odd if and only if $n\in\{2^a:\ a\in\N\}$.

{\rm (ii)} Let $p\not=7,31$ be an odd prime. Then
\begin{equation}\label{S1519-pn}\sum_{k=0}^{pn-1}\f{39468k+7817}{(-6076)^k}S_k(1,1519)-
p\l(\f{-31}p\r)\sum_{k=0}^{n-1}\f{39468k+7817}{(-6076)^k}S_k(1,1519)
\end{equation}
divided by $(pn)^2$ is a $p$-adic integer
for any $n\in\Z^+$.

{\rm (iii)} For any prime $p>3$ with $p\not=7,31$, we have
\begin{equation}\label{S1519-q}
\begin{aligned}&\sum_{k=0}^{p-1}\f{S_k(1,1519)}{(-6076)^k}
\\\eq&\begin{cases}4x^2-2p\pmod{p^2}&\t{if}\ (\f{-1}p)=(\f p3)=(\f p{31})=1\ \&\ p=x^2+93y^2,
\\2x^2-2p\pmod{p^2}&\t{if}\ (\f{p}{31})=1,\ (\f {-1}p)=(\f p{3})=-1\ \&\ 2p=x^2+93y^2,
\\2p-12x^2\pmod{p^2}&\t{if}\ (\f p{3})=1,\ (\f {-1}p)=(\f p{31})=-1\ \&\ p=3x^2+31y^2,
\\2p-6x^2\pmod{p^2}&\t{if}\ (\f{-1}p)=1,\ (\f {p}3)=(\f p{31})=-1\ \&\ 2p=3x^2+31y^2,
\\0\pmod{p^2}&\t{if}\ (\f{-93}p)=-1,
\end{cases}\end{aligned}
\end{equation}
where $x$ and $y$ are integers.
\end{conjecture}

\begin{conjecture}\label{Conj-S2450} {\rm (i)} For any $n\in\Z^+$, we have
\begin{equation}\label{S2450-n}\f1n\sum_{k=0}^{n-1}(41667k+7879)9800^{n-1-k}S_k(1,-2450)\in\Z^+.
\end{equation}

{\rm (ii)} Let $p\not=5,7$ be an odd prime. Then
\begin{equation}\label{S2450-p}
\sum_{k=0}^{p-1}\f{41667k+7879}{9800^k}S_k(1,-2450)
\eq\f p2\l(15741\l(\f{-6}p\r)+17\l(\f 2p\r)\r)\pmod{p^2}.
\end{equation}
If $p\eq1\pmod3$, then
\begin{equation}\label{S2450-pn}\sum_{k=0}^{pn-1}\f{41667k+7879}{9800^k}S_k(1,-2450)-
p\l(\f{2}p\r)\sum_{k=0}^{n-1}\f{41667k+7879}{9800^k}S_k(1,-2450)
\end{equation}
divided by $(pn)^2$ is a $p$-adic integer for any $n\in\Z^+$.

{\rm (iii)} For any prime $p>7$ with $p\not=17$, we have
\begin{equation}\label{S2450-q}
\begin{aligned}&\sum_{k=0}^{p-1}\f{S_k(1,-2450)}{9800^k}
\\\eq&\begin{cases}4x^2-2p\pmod{p^2}&\t{if}\ (\f{2}p)=(\f p3)=(\f p{17})=1\ \&\ p=x^2+102y^2,
\\8x^2-2p\pmod{p^2}&\t{if}\ (\f{p}{17})=1,\ (\f {2}p)=(\f p{3})=-1\ \&\ p=2x^2+51y^2,
\\2p-12x^2\pmod{p^2}&\t{if}\ (\f p{3})=1,\ (\f {2}p)=(\f p{17})=-1\ \&\ p=3x^2+34y^2,
\\2p-24x^2\pmod{p^2}&\t{if}\ (\f{2}p)=1,\ (\f {p}3)=(\f p{17})=-1\ \&\ p=6x^2+17y^2,
\\0\pmod{p^2}&\t{if}\ (\f{-102}p)=-1,
\end{cases}\end{aligned}
\end{equation}
where $x$ and $y$ are integers.
\end{conjecture}

\begin{conjecture}\label{Conj-S530} {\rm (i)} For any $n\in\Z^+$, we have
\begin{equation}\label{S530-n}\f1n\sum_{k=0}^{n-1}(74613k+10711)(-1)^k530^{2(n-1-k)}S_k(1,265^2)\in\Z^+.
\end{equation}

{\rm (ii)} Let $p\not=5,53$ be an odd prime. Then
\begin{equation}\label{S530-pn}\sum_{k=0}^{pn-1}\f{74613k+10711}{(-530^2)^k}S_k(1,265^2)-
p\l(\f{-1}p\r)\sum_{k=0}^{n-1}\f{74613k+10711}{(-530^2)^k}S_k(1,265^2)
\end{equation}
divided by $(pn)^2$ is a $p$-adic integer for any $n\in\Z^+$.

{\rm (iii)} For any prime $p>5$ with $p\not=59$, we have
\begin{equation}\label{S530-q}
\begin{aligned}&\sum_{k=0}^{p-1}\f{S_k(1,265^2)}{(-530^2)^k}
\\\eq&\begin{cases}4x^2-2p\pmod{p^2}&\t{if}\ (\f{-1}p)=(\f p3)=(\f p{59})=1\ \&\ p=x^2+177y^2,
\\2x^2-2p\pmod{p^2}&\t{if}\ (\f{-1}{p})=1,\ (\f {p}3)=(\f p{59})=-1\ \&\ 2p=x^2+177y^2,
\\2p-12x^2\pmod{p^2}&\t{if}\ (\f p{59})=1,\ (\f {-1}p)=(\f p{3})=-1\ \&\ p=3x^2+59y^2,
\\2p-6x^2\pmod{p^2}&\t{if}\ (\f{p}3)=1,\ (\f {-1}p)=(\f p{59})=-1\ \&\ 2p=3x^2+59y^2,
\\0\pmod{p^2}&\t{if}\ (\f{-177}p)=-1,
\end{cases}\end{aligned}
\end{equation}
where $x$ and $y$ are integers.
\end{conjecture}

\begin{conjecture}\label{Conj-S7} For any odd prime $p$,
\begin{equation}\label{S7-q}
\begin{aligned}&\sum_{k=0}^{p-1}\f{S_k}{(-4)^k}
\\\eq&\begin{cases}4x^2-2p\pmod{p^2}&\t{if}\ 12\mid p-1\ \&\ p=x^2+y^2\ (x,y\in\Z\ \&\ 3\nmid x),
\\4xy\pmod{p^2}&\t{if}\ 12\mid p-5\ \&\ p=x^2+y^2\ (x,y\in\Z\ \&\ 3\mid x-y),
\\0\pmod{p^2}&\t{if}\ p\eq3\pmod4.
\end{cases}\end{aligned}
\end{equation}
Also, for any prime $p\eq1\pmod4$ we have
\begin{equation}\label{S7-p}
\sum_{k=0}^{p-1}(8k+5)\f{S_k}{(-4)^k}\eq4p\pmod{p^2}.
\end{equation}
\end{conjecture}

\begin{conjecture}\label{Conj-S8} {\rm (i)} For any $n\in\Z^+$, we have
\begin{equation}\label{S8-n}
\f1n\sum_{k=0}^{n-1}(4k+3)4^{n-1-k}S_k(1,-1)\in\Z,
\end{equation}
and this number is odd if and only if $n$ is a power of two.

{\rm (ii)} For any odd prime $p$ and positive integer $n$, we have
\begin{equation}\label{S8-pn}\f1{(pn)^2}\(\sum_{k=0}^{pn-1}\f{4k+3}{4^k}S_k(1,-1)-
p\sum_{k=0}^{n-1}\f{4k+3}{4^k}S_k(1,-1)\)\in\Z_p.
\end{equation}

{\rm (iii)} Let $p$ be an odd prime. Then
\begin{equation}\label{S8-q}
\begin{aligned}&\sum_{k=0}^{p-1}\f{S_k(1,-1)}{4^k}
\\\eq&\begin{cases}4x^2-2p\pmod{p^2}&\t{if}\ p\eq1,9\pmod{20}\ \&\ p=x^2+5y^2\ (x,y\in\Z),
\\2x^2-2p\pmod{p^2}&\t{if}\ p\eq3,7\pmod{20}\ \&\ 2p=x^2+5y^2\ (x,y\in\Z),
\\0\pmod{p^2}&\t{if}\ (\f{-5}p)=-1.
\end{cases}\end{aligned}
\end{equation}

\end{conjecture}

\begin{conjecture}\label{Conj-S9} {\rm (i)} For any $n\in\Z^+$, we have
\begin{equation}\label{S9-n}\f1n\sum_{k=0}^{n-1}(33k+25)S_k(1,-6)(-6)^{n-1-k}\in\Z,
\end{equation}
and this number is odd if and only if $n$ is a power of two.

{\rm (ii)} Let $p>3$ be a prime. Then
\begin{equation}\label{S9-p}\sum_{k=0}^{p-1}\f{33k+25}{(-6)^k}S_k(1,-6)\eq p\l(35-10\l(\f 3p\r)\r)\pmod{p^2}.
\end{equation}
If $p\eq\pm1\pmod{12}$, then
\begin{equation}\label{S9-pn}\f1{(pn)^2}\(\sum_{k=0}^{pn-1}\f{33k+25}{(-6)^k}S_k(1,-6)
-p\sum_{k=0}^{n-1}\f{33k+25}{(-6)^k}S_k(1,-6)\)\in\Z_p
\end{equation}
for all $n\in\Z^+$.

{\rm (iii)} For any prime $p>3$, we have
\begin{equation}\label{S9-q}
\sum_{k=0}^{p-1}\f{S_k(1,-6)}{(-6)^k}\eq\begin{cases}(\f{-1}p)(4x^2-2p)\pmod{p^2}&\t{if}\ p=x^2+3y^2\ (x,y\in\Z),\\0\pmod{p^2}&\t{if}\ p\eq2\pmod3.\end{cases}
\end{equation}

\end{conjecture}

\begin{conjecture}\label{Conj-S10} {\rm (i)} For any $n\in\Z^+$, we have
\begin{equation}\label{S10-n}
n\ \ \bigg| \ \sum_{k=0}^{n-1}(18k+13)S_k(2,9)8^{n-1-k}.
\end{equation}

{\rm (ii)} Let $p$ be an odd prime. Then
\begin{equation}\label{S10-p}
\sum_{k=0}^{p-1}\f{18k+13}{8^k}S_k(2,9)
\eq p\l(1+12\l(\f p3\r)\r)\pmod{p^2}.
\end{equation}
If $p\eq1\pmod3$, then
\begin{equation}\label{S10-pn}
\f1{(pn)^2}\(\sum_{k=0}^{pn-1}\f{18k+13}{8^k}S_k(2,9)-p\sum_{k=0}^{n-1}\f{18k+13}{8^k}S_k(2,9)\)\in\Z_p
\end{equation}
for all $n\in\Z^+$.

{\rm (iii)} For any prime $p>3$, we have
\begin{equation}\label{S10-q}
\begin{aligned}&\sum_{k=0}^{p-1}\f{S_k(1,-2)}{8^k}\eq\l(\f p3\r)\sum_{k=0}^{p-1}\f{S_k(2,9)}{8^k}
\\\eq&\begin{cases}4x^2-2p\pmod{p^2}&\t{if}\ p\eq1,7\pmod{24}\ \&\ p=x^2+6y^2\ \\8x^2-2p\pmod{p^2}&\t{if}\ p\eq5,11\pmod{24}\ \&\ p=2x^2+3y^2,\\0\pmod{p^2}&\t{if}\ (\f{-6}p)=-1,\end{cases}
\end{aligned}
\end{equation}
where $x$ and $y$ are integers.
\end{conjecture}

\begin{conjecture}\label{Conj-S11}
Let $p$ be an odd prime with $p\not=5$. Then
\begin{equation}\label{S11-q}
\sum_{k=0}^{p-1}\f{S_k(3,1)}{4^k}
\eq\begin{cases}4x^2-2p\pmod{p^2}&\t{if}\ p\eq1,9\pmod{20}\ \&\ p=x^2+5y^2,\ \\2x^2-2p\pmod{p^2}&\t{if}\ p\eq3,7\pmod{20}\ \&\ 2p=x^2+5y^2,\\0\pmod{p^2}&\t{if}\ p\eq11,13,17,19\pmod{20},\end{cases}
\end{equation}
where $x$ and $y$ are integers. If $(\f{-5}p)=1$, then
$$\sum_{k=0}^{p-1}\f{40k+29}{4^k}S_k(3,1)\eq 18p\pmod{p^2}.$$
\end{conjecture}
\begin{remark}\label{Rem-S11}\rm We also have some similar conjectures involving
\begin{gather*}\sum_{k=0}^{p-1}\f{S_k(5,4)}{4^k},\ \sum_{k=0}^{p-1}\f{S_k(4,-5)}{4^k},
\ \sum_{k=0}^{p-1}\f{S_k(7,6)}{6^k},
\\ \sum_{k=0}^{p-1}\f{S_k(10,-2)}{32^k},
\ \sum_{k=0}^{p-1}\f{S_k(14,9)}{72^k},\ \sum_{k=0}^{p-1}\f{S_k(19,9)}{36^k}
\end{gather*}
modulo $p^2$, where $p$ is a prime greater than $3$.
\end{remark}

Motivated by Theorem \ref{Th2.1}, we pose the following general conjecture.

\begin{conjecture} For any odd prime $p$ and integer $m\not\eq0\pmod p$, we have
\begin{equation}\label{Sp}\sum_{k=0}^{p-1}\f{S_k(4,-m)}{m^k}\eq\sum_{k=0}^{p-1}\f{\bi{2k}kf_k}{m^k}\pmod{p^2}.
\end{equation}
and
\begin{equation}\label{kSp}\f{m+16}2\sum_{k=0}^{p-1}\f{kS_k(4,-m)}{m^k}
-\sum_{k=0}^{p-1}((m+4)k-4)\f{\bi{2k}kf_k}{m^k}\eq4p\l(\f mp\r)\pmod{p^2}.
\end{equation}
\end{conjecture}
\begin{remark} We have checked this conjecture via {\tt Mathematica}.
In view of the proof of Theorem \ref{Th2.1}, both \eqref{Sp} and \eqref{kSp}
hold modulo $p$.
\end{remark}

\section{Series for $1/\pi$ involving $T_n(b,c)$ and $Z_n=\sum_{k=0}^n\bi nk\bi{2k}k\bi{2(n-k)}{n-k}$}
 \setcounter{equation}{0}

The numbers
$$Z_n:=\sum_{k=0}^n\bi nk\bi{2k}k\bi{2(n-k)}{n-k}\ \ (n=0,1,2,\ldots)$$
were first introduced by D. Zagier in his paper \cite{Zag} the preprint of which was released in 2002.
Thus we name such numbers as {\it Zagier numbers}. As pointed out by the author \cite[Remark 4.3]{S14d},
for any $n\in\N$ the number $2^nZ_n$ coincides with the so-called CLF (Catalan-Larcombe-French) number
$${\mathcal P}_n:=2^n\sum_{k=1}^{\lfloor n/2\rfloor}\bi n{2k}\bi{2k}k^24^{n-2k}=\sum_{k=0}^n\f{\bi{2k}k^2\bi{2(n-k)}{n-k}^2}{\bi nk}.$$

Let $p$ be an odd prime. For any $k=0,\ldots,p-1$, we have
$${\mathcal P}_k\eq\l(\f{-1}p\r)128^k{\mathcal P}_{p-1-k}\pmod p$$
by F. Jarvis and H.A. Verrill \cite[Corollary 2.2]{JV}, and hence
$$Z_k=\f{{\mathcal P}_k}{2^k}\eq\l(\f{-1}p\r)64^k(2^{p-1-k}Z_{p-1-k})\eq\l(\f{-1}p\r)32^kZ_{p-1-k}\pmod p.$$
Combining this with Remark \ref{Rem-Dual}(ii), we see that
\begin{align*}\sum_{k=0}^{p-1}\f{Z_kT_k(b,c)}{m^k}\eq&
\l(\f{4c-b^2}p\r)\sum_{k=0}^{p-1}\l(\f{32(b^2-4c)}m\r)^kZ_{p-1-k}T_{p-1-k}(b,c)
\\\eq&\l(\f{4c-b^2}p\r)\sum_{k=0}^{p-1}\f{Z_kT_k(b,c)}{(32(b^2-4c)/m)^k}\pmod p
\end{align*}
for any $b,c,m\in\Z$ with $p\nmid (b^2-4c)m$.

J. Wan and Zudilin \cite{WZ} obtained the following irrational series for $1/\pi$ involving the Legendre polynomials and the Zagier numbers:
$$\sum_{k=0}^\infty(15k+4-2\sqrt6)Z_kP_k\l(\f{24-\sqrt6}{15\sqrt2}\r)\l(\f{4-\sqrt6}{10\sqrt3}\r)^k
=\f{6}{\pi}(7+3\sqrt6).$$
Via our congruence approach (including Conjecture \ref{Conj-Dual}), we find 24 rational series for $1/\pi$
involving $T_n(b,c)$ and the Zagier numbers. Theorem 1 of \cite{WZ} might be helpful to solve some of them.

\begin{conjecture}\label{Conj-Z} We have the following identities for $1/\pi$.
\begin{align}\sum_{k=1}^\infty\f{5k+1}{32^k}T_kZ_k&=\f{8(2+\sqrt5)}{3\pi},
\\\sum_{k=0}^\infty\f{21k+5}{(-252)^k}T_k(1,16)Z_k&=\f{6\sqrt7}{\pi},
\\\sum_{k=0}^\infty\f{3k+1}{36^k}T_k(1,-2)Z_k&=\f{3}{\pi},
\end{align}\begin{align}\sum_{k=0}^\infty\f{k}{192^k}T_k(14,1)Z_k&=\f{8}{3\pi},
\\\sum_{k=0}^\infty\f{30k+11}{(-192)^k}T_k(14,1)Z_k&=\f{12}{\pi},
\\\sum_{k=0}^\infty\f{15k+1}{480^k}T_k(22,1)Z_k&=\f{6\sqrt{10}}{\pi},
\\\sum_{k=0}^\infty\f{7k+2}{(-672)^k}T_k(26,1)Z_k&=\f{2\sqrt{21}}{3\pi},
\\\sum_{k=0}^\infty\f{21k+2}{1152^k}T_k(34,1)Z_k&=\f{18}{\pi},
\\\sum_{k=0}^\infty\f{30k-7}{640^k}T_k(62,1)Z_k&=\f{160}{\pi},
\\\sum_{k=0}^\infty\f{195k+34}{(-9600)^k}T_k(98,1)Z_k&=\f{80}{\pi},
\\\sum_{k=0}^\infty\f{195k+22}{11232^k}T_k(106,1)Z_k&=\f{27\sqrt{13}}{\pi},
\\\sum_{k=0}^\infty\f{42k+17}{(-1440)^k}T_k(142,1)Z_k&=\f{33}{\sqrt5\,\pi},
\\\sum_{k=0}^\infty\f{2k-1}{1792^k}T_k(194,1)Z_k&=\f{56}{3\pi},
\\\sum_{k=0}^\infty\f{1785k+254}{(-37632)^k}T_k(194,1)Z_k&=\f{672}{\pi},\\
\sum_{k=0}^\infty\f{210k+23}{40800^k}T_k(202,1)Z_k&=\f{15\sqrt{34}}{\pi},
\\\sum_{k=0}^\infty\f{210k-1}{4608^k}T_k(254,1)Z_k&=\f{288}{\pi},
\\\sum_{k=0}^\infty\f{21k-5}{5600^k}T_k(502,1)Z_k&=\f{105}{\sqrt2\,\pi},
\\\sum_{k=0}^\infty\f{7410k+1849}{(-36992)^k}T_k(1154,1)Z_k&=\f{2992}{\pi},
\\\sum_{k=0}^\infty\f{1326k+101}{57760^k}T_k(1442,1)Z_k&=\f{2014}{\sqrt5\,\pi},
\\
\sum_{k=0}^\infty\f{78k-131}{20800^k}T_k(2498,1)Z_k&=\f{2600}{\pi},
\\\sum_{k=0}^\infty\f{62985k+11363}{(-394272)^k}T_k(5474,1)Z_k&=\f{7659\sqrt{10}}{\pi},
\end{align}
\begin{align}\sum_{k=0}^\infty\f{358530k+33883}{486720^k}T_k(6082,1)Z_k&=\f{176280}{\pi},
\\\sum_{k=0}^\infty\f{510k-1523}{78400^k}T_k(9602,1)Z_k&=\f{33320}{\pi},
\\\sum_{k=0}^\infty\f{570k-457}{93600^k}T_k(10402,1)Z_k&=\f{1590\sqrt{13}}{\pi}.
\end{align}
\end{conjecture}

Below we present some conjectures on congruences related to $(5.1)$, $(5.2)$, $(5.4)$ and $(5.9)$.

\begin{conjecture}\label{Conj-Z1} {\rm (i)} For any $n\in\Z^+$, we have
\begin{equation}\label{Z1-n}
n\ \bigg|\ \sum_{k=0}^{n-1}(5k+1)T_kZ_k32^{n-1-k}.
\end{equation}

{\rm (ii)} Let $p$ be an odd prime with $p\not=5$. Then
\begin{equation}\label{Z1-p}\sum_{k=0}^{p-1}\f{5k+1}{32^k}T_kZ_k\eq \f p3\l(5\l(\f{-5}p\r)-2\l(\f{-1}p\r)\r)\pmod{p^2}.
\end{equation}
If $p\eq\pm1\pmod{5}$, then
\begin{equation}\label{Z1-pn}
\f1{(pn)^2}\(\sum_{k=0}^{pn-1}\f{5k+1}{32^k}T_kZ_k-p\l(\f{-1}p\r)\sum_{k=0}^{n-1}\f{5k+1}{32^k}T_kZ_k\)
\in\Z_p
\end{equation}
for all $n\in\Z^+$.

{\rm (iii)} For any prime $p>5$, we have
\begin{equation}\label{Z1-q}\begin{aligned}&\l(\f{-1}p\r)\sum_{k=0}^{p-1}\f{T_kZ_k}{32^k}
\\\eq&\begin{cases}4x^2-2p\pmod{p^2}&\t{if}\ p\eq1,4\pmod{15}\ \&\ p=x^2+15y^2\ (x,y\in\Z),
\\12x^2-2p\pmod{p^2}&\t{if}\ p\eq2,8\pmod{15}\ \&\ p=3x^2+5y^2\ (x,y\in\Z),
\\0\pmod{p^2}&\t{if}\ (\f{-15}p)=-1.
\end{cases}
\end{aligned}\end{equation}
\end{conjecture}

\begin{conjecture}\label{Conj-Z2} {\rm (i)} For any $n\in\Z^+$, we have
\begin{equation}\label{Z2-n}
\f1n\sum_{k=0}^{n-1}(-1)^k(21k+5)T_k(1,16)Z_k252^{n-1-k}\in\Z^+.
\end{equation}

{\rm (ii)} Let $p>3$ be a prime with $p\not=7$. Then
\begin{equation}\label{Z2-p}\sum_{k=0}^{p-1}\f{21k+5}{(-252)^k}T_k(1,16)Z_k\eq \f p3\l(16\l(\f{-7}p\r)-\l(\f{-1}p\r)\r)\pmod{p^2}.
\end{equation}
If $(\f 7p)=1$, then
\begin{equation}\label{Z2-pn}
\f1{(pn)^2}\(\sum_{k=0}^{pn-1}\f{21k+5}{(-252)^k}T_k(1,16)Z_k-p\l(\f{-1}p\r)
\sum_{k=0}^{n-1}\f{21k+5}{(-252)^k}T_k(1,16)Z_k\)\in\Z_p
\end{equation}
for all $n\in\Z^+$.

{\rm (iii)} For any prime $p>3$ with $p\not=7$, we have
\begin{equation}\label{Z2-q}\begin{aligned}&\sum_{k=0}^{p-1}\f{T_k(1,16)Z_k}{(-252)^k}
\\\eq&\begin{cases}4x^2-2p\pmod{p^2}&\t{if}\ p\eq1,2,4\pmod{7}\ \&\ p=x^2+7y^2\ (x,y\in\Z),
\\0\pmod{p^2}&\t{if}\ p\eq 3,5,6\pmod 7.
\end{cases}
\end{aligned}\end{equation}
\end{conjecture}

\begin{conjecture}\label{Conj-Z4}
{\rm (i)} For any $n\in\Z^+$, we have
\begin{equation}\label{Z4-n}
n\ \bigg|\ \sum_{k=0}^{n-1}kT_k(14,1)Z_k192^{n-1-k}.
\end{equation}

{\rm (ii)} Let $p>3$ be a prime. Then
\begin{equation}\label{Z4-p}\sum_{k=0}^{p-1}\f{k}{192^k}T_k(14,1)Z_k\eq \f p9\l(\l(\f{-1}p\r)-\l(\f{2}p\r)\r)\pmod{p^2}.
\end{equation}
If $p\eq1,3\pmod8$, then
\begin{equation}\label{Z4-pn}
\f1{(pn)^2}\(\sum_{k=0}^{pn-1}\f{k}{192^k}T_k(14,1)Z_k-p\l(\f{-1}p\r)
\sum_{k=0}^{n-1}\f{k}{192^k}T_k(14,1)\)\in\Z_p
\end{equation}
for all $n\in\Z^+$.

{\rm (iii)} For any prime $p>3$, we have
\begin{equation}\label{Z4-q}\begin{aligned}&\l(\f 3p\r)\sum_{k=0}^{p-1}\f{T_k(14,1)Z_k}{192^k}
\\\eq&\begin{cases}4x^2-2p\pmod{p^2}&\t{if}\ p\eq1,3\pmod{8}\ \&\ p=x^2+2y^2\ (x,y\in\Z),
\\0\pmod{p^2}&\t{if}\ p\eq 5,7\pmod 8.
\end{cases}
\end{aligned}\end{equation}
\end{conjecture}

\begin{conjecture}\label{Conj-Z9} {\rm (i)} For any $n\in\Z^+$, we have
\begin{equation}\label{Z9-n}
n\ \bigg|\ \sum_{k=0}^{n-1}(30k-7)T_k(62,1)Z_k640^{n-1-k}.
\end{equation}

{\rm (ii)} Let $p$ be an odd prime with $p\not=5$. Then
\begin{equation}\label{Z9-p}\sum_{k=0}^{p-1}\f{30k-7}{640^k}T_k(62,1)Z_k\eq p\l(2\l(\f{-1}p\r)-9\l(\f{15}p\r)\r)\pmod{p^2}.
\end{equation}
If $(\f{-15}p)=1$, then
\begin{equation}\label{Z9-pn}
\f1{(pn)^2}\(\sum_{k=0}^{pn-1}\f{30k-7}{640^k}T_k(62,1)Z_k
-p\l(\f{-1}p\r)\sum_{k=0}^{n-1}\f{30k-7}{640^k}T_k(62,1)Z_k\)\in\Z_p
\end{equation}
for all $n\in\Z^+$.

{\rm (iii)} For any prime $p>5$, we have
\begin{equation}\label{Z9-q}\begin{aligned}&\l(\f{-1}p\r)\sum_{k=0}^{p-1}\f{T_k(62,1)Z_k}{640^k}
\\\eq&\begin{cases}4x^2-2p\pmod{p^2}&\t{if}\ (\f 2p)=(\f p3)=(\f p5)=1\ \&\ p=x^2+30y^2,
\\8x^2-2p\pmod{p^2}&\t{if}\ (\f 2p)=1,\ (\f p3)=(\f p5)=-1\ \&\ p=2x^2+15y^2,
\\2p-12x^2\pmod{p^2}&\t{if}\ (\f p3)=1,\ (\f 2p)=(\f p5)=-1\ \&\ p=3x^2+10y^2,
\\20x^2-2p\pmod{p^2}&\t{if}\ (\f p5)=1,\ (\f 2p)=(\f p3)=-1\ \&\ p=5x^2+6y^2,
\\0\pmod{p^2}&\t{if}\ (\f{-30}p)=-1,
\end{cases}
\end{aligned}\end{equation}
where $x$ and $y$ are integers.
\end{conjecture}

\section{Series for $1/\pi$ involving $T_k(b,c)$ and the Franel numbers}
\setcounter{equation}{0}

Sun \cite{S13f,S13e} obtained some supercongruences involving the Franel numbers $f_n=\sum_{k=0}^n\bi nk^3\ (n\in\N)$.
M. Rogers and A. Straub \cite{RS} confirmed the $520$-series for $1/\pi$
involving Franel polynomials conjectured by Sun \cite{S-11}.

Let $p$ be an odd prime. By \cite[Lemma 2.6]{JV}, we have
$f_k\eq(-8)^kf_{p-1-k}\pmod p$ for each $k=0,\ldots,p-1$.
Combining this with Remark \ref{Rem-Dual}(ii), we see that
\begin{align*}\sum_{k=0}^{p-1}\f{f_kT_k(b,c)}{m^k}\eq&
\l(\f{b^2-4c}p\r)\sum_{k=0}^{p-1}\l(\f{-8(b^2-4c)}m\r)^kf_{p-1-k}T_{p-1-k}(b,c)
\\\eq&\l(\f{b^2-4c}p\r)\sum_{k=0}^{p-1}\f{f_kT_k(b,c)}{(8(4c-b^2)/m)^k}\pmod p
\end{align*}
for any $b,c,m\in\Z$ with $p\nmid (b^2-4c)m$.

Wan and Zudilin \cite{WZ} deduced the following irrational series for $1/\pi$ involving the Legendre polynomials and the Franel numbers:
$$\sum_{k=0}^\infty(18k+7-2\sqrt3)f_kP_k\l(\f{1+\sqrt3}{\sqrt6}\r)\l(\f{2-\sqrt3}{2\sqrt6}\r)^k
=\f{27+11\sqrt3}{\sqrt2\,\pi}.$$
 Via our congruence approach (including Conjecture \ref{Conj-Dual}), we find $12$ rational series for $1/\pi$
involving $T_n(b,c)$ and the Franel numbers; Theorem 1 of \cite{WZ} might be helpful to solve some of them.

\begin{conjecture}\label{Conj-F} We have
\begin{align}\sum_{k=0}^\infty\f{3k+1}{(-48)^k}f_kT_k(4,-2)&=\f{4\sqrt2}{3\pi},
\\\sum_{k=0}^\infty\f{99k+23}{(-288)^k}f_kT_k(8,-2)&=\f{39\sqrt2}{\pi},\\
\sum_{k=0}^\infty\f{105k+17}{480^k}f_kT_k(8,1)&=\f{92\sqrt5}{3\pi},
\\\sum_{k=0}^\infty\f{45k-2}{441^k}f_kT_k(47,1)&=\f{483\sqrt5}{4\pi},
\end{align}
\begin{align}\sum_{k=0}^\infty\f{165k+46}{(-2352)^k}f_kT_k(194,1)&=\f{112\sqrt5}{3\pi},
\\\sum_{k=0}^\infty\f{42k+5}{11616^k}f_kT_k(482,1)&=\f{374\sqrt2}{15\pi},
\\
\sum_{k=0}^\infty\f{990k+31}{11200^k}f_kT_k(898,1)&=\f{680\sqrt7}{\pi},
\\\sum_{k=0}^\infty\f{585k+172}{(-13552)^k}f_kT_k(1454,1)&=\f{110\sqrt7}{\pi},
\\\sum_{k=0}^\infty\f{90k+11}{101568^k}f_kT_k(2114,1)&=\f{92\sqrt{15}}{7\pi},
\\\sum_{k=0}^\infty\f{94185k+17014}{(-105984)^k}f_kT_k(2302,1)&=\f{8520\sqrt{23}}{\pi},
\\\sum_{k=0}^\infty\f{5355k+1381}{(-61952)^k}f_kT_k(4354,1)&=\f{968\sqrt{7}}{\pi},
\\\sum_{k=0}^\infty\f{210k+23}{475904^k}f_kT_k(16898,1)&=\f{2912\sqrt{231}}{297\pi}.
\end{align}
\end{conjecture}

We now present a conjecture on congruence related to $(6.3)$.

\begin{conjecture} \label{Conj-f1} {\rm (i)} For any $n\in\Z^+$, we have
\begin{equation}\label{f1-n}\f1n\sum_{k=0}^{n-1}(105k+17)480^{n-1-k}f_kT_k(8,1)\in\Z^+.
\end{equation}

{\rm (ii)} Let $p>5$ be a prime. Then
\begin{equation}\label{f1-p}
\sum_{k=0}^{p-1}\f{105k+17}{480^k}f_kT_k(8,1)
\eq\f p9\l(161\l(\f{-5}p\r)-8\r)\pmod{p^2}.
\end{equation}
If $(\f{-5}p)=1$, then
\begin{equation}\label{f1-pn}
\f1{(pn)^2}\(\sum_{k=0}^{pn-1}\f{105k+17}{480^k}f_kT_k(8,1)
-p\sum_{k=0}^{n-1}\f{105k+17}{480^k}f_kT_k(8,1)\)\in\Z_p
\end{equation}
for all $n\in\Z^+$.

{\rm (iii)} For any prime $p>5$, we have
\begin{equation}\label{f1-q}
\begin{aligned}&\l(\f{-1}p\r)\sum_{k=0}^{p-1}\f{f_kT_k(8,1)}{480^k}
\\\eq&\begin{cases}4x^2-2p\pmod{p^2}&\t{if}\ p\eq1,4\pmod{15}\ \&\ p=x^2+15y^2\ (x,y\in\Z),
\\2p-12x^2\pmod{p^2}&\t{if}\ p\eq2,8\pmod{15}\ \&\ p=3x^2+5y^2\ (x,y\in\Z),
\\0\pmod{p^2}&\t{if}\ (\f{-15}p)=-1.\end{cases}
\end{aligned}
\end{equation}
\end{conjecture}
\begin{remark}\rm This conjecture was formulated by the author on Oct. 25, 2019.
\end{remark}

\begin{conjecture} \label{conj-f1-d} For any $n\in\Z^+$, we have
\begin{equation}\label{f1-d-n}
\f1{4n}\sum_{k=0}^{n-1}(-1)^{n-1-k}(105k+88)f_kT_k(8,1)\in\Z^+.
\end{equation}

{\rm (ii)} Let $p$ be an odd prime. Then
\begin{equation}\label{f1-d-p}
\sum_{k=0}^{p-1}(-1)^k(105k+88)f_kT_k(8,1)
\eq \f 83p\l(23\l(\f {-3}p\r)+10\l(\f{15}p\r)\r)\pmod{p^2}.
\end{equation}
If $(\f{-5}p)=1$, then
\begin{equation}\label{f1-d-pn}
\sum_{k=0}^{pn-1}(-1)^k(105k+88)f_kT_k(8,1)-p\l(\f p3\r)\sum_{k=0}^{n-1}(-1)^k(105k+88)f_kT_k(8,1)
\end{equation}
divided by $(pn)^2$ is a $p$-adic integer for any $n\in\Z^+$.

{\rm (iii)} Let $p>5$ be a prime. Then
\begin{equation}\label{f1-d-q}
\begin{aligned}&\sum_{k=0}^{p-1}(-1)^kf_kT_k(8,1)
\\\eq&\begin{cases}4x^2-2p\pmod{p^2}&\t{if}\ p\eq1,4\pmod{15}\ \&\ p=x^2+15y^2\ (x,y\in\Z),
\\2p-12x^2\pmod{p^2}&\t{if}\ p\eq2,8\pmod{15}\ \&\ p=3x^2+5y^2\ (x,y\in\Z),
\\0\pmod{p^2}&\t{if}\ (\f{-15}p)=-1.\end{cases}
\end{aligned}
\end{equation}
\end{conjecture}
\begin{remark}\rm This conjecture is the dual of Conjecture \ref{Conj-f1}.
\end{remark}

The following conjecture is related to the identity $(6.8)$.

\begin{conjecture}\label{Conj-f-3} {\rm (i)} For any $n\in\Z^+$, we have
\begin{equation}\label{f-3-n}\f1{2n}\sum_{k=0}^{n-1}(-1)^k(585k+172)13552^{n-1-k}f_kT_k(1454,1)\in\Z^+.
\end{equation}

{\rm (ii)} Let $p>2$ be a prime with $p\not=7,11$. Then
\begin{equation}\label{f-3-p}
\sum_{k=0}^{p-1}\f{585k+172}{(-13552)^k}f_kT_k(1454,1)\eq \f p{11}\l(1580\l(\f{-7}p\r)
+312\l(\f{273}p\r)\r)\pmod{p^2}.
\end{equation}
If $(\f {-39}p)=1$, then
\begin{equation}\label{f-3-pn}
\sum_{k=0}^{pn-1}\f{585k+172}{(-13552)^k}f_kT_k(1454,1)
-p\l(\f p7\r)\sum_{k=0}^{n-1}\f{585k+172}{(-13552)^k}f_kT_k(1454,1)
\end{equation}
divided by $(pn)^2$ is a $p$-adic integer for any $n\in\Z^+$.

{\rm (iii)} Let $p>3$ be a prime with $p\not=7,11,13$. Then
\begin{equation}\label{f-3-q}\begin{aligned}&\sum_{k=0}^{p-1}\f{f_kT_k(1454,1)}{(-13552)^k}
\\[2mm]\eq&\begin{cases}4x^2-2p\pmod{p^2}&\t{if}\ (\f{-1}p)=(\f p3)=(\f p7)=(\f p{13})=1,\ p=x^2+273y^2,
\\[2mm]2x^2-2p\pmod{p^2}&\t{if}\ (\f{-1}p)=(\f p{7})=1,\ (\f p3)=(\f p{13})=-1,\ 2p=x^2+273y^2,
\\[2mm]2p-12x^2\pmod{p^2}&\t{if}\ (\f{-1}p)=(\f p7)=-1,\ (\f p3)=(\f p{13})=1,\  p=3x^2+91y^2,
\\[2mm]2p-6x^2\pmod{p^2}&\t{if}\ (\f{-1}p)=(\f p3)=(\f p7)=(\f p{13})=-1,\ 2p=3x^2+91y^2,
\\[2mm]28x^2-2p\pmod{p^2}&\t{if}\ (\f{-1}p)=(\f p{13})=-1,\ (\f p3)=(\f p{7})=1,\ p=7x^2+39y^2,
\\[2mm]14x^2-2p\pmod{p^2}&\t{if}\ (\f{-1}p)=(\f p3)=-1,\ (\f p7)=(\f p{13})=1,\ 2p=7x^2+39y^2,
\\[2mm]52x^2-2p\pmod{p^2}&\t{if}\ (\f{-1}p)=(\f p{3})=1,\ (\f p7)=(\f p{13})=-1,\ p=13x^2+21y^2,
\\[2mm]26x^2-2p\pmod{p^2}&\t{if}\ (\f{-1}p)=(\f p{13})=1,\ (\f p3)=(\f p{7})=-1,\ 2p=13x^2+21y^2,
\\[2mm]0\pmod{p^2}&\t{if}\ (\f{-273}p)=-1,
\end{cases}
\end{aligned}\end{equation}
where $x$ and $y$ are integers.
\end{conjecture}
\begin{remark}\rm Note that the imaginary quadratic field $\Q(\sqrt{-273})$ has class number $8$.
\end{remark}

The following conjecture is related to the identity $(6.10)$.

\begin{conjecture}\label{Conj-f-4} {\rm (i)} For any $n\in\Z^+$, we have
\begin{equation}\label{f-4-n}\f1{2n}\sum_{k=0}^{n-1}(-1)^k(94185k+17014)105984^{n-1-k}f_kT_k(2302,1)\in\Z^+.
\end{equation}

{\rm (ii)} Let $p>3$ be a prime with $p\not=23$. Then
\begin{equation}\label{f-4-p}\begin{aligned}
&\sum_{k=0}^{p-1}\f{94185k+17014}{(-105984)^k}f_kT_k(2302,1)
\\\eq& \f p{16}\l(22659+249565\l(\f{-23}p\r)\r)\pmod{p^2}.
\end{aligned}
\end{equation}
If $(\f p{23})=1$, then
\begin{equation}\label{f-4-pn}
\sum_{k=0}^{pn-1}\f{94185k+17014}{(-105984)^k}f_kT_k(2302,1)
-p\sum_{k=0}^{n-1}\f{94185k+17014}{(-105984)^k}f_kT_k(2302,1)
\end{equation}
divided by $(pn)^2$ is a $p$-adic integer for any $n\in\Z^+$.

{\rm (iii)} Let $p>3$ be a prime with $p\not=23$. Then
\begin{equation}\label{f-4-q}\begin{aligned}
&\sum_{k=0}^{p-1}\f{f_kT_k(2302,1)}{(-105984)^k}
\\\eq&\begin{cases}4x^2-2p\pmod{p^2}&\t{if}\ (\f{-1}p)=(\f p3)=(\f p5)=(\f p{23})=1,\ p=x^2+345y^2,
\\[2mm]2x^2-2p\pmod{p^2}&\t{if}\ (\f{-1}p)=(\f p{23})=1,\ (\f p3)=(\f p{5})=-1,\ 2p=x^2+345y^2,
\\[2mm]12x^2-2p\pmod{p^2}&\t{if}\ (\f{-1}p)=(\f p5)=-1,\ (\f p3)=(\f p{23})=1,\  p=3x^2+115y^2,
\\[2mm]6x^2-2p\pmod{p^2}&\t{if}\ (\f{-1}p)=(\f p3)=-1,\ (\f p5)=(\f p{23})=1,\ 2p=3x^2+115y^2,
\\[2mm]2p-20x^2\pmod{p^2}&\t{if}\ (\f{-1}p)=(\f p{5})=1,\ (\f p3)=(\f p{23})=-1,\ p=5x^2+69y^2,
\\[2mm]2p-10x^2\pmod{p^2}&\t{if}\ (\f{-1}p)=(\f p3)=1,\ (\f p5)=(\f p{23})=-1,\  2p=5x^2+69y^2,
\\[2mm]2p-60x^2\pmod{p^2}&\t{if}\ (\f{-1}p)=(\f p{3})=(\f p5)=(\f p{23})=-1,\  p=15x^2+23y^2,
\\[2mm]2p-30x^2\pmod{p^2}&\t{if}\ (\f{-1}p)=(\f p{23})=-1,\ (\f p3)=(\f p{5})=1,\ 2p=15x^2+23y^2,
\\[2mm]0\pmod{p^2}&\t{if}\ (\f{-345}p)=-1,
\end{cases}
\end{aligned}\end{equation}
where $x$ and $y$ are integers.
\end{conjecture}
\begin{remark}\rm Note that the imaginary quadratic field $\Q(\sqrt{-345})$ has class number $8$.
\end{remark}

The following conjecture is related to the identity $(6.12)$.

\begin{conjecture}\label{Conj-f-462} {\rm (i)} For any $n\in\Z^+$, we have
\begin{equation}\label{f-462-n}\f1{n}\sum_{k=0}^{n-1}(210k+23)475904^{n-1-k}f_kT_k(16898,1)\in\Z^+.
\end{equation}

{\rm (ii)} Let $p$ be an odd prime with $p\not=11,13$. Then
\begin{equation}\label{f-462-p}\begin{aligned}
&\sum_{k=0}^{p-1}\f{210k+23}{475904^k}f_kT_k(16898,1)
\\\eq& \f p{1287}\l(40621\l(\f{-231}p\r)-11020\l(\f{66}p\r)\r)\pmod{p^2}.
\end{aligned}\end{equation}
If $(\f {-14}p)=1$, then
\begin{equation}\label{f-462-pn}\begin{aligned}
&\sum_{k=0}^{pn-1}\f{210k+23}{475904^k}f_kT_k(16898,1)
\\&-p\l(\f{66}p\r)\sum_{k=0}^{n-1}\f{210k+23}{475904^k}f_kT_k(16898,1)
\end{aligned}\end{equation}
divided by $(pn)^2$ is a $p$-adic integer for any $n\in\Z^+$.

{\rm (iii)} Let $p>3$ be a prime with $p\not=7,11,13$. Then
\begin{equation}\label{f-462-q}\begin{aligned}
&\sum_{k=0}^{p-1}\f{f_kT_k(16898,1)}{475904^k}
\\\eq&\begin{cases}4x^2-2p\pmod{p^2}&\t{if}\ (\f{2}p)=(\f p3)=(\f p7)=(\f p{11})=1\ \&\ p=x^2+462y^2,
\\8x^2-2p\pmod{p^2}&\t{if}\ (\f{2}p)=(\f p7)=1,\ (\f p3)=(\f p{11})=-1\ \&\ p=2x^2+231y^2,
\\2p-12x^2\pmod{p^2}&\t{if}\ (\f{2}p)=(\f p7)=-1,\ (\f p3)=(\f p{11})=1\ \&\ p=3x^2+154y^2,
\\2p-24x^2\pmod{p^2}&\t{if}\ (\f{2}p)=(\f p3)=(\f p7)=(\f p{11})=-1\ \&\ p=6x^2+77y^2,
\\28x^2-2p\pmod{p^2}&\t{if}\ (\f{2}p)=(\f p{3})=1,\ (\f p7)=(\f p{11})=-1\ \&\ p=7x^2+66y^2,
\\44x^2-2p\pmod{p^2}&\t{if}\ (\f{2}p)=(\f p3)=-1,\ (\f p7)=(\f p{11})=1\ \&\ p=11x^2+42y^2,
\\56x^2-2p\pmod{p^2}&\t{if}\ (\f{2}p)=(\f p{11})=1,\ (\f p3)=(\f p7)=-1\ \&\ p=14x^2+33y^2,
\\2p-84x^2\pmod{p^2}&\t{if}\ (\f{2}p)=(\f p{11})=-1,\ (\f p3)=(\f p{7})=1\ \&\ p=21x^2+22y^2,
\\0\pmod{p^2}&\t{if}\ (\f{-462}p)=-1,
\end{cases}
\end{aligned}\end{equation}
where $x$ and $y$ are integers.
\end{conjecture}
\begin{remark}\rm Note that the imaginary quadratic field $\Q(\sqrt{-462})$ has class number $8$.
\end{remark}

The identities $(6.5),\, (6.6),\,(6.7),\,(6.9),\,(6.11)$ are related to the quadratic fields $$\Q(\sqrt{-165}),\ \Q(\sqrt{-210}),\ \Q(\sqrt{-210}),\ \Q(\sqrt{-330}),\ \Q(\sqrt{-357})$$ (with class number $8$) respectively. We also have conjectures on related congruences similar to Conjectures \ref{Conj-f-3}, \ref{Conj-f-4} and \ref{Conj-f-462}.

\section{Series for $1/\pi$ involving $T_n(b,c)$ and $g_n=\sum_{k=0}^n\bi nk^2\bi{2k}k$}
 \setcounter{equation}{0}

For $n\in\N$ let
$$g_n:=\sum_{k=0}^n\bi nk^2\bi{2k}k.$$
It is known that $g_n=\sum_{k=0}^n\bi nk f_k$ for all $n\in\N$.
See \cite{S16RJ,GMP,MS} for some congruences on polynomials related to these numbers.

Let $p>3$ be a prime. For any $k=0,\ldots,p-1$, we have
$$g_k\eq\l(\f{-3}p\r)9^kg_{p-1-k}\pmod p$$ by \cite[Lemma 2.7(ii)]{JV}.
Combining this with Remark \ref{Rem-Dual}(ii), we see that
\begin{align*}\sum_{k=0}^{p-1}\f{g_kT_k(b,c)}{m^k}\eq&
\l(\f{-3(b^2-4c)}p\r)\sum_{k=0}^{p-1}\l(\f{9(b^2-4c)}m\r)^kg_{p-1-k}T_{p-1-k}(b,c)
\\\eq&\l(\f{3(4c-b^2)}p\r)\sum_{k=0}^{p-1}\f{g_kT_k(b,c)}{(9(b^2-4c)/m)^k}\pmod p
\end{align*}
for any $b,c,m\in\Z$ with $p\nmid (b^2-4c)m$.

Wan and Zudilin \cite{WZ} obtained the following irrational series for $1/\pi$ involving the Legendre polynomials and the sequence $(g_n)_{n\gs0}$:
$$\sum_{k=0}^\infty(22k+7-3\sqrt3)g_kP_k\l(\f{\sqrt{14\sqrt3-15}}3\r)\l(\f{\sqrt{2\sqrt3-3}}{9}\r)^k
=\f{9}{2\pi}(9+4\sqrt3).$$
Using our congruence approach (including Conjecture \ref{Conj-Dual}), we find 12 rational series for $1/\pi$
involving $T_n(b,c)$ and $g_n$; Theorem 1 of \cite{WZ} might be helpful to solve some of them.

\begin{conjecture}\label{Conj-g} We have the following identities.
\begin{align}\sum_{k=0}^\infty\f{8k+3}{(-81)^k}g_kT_k(7,-8)&=\f{9\sqrt3}{4\pi},
\\\sum_{k=0}^\infty\f{4k+1}{(-1089)^k}g_kT_k(31,-32)&=\f{33}{16\pi},
\\\sum_{k=0}^\infty\f{7k-1}{540^k}g_kT_k(52,1)&=\f{30\sqrt3}{\pi},
\\\sum_{k=0}^\infty\f{20k+3}{3969^k}g_kT_k(65,64)&=\f{63\sqrt3}{8\pi},
\\\sum_{k=0}^\infty\f{280k+93}{(-1980)^k}g_kT_k(178,1)&=\f{20\sqrt{33}}{\pi},
\\\sum_{k=0}^\infty\f{176k+15}{12600^k}g_kT_k(502,1)&=\f{25\sqrt{42}}{\pi},
\\\sum_{k=0}^\infty\f{560k-23}{13068^k}g_kT_k(970,1)&=\f{693\sqrt3}{\pi},
\\\sum_{k=0}^\infty\f{12880k+1353}{105840^k}g_kT_k(2158,1)&=\f{4410\sqrt3}{\pi},
\\\sum_{k=0}^\infty\f{299k+59}{(-101430)^k}g_kT_k(2252,1)&=\f{735\sqrt{115}}{64\pi},
\\\sum_{k=0}^\infty\f{385k+118}{(-53550)^k}g_kT_k(4048,1)&=\f{2415\sqrt{17}}{64\pi},
\\\sum_{k=0}^\infty\f{385k-114}{114264^k}g_kT_k(10582,1)&=\f{15939\sqrt3}{16\pi},
\\\sum_{k=0}^\infty\f{16016k+1273}{510300^k}g_kT_k(17498,1)&=\f{14175\sqrt3}{2\pi}.
\end{align}
\end{conjecture}

Now we present a conjecture on congruences related to $(7.6)$.

\begin{conjecture}\label{Conj-g-2} {\rm (i)} For any $n\in\Z^+$, we have
\begin{equation}\label{g-2-n}\f1{3n}\sum_{k=0}^{n-1}(176k+15)12600^{n-1-k}g_kT_k(502,1)\in\Z^+,
\end{equation}
and this number is odd if and only if $n\in\{2^a:\ a\in\N\}$.

{\rm (ii)} Let $p>7$ be a prime. Then
\begin{equation}\label{g-2-p}
\sum_{k=0}^{p-1}\f{176k+15}{12600^k}g_kT_k(502,1)\eq p\l(26\l(\f{-42}p\r)-11\l(\f{21}p\r)\r)\pmod{p^2}.
\end{equation}
If $p\eq1,3\pmod8$, then
\begin{equation}\label{g-2-pn}
\sum_{k=0}^{pn-1}\f{176k+15}{12600^k}g_kT_k(502,1)
-p\l(\f{21}p\r)\sum_{k=0}^{n-1}\f{176k+15}{12600^k}g_kT_k(502,1)
\end{equation}
divided by $(pn)^2$ is a $p$-adic integer for any $n\in\Z^+$.

{\rm (iii)} Let $p>7$ be a prime. Then
\begin{equation}\label{g-2-q}\begin{aligned}&\sum_{k=0}^{p-1}\f{g_kT_k(502,1)}{12600^k}
\\\eq&\begin{cases}4x^2-2p\pmod{p^2}&\t{if}\ (\f{-2}p)=(\f p3)=(\f p5)=(\f p7)=1\ \&\ p=x^2+210y^2,
\\2p-8x^2\pmod{p^2}&\t{if}\ (\f{-2}p)=(\f p7)=1,\ (\f p3)=(\f p5)=-1\ \&\ p=2x^2+105y^2,
\\2p-12x^2\pmod{p^2}&\t{if}\ (\f{-2}p)=(\f p3)=1,\ (\f p5)=(\f p7)=-1\ \&\ p=3x^2+70y^2,
\\2p-20x^2\pmod{p^2}&\t{if}\ (\f{-2}p)=(\f p3)=(\f p5)=(\f p7)=-1\ \&\ p=5x^2+42y^2,
\\24x^2-2p\pmod{p^2}&\t{if}\ (\f{-2}p)=(\f p5)=1,\ (\f p3)=(\f p7)=-1\ \&\ p=6x^2+35y^2,
\\28x^2-2p\pmod{p^2}&\t{if}\ (\f{-2}p)=(\f p5)=-1,\ (\f p3)=(\f p7)=1\ \&\ p=7x^2+30y^2,
\\40x^2-2p\pmod{p^2}&\t{if}\ (\f{-2}p)=(\f p7)=-1,\ (\f p3)=(\f p5)=1\ \&\ p=10x^2+21y^2,
\\56x^2-2p\pmod{p^2}&\t{if}\ (\f{-2}p)=(\f p3)=-1,\ (\f p5)=(\f p7)=1\ \&\ p=14x^2+15y^2,
\\0\pmod{p^2}&\t{if}\ (\f{-210}p)=-1,
\end{cases}
\end{aligned}\end{equation}
where $x$ and $y$ are integers.
\end{conjecture}
\begin{remark}\rm Note that the imaginary quadratic field $\Q(\sqrt{-210})$ has class number $8$.
\end{remark}

The following conjecture is related to the identity $(7.8)$.

\begin{conjecture}\label{Conj-g-330} {\rm (i)} For any $n\in\Z^+$, we have
\begin{equation}\label{g-330-n}\f1{3n}\sum_{k=0}^{n-1}(12880k+1353)105840^{n-1-k}g_kT_k(2158,1)\in\Z^+,
\end{equation}
and this number is odd if and only if $n\in\{2^a:\ a\in\N\}$.

{\rm (ii)} Let $p>7$ be a prime. Then
\begin{equation}\label{g-330-p}\begin{aligned}
&\sum_{k=0}^{p-1}\f{12880k+1353}{105840^k}g_kT_k(2158,1)
\\\eq& \f p{2}\l(3419\l(\f {-3}p\r)-713\l(\f{5}p\r)\r)\pmod{p^2}.
\end{aligned}\end{equation}
If $(\f p3)=(\f p{5})$, then
\begin{equation}\label{g-330-pn}
\sum_{k=0}^{pn-1}\f{12880k+1353}{105840^k}g_kT_k(2158,1)
-p\l(\f p3\r)\sum_{k=0}^{n-1}\f{12880k+1353}{105840^k}g_kT_k(2158,1)
\end{equation}
divided by $(pn)^2$ is a $p$-adic integer for any $n\in\Z^+$.

{\rm (iii)} Let $p>11$ be a prime. Then
\begin{equation}\label{g-330-q}\begin{aligned}
&\sum_{k=0}^{p-1}\f{g_kT_k(2158,1)}{105840^k}
\\\eq&\begin{cases}4x^2-2p\pmod{p^2}&\t{if}\ (\f{-1}p)=(\f p3)=(\f p5)=(\f p{11})=1,\ p=x^2+330y^2,
\\[2mm]2p-8x^2\pmod{p^2}&\t{if}\ (\f{-1}p)=(\f p3)=(\f p5)=(\f p{11})=-1,\ p=2x^2+165y^2,
\\[2mm]2p-12x^2\pmod{p^2}&\t{if}\ (\f{-1}p)=(\f p{11})=1,\ (\f p3)=(\f p5)=-1,\ p=3x^2+110y^2,
\\[2mm]2p-20x^2\pmod{p^2}&\t{if}\ (\f{-1}p)=(\f p3)=-1,\ (\f p5)=(\f p{11})=1,\ p=5x^2+66y^2,
\\[2mm]24x^2-2p\pmod{p^2}&\t{if}\ (\f{-1}p)=(\f p{11})=-1,\ (\f p3)=(\f p5)=1,\ p=6x^2+55y^2,
\\[2mm]40x^2-2p\pmod{p^2}&\t{if}\ (\f{-1}p)=(\f p3)=1,\ (\f p5)=(\f p{11})=-1,\  p=10x^2+33y^2,
\\[2mm]44x^2-2p\pmod{p^2}&\t{if}\ (\f{-1}p)=(\f p{5})=1,\ (\f p3)=(\f p{11})=-1,\ p=11x^2+30y^2,
\\[2mm]60x^2-2p\pmod{p^2}&\t{if}\ (\f{-1}p)=(\f p5)=-1,\ (\f p3)=(\f p{11})=1,\  p=15x^2+22y^2,
\\[2mm]0\pmod{p^2}&\t{if}\ (\f{-330}p)=-1,
\end{cases}
\end{aligned}\end{equation}
where $x$ and $y$ are integers.
\end{conjecture}
\begin{remark}\rm Note that the imaginary quadratic field $\Q(\sqrt{-330})$ has class number $8$.
\end{remark}

Now we pose a conjecture related to the identity $(7.10)$.

\begin{conjecture}\label{Conj-g-3} {\rm (i)} For any $n\in\Z^+$, we have
\begin{equation}\label{g-3-n}\f1{2n}\sum_{k=0}^{n-1}(-1)^k(385k+118)53550^{n-1-k}g_kT_k(4048,1)\in\Z^+.
\end{equation}

{\rm (ii)} Let $p>7$ be a prime with $p\not=17$. Then
\begin{equation}\label{g-3-p}\begin{aligned}
&\sum_{k=0}^{p-1}\f{385k+118}{(-53550)^k}g_kT_k(4048,1)
\\[2mm]\eq& \f p{320}\l(29279\l(\f{-17}p\r)+8481\l(\f{7}p\r)\r)\pmod{p^2}.
\end{aligned}
\end{equation}
If $(\f p7)=(\f p{17})$, then
\begin{equation}\label{g-3-pn}
\f1{(pn)^2}\(\sum_{k=0}^{pn-1}\f{385k+118}{(-53550)^k}g_kT_k(4048,1)
-p\l(\f{7}p\r)\sum_{k=0}^{n-1}\f{385k+118}{(-53550)^k}g_kT_k(4048,1)\)
\end{equation}
is a $p$-adic integer for any $n\in\Z^+$.

{\rm (iii)} Let $p>7$ be a prime with $p\not=17$. Then

\begin{equation}\label{g-3-q}\begin{aligned}&\sum_{k=0}^{p-1}\f{g_kT_k(4048,1)}{(-53550)^k}
\\\eq&\begin{cases}4x^2-2p\pmod{p^2}&\t{if}\ (\f{-1}p)=(\f p3)=(\f p7)=(\f p{17})=1,\ p=x^2+357y^2,
\\[2mm]2p-2x^2\pmod{p^2}&\t{if}\ (\f{-1}p)=(\f p3)=-1,\ (\f p7)=(\f p{17})=1,\ 2p=x^2+357y^2,
\\[2mm]12x^2-2p\pmod{p^2}&\t{if}\ (\f{-1}p)=(\f p3)=(\f p5)=(\f p7)=-1,\ p=3x^2+119y^2,
\\[2mm]2p-6x^2\pmod{p^2}&\t{if}\ (\f{-1}p)=(\f p3)=1,\ (\f p7)=(\f p{17})=-1,\ 2p=3x^2+119y^2,
\\[2mm]28x^2-2p\pmod{p^2}&\t{if}\ (\f{-1}p)=(\f p{17})=-1,\ (\f p3)=(\f p7)=1,\  p=7x^2+51y^2,
\\[2mm]2p-14x^2\pmod{p^2}&\t{if}\ (\f{-1}p)=(\f p7)=1,\ (\f p3)=(\f p{17})=-1,\ 2p=7x^2+51y^2,
\\[2mm]2p-68x^2\pmod{p^2}&\t{if}\ (\f{-1}p)=(\f p{17})=1,\ (\f p3)=(\f p7)=-1,\ p=17x^2+21y^2,
\\[2mm]34x^2-2p\pmod{p^2}&\t{if}\ (\f{-1}p)=(\f p7)=-1,\ (\f p3)=(\f p{17})=1,\ 2p=17x^2+21y^2,
\\[2mm]0\pmod{p^2}&\t{if}\ (\f{-357}p)=-1,
\end{cases}
\end{aligned}\end{equation}
where $x$ and $y$ are integers.
\end{conjecture}
\begin{remark}\rm Note that the imaginary quadratic field $\Q(\sqrt{-357})$ has class number $8$.
\end{remark}

Now we pose a conjecture related to the identity $(7.12)$.

\begin{conjecture}\label{Conj-g-462} {\rm (i)} For any $n\in\Z^+$, we have
\begin{equation}\label{g-462-n}\f1{n}\sum_{k=0}^{n-1}(16016k+1273)510300^{n-1-k}g_kT_k(17498,1)\in\Z^+,
\end{equation}
and this number is odd if and only if $n\in\{2^a:\ a\in\N\}$.

{\rm (ii)} Let $p>7$ be a prime. Then
\begin{equation}\label{g-462-p}\begin{aligned}
&\sum_{k=0}^{p-1}\f{16016k+1273}{510300^k}g_kT_k(17498,1)
\\[2mm]\eq& \f p{3}\l(6527\l(\f{-3}p\r)-2708\l(\f{42}p\r)\r)\pmod{p^2}.
\end{aligned}
\end{equation}
If $(\f {-14}p)=1$, then
\begin{equation}\label{g-462-pn}\begin{aligned}
&\sum_{k=0}^{pn-1}\f{16016k+1273}{510300^k}g_kT_k(17498,1)
\\&-p\l(\f{p}3\r)\sum_{k=0}^{n-1}\f{16016k+1273}{510300^k}g_kT_k(17498,1)
\end{aligned}\end{equation}
divided by $(pn)^2$ is a $p$-adic integer for each $n\in\Z^+$.

{\rm (iii)} Let $p>11$ be a prime. Then
\begin{equation}\label{g-462-q}\begin{aligned}&\sum_{k=0}^{p-1}\f{g_kT_k(17498,1)}{510300^k}
\\\eq&\begin{cases}4x^2-2p\pmod{p^2}&\t{if}\ (\f{2}p)=(\f p3)=(\f p7)=(\f p{11})=1\ \&\ p=x^2+462y^2,
\\2p-8x^2\pmod{p^2}&\t{if}\ (\f{2}p)=(\f p7)=1,\ (\f p3)=(\f p{11})=-1\ \&\ p=2x^2+231y^2,
\\12x^2-2p\pmod{p^2}&\t{if}\ (\f{2}p)=(\f p7)=-1,\ (\f p3)=(\f p{11})=1\ \&\ p=3x^2+154y^2,
\\2p-24x^2\pmod{p^2}&\t{if}\ (\f{2}p)=(\f p3)=(\f p7)=(\f p{11})=-1\ \&\ p=6x^2+77y^2,
\\28x^2-2p\pmod{p^2}&\t{if}\ (\f{2}p)=(\f p{3})=1,\ (\f p7)=(\f p{11})=-1\ \&\ p=7x^2+66y^2,
\\44x^2-2p\pmod{p^2}&\t{if}\ (\f{2}p)=(\f p3)=-1,\ (\f p7)=(\f p{11})=1\ \&\ p=11x^2+42y^2,
\\2p-56x^2\pmod{p^2}&\t{if}\ (\f{2}p)=(\f p{11})=1,\ (\f p3)=(\f p7)=-1\ \&\ p=14x^2+33y^2,
\\84x^2-2p\pmod{p^2}&\t{if}\ (\f{2}p)=(\f p{11})=-1,\ (\f p3)=(\f p{7})=1\ \&\ p=21x^2+22y^2,
\\0\pmod{p^2}&\t{if}\ (\f{-462}p)=-1,
\end{cases}
\end{aligned}\end{equation}
where $x$ and $y$ are integers.
\end{conjecture}
\begin{remark}\rm Note that the imaginary quadratic field $\Q(\sqrt{-462})$ has class number $8$.
We believe that $462$ is the largest positive squarefree number $d$ for which the imaginary quadratic field $\Q(\sqrt{-d})$ can be used to construct a Ramanujan-type series for $1/\pi$.
\end{remark}

The identities $(7.5),\,(7.7),\,(7.9)$ are related to the quadratic fields $\Q(\sqrt{-165})$,
$\Q(\sqrt{-210})$, $\Q(\sqrt{-345})$ (with class number $8$) respectively. We also have conjectures on related congruences similar to Conjectures \ref{Conj-g-2}, \ref{Conj-g-330}, \ref{Conj-g-3}
and \ref{Conj-g-462}.

To conclude this section, we confirm an open series for $1/\pi$ conjectured by the author
(cf. \cite[(3.28)]{S-11} and \cite[Conjecture 7.9]{S13d}) in 2011.

\begin{theorem}\label{Th7.1} We have
\begin{equation}\label{g-20}\sum_{n=0}^\infty\f{16n+5}{324^n}\bi{2n}ng_n(-20)=\f{189}{25\pi},
\end{equation}
where
$$g_n(x):=\sum_{k=0}^n\bi nk^2\bi{2k}kx^k.$$
\end{theorem}
\Proof. The Franel numbers of order $4$ are given by $f_n^{(4)}=\sum_{k=0}^n\bi nk^4\ (n\in\N)$.
Note that
$$f_n^{(4)}\ls\(\sum_{k=0}^n\bi nk^2\)^2=\bi{2n}n^2\ls ((1+1)^{2n})^2=16^n.$$
By \cite[(8.1)]{CWZ}, for $|x|<1/16$ and $a,b\in\Z$, we have
\begin{equation}\label{cwz}\begin{aligned}&\sum_{n=0}^\infty\bi{2n}n\f{(an+b)x^n}{(1+2x)^{2n}}\sum_{k=0}^n\bi nk^2\bi{2(n-k)}{n-k}x^k
\\=&(1+2x)\sum_{n=0}^\infty\l(\f{4a(1-x)(1+2x)n+6ax(2-x)}{5(1-4x)}+b\r)f_n^{(4)}x^n.
\end{aligned}\end{equation}
Since
\begin{align*}&\f{x^n}{(1+2x)^{2n}}\sum_{k=0}^n\bi nk^2\bi{2n-2k}{n-k}x^k
\\=&\f{x^n}{(1+2x)^{2n}}\sum_{k=0}^n\bi nk^2\bi{2k}{k}x^{n-k}=(2+x^{-1})^{-2n}g_n(x^{-1}),
\end{align*}
putting $a=16$, $b=5$ and $x=-1/20$ in \eqref{cwz} we obtain
$$\sum_{n=0}^\infty\f{16n+5}{18^{2n}}\bi{2n}ng_n(-20)
=\f{378}{125}\sum_{n=0}^\infty\f{3n+1}{(-20)^n}f_n^{(4)}.$$
As
$$\sum_{n=0}^\infty\f{3n+1}{(-20)^n}f_n^{(4)}=\f{5}{2\pi}$$
by Cooper \cite{Co}, we finally get
$$\sum_{n=0}^\infty\f{16n+5}{18^{2n}}\bi{2n}ng_n(-20)=\f{378}{125}\times \f 5{2\pi}=\f{189}{25\pi}.$$
This concludes the proof of \eqref{g-20}. \qed

\section{Series and congruences involving $T_n(b,c)$
 and $\beta_n=\sum_{k=0}^n\bi nk^2\bi{n+k}k$}
 \setcounter{equation}{0}

Recall that the numbers
$$\beta_n:=\sum_{k=0}^n\bi nk^2\bi{n+k}k\ \ \ (n=0,1,2,\ldots)$$
are a kind of Ap\'ery numbers.
Let $p$ be an odd prime. For any $k=0,1,\ldots,p-1$, we have
$$\beta_k\eq(-1)^k\beta_{p-1-k}\pmod p$$
by \cite[Lemma 2.7(i)]{JV}. Combining this with Remark \ref{Rem-Dual}(ii), we see that
\begin{align*}\sum_{k=0}^{p-1}\f{\beta_kT_k(b,c)}{m^k}\eq&
\l(\f{b^2-4c}p\r)\sum_{k=0}^{p-1}\l(\f{-(b^2-4c)}m\r)^k\beta_{p-1-k}T_{p-1-k}(b,c)
\\\eq&\l(\f{b^2-4c}p\r)\sum_{k=0}^{p-1}\f{\beta_kT_k(b,c)}{((4c-b^2)/m)^k}\pmod p
\end{align*}
for any $b,c,m\in\Z$ with $p\nmid (b^2-4c)m$.

Wan and Zudilin \cite{WZ} obtained the following irrational series for $1/\pi$ involving the Legendre polynomials and the numbers $\beta_n$:
$$\sum_{k=0}^\infty(60k+16-5\sqrt{10})\beta_kP_k\l(\f{5\sqrt2+17\sqrt5}{45}\r)\l(\f{5\sqrt2-3\sqrt5}
5\r)^k
=\f{135\sqrt2+81\sqrt5}{\sqrt2\,\pi}.$$
Using our congruence approach (including Conjecture \ref{Conj-Dual}), we find one rational series for $1/\pi$
involving $T_n(b,c)$ and the Ap\'ery numbers $\beta_n$ (see \eqref{beta} below); Theorem 1 of \cite{WZ} might be helpful to solve it.

\begin{conjecture} \label{conj-beta} {\rm (i)} We have
\begin{equation}\label{beta}\sum_{k=0}^\infty\f{145k+9}{900^k}\beta_kT_k(52,1)=\f{285}{\pi}.
\end{equation}
Also, for any $n\in\Z^+$ we have
\begin{equation}\label{beta-n}\f1n\sum_{k=0}^{n-1}(145k+9)900^{n-1-k}\beta_kT_k(52,1)\in\Z^+.
\end{equation}

{\rm (ii)} Let $p>5$ be a prime. Then
\begin{equation}\label{beta-p}
\sum_{k=0}^{p-1}\f{145k+9}{900^k}\beta_kT_k(52,1)
\eq\f p5\l(133\l(\f{-1}p\r)-88\r)\pmod{p^2}.
\end{equation}
If $p\eq1\pmod4$, then
\begin{equation}\label{beta-pn}
\f1{(pn)^2}\(\sum_{k=0}^{pn-1}\f{145k+9}{900^k}\beta_kT_k(52,1)
-p\sum_{k=0}^{n-1}\f{145k+9}{900^k}\beta_kT_k(52,1)\)
\in\Z_p
\end{equation}
for all $n\in\Z^+$.

{\rm (iii)} Let $p>5$ be a prime. Then
\begin{equation}\label{beta-q}
\begin{aligned}&\l(\f{-1}p\r)\sum_{k=0}^{p-1}\f{\beta_kT_k(52,1)}{900^k}
\\\eq&\begin{cases}4x^2-2p\pmod{p^2}&\t{if}\ p\eq1,4\pmod{15}\ \&\ p=x^2+15y^2\ (x,y\in\Z),
\\2p-12x^2\pmod{p^2}&\t{if}\ p\eq2,8\pmod{15}\ \&\ p=3x^2+5y^2\ (x,y\in\Z),
\\0\pmod{p^2}&\t{if}\ (\f{-15}p)=-1.\end{cases}
\end{aligned}
\end{equation}
\end{conjecture}
\begin{remark}\label{Rem-beta} \rm This conjecture was formulated by the author on Oct. 27, 2019.
\end{remark}

\begin{conjecture} \label{conj-beta1} {\rm (i)} For any $n\in\Z^+$, we have
\begin{equation}\label{beta1-n}\f1{2n}\sum_{k=0}^{n-1}(-1)^k(15k+8)\beta_kT_k(4,-1)\in\Z,
\end{equation}
and this number is odd if and only if $n\in\{2^a:\ a\in\Z^+\}$.

{\rm (ii)} Let $p$ be a prime. Then
\begin{equation}\label{beta1-p}
\sum_{k=0}^{p-1}(-1)^k(15k+8)\beta_kT_k(4,-1)
\eq\f p4\l(27\l(\f p3\r)+5\l(\f p5\r)\r)\pmod{p^2}.
\end{equation}
If $(\f {-15}p)=1\ ($i.e., $p\eq1,2,4,8\pmod{15})$, then
\begin{equation}\label{beta1-pn}
\sum_{k=0}^{pn-1}(-1)^k(15k+8)\beta_kT_k(4,-1)
-p\l(\f p3\r)\sum_{k=0}^{n-1}(-1)^k(15k+8)\beta_kT_k(2,2)
\end{equation}
divided by $(pn)^2$ is a $p$-adic integer for any $n\in\Z^+$.

{\rm (iii)} For any prime $p>5$, we have
\begin{equation}\label{beta1-q}
\begin{aligned}&\sum_{k=0}^{p-1}(-1)^k\beta_kT_k(4,-1)
\\\eq&\begin{cases}4x^2-2p\pmod{p^2}&\t{if}\ p\eq1,4\pmod{15}\ \&\ p=x^2+15y^2\ (x,y\in\Z),
\\12x^2-2p\pmod{p^2}&\t{if}\ p\eq2,8\pmod{15}\ \&\ p=3x^2+5y^2\ (x,y\in\Z),
\\0\pmod{p^2}&\t{if}\ (\f {-15}p)=-1.\end{cases}
\end{aligned}
\end{equation}
\end{conjecture}
\begin{remark}\label{Rem-beta1} \rm This conjecture was formulated by the author on Nov. 13, 2019.
\end{remark}

\begin{conjecture} \label{conj-beta2} {\rm (i)} For any $n\in\Z^+$, we have
\begin{equation}\label{beta2-n}\f3{n2^{\lfloor n/2\rfloor}}\sum_{k=0}^{n-1}(2k+1)(-2)^{n-1-k}\beta_kT_k(2,2)\in\Z^+,
\end{equation}
and this number is odd if and only if $n$ is a power of two.

{\rm (ii)} Let $p>3$ be a prime. Then
\begin{equation}\label{beta2-p}
\sum_{k=0}^{p-1}\f{2k+1}{(-2)^k}\beta_kT_k(2,2)
\eq\f p3\l(1+2\l(\f{-1}p\r)\r)\pmod{p^2}.
\end{equation}
If $p\eq1\pmod4$, then
\begin{equation}\label{beta2-pn}
\f1{(pn)^2}\(\sum_{k=0}^{pn-1}\f{2k+1}{(-2)^k}\beta_kT_k(2,2)
-p\sum_{k=0}^{n-1}\f{2k+1}{(-2)^k}\beta_kT_k(2,2)\)\in\Z_p
\end{equation}
for all $n\in\Z^+$.

{\rm (iii)} For any odd prime $p$, we have
\begin{equation}\label{beta2-q}
\begin{aligned}&\sum_{k=0}^{p-1}\f{\beta_kT_k(2,2)}{(-2)^k}
\\\eq&\begin{cases}4x^2-2p\pmod{p^2}&\t{if}\ p\eq1\pmod4\ \&\ p=x^2+4y^2\ (x,y\in\Z),
\\0\pmod{p^2}&\t{if}\ p\eq3\pmod4.\end{cases}
\end{aligned}
\end{equation}
\end{conjecture}
\begin{remark}\label{Rem-beta2} \rm This conjecture was formulated by the author on Nov. 13, 2019.
\end{remark}

\begin{conjecture} \label{conj-beta3} {\rm (i)} For any $n\in\Z^+$, we have
\begin{equation}\label{beta3-n}
\f1{n2^{\lfloor(n+1)/2\rfloor}}\sum_{k=0}^{n-1}(3k+2)(-2)^{n-1-k}\beta_kT_k(20,2)\in\Z^+,
\end{equation}
and this number is odd if and only if $n\in\{2^a:\ a=0,2,3,4,\ldots\}$.

{\rm (ii)} Let $p$ be any odd prime. Then
\begin{equation}\label{beta3-p}
\sum_{k=0}^{p-1}\f{3k+2}{(-2)^k}\beta_kT_k(20,2)
\eq2p\l(\f 2p\r)\pmod{p^2}.
\end{equation}
If $p\eq\pm1\pmod8$, then
\begin{equation}\label{beta3-pn}
\f1{(pn)^2}\(\sum_{k=0}^{pn-1}\f{3k+2}{(-2)^k}\beta_kT_k(20,2)
-p\sum_{k=0}^{n-1}\f{3k+2}{(-2)^k}\beta_kT_k(20,2)\)\in\Z_p
\end{equation}
for all $n\in\Z^+$.

{\rm (iii)} For any odd prime $p$, we have
\begin{equation}\label{beta3-q}
\begin{aligned}&\sum_{k=0}^{p-1}\f{\beta_kT_k(20,2)}{(-2)^k}
\\\eq&\begin{cases}4x^2-2p\pmod{p^2}&\t{if}\ p\eq1\pmod4\ \&\ p=x^2+4y^2\ (x,y\in\Z),
\\0\pmod{p^2}&\t{if}\ p\eq3\pmod4.\end{cases}
\end{aligned}
\end{equation}
\end{conjecture}

\begin{conjecture} \label{conj-beta4} {\rm (i)} For any $n\in\Z^+$, we have
\begin{equation}\label{beta4-n}\f3n\sum_{k=0}^{n-1}(5k+3)4^{n-1-k}\beta_kT_k(14,-1)\in\Z.
\end{equation}

{\rm (ii)} Let $p>3$ be a prime. Then
\begin{equation}\label{beta4-p}
\sum_{k=0}^{p-1}\f{5k+3}{4^k}\beta_kT_k(14,-1)
\eq\f p3\l(4\l(\f{-2}p\r)+5\l(\f 2p\r)\r)\pmod{p^2}.
\end{equation}
If $p\eq1\pmod4$, then
\begin{equation}\label{beta4-pn}
\f1{(pn)^2}\(\sum_{k=0}^{pn-1}\f{5k+3}{4^k}\beta_kT_k(14,-1)
-p\l(\f 2p\r)\sum_{k=0}^{n-1}\f{5k+3}{4^k}\beta_kT_k(14,-1)\)\in\Z_p
\end{equation}
for all $n\in\Z^+$.

{\rm (iii)} Let $p\not=2,5$ be a prime. Then
\begin{equation}\label{beta4-q}
\begin{aligned}&\sum_{k=0}^{p-1}\f{\beta_kT_k(14,-1)}{4^k}
\\\eq&\begin{cases}4x^2-2p\pmod{p^2}&\t{if}\ (\f {-2}p)=(\f 5p)=1\ \&\ p=x^2+10y^2\ (x,y\in\Z),
\\8x^2-2p\pmod{p^2}&\t{if}\ (\f {-2}p)=(\f 5p)=-1\ \&\ p=2x^2+5y^2\ (x,y\in\Z),
\\0\pmod{p^2}&\t{if}\ (\f {-10}p)=-1.\end{cases}
\end{aligned}
\end{equation}
\end{conjecture}

\begin{conjecture} \label{conj-beta6} {\rm (i)} For any $n\in\Z^+$, we have
\begin{equation}\label{beta6-n}\f1{3n}\sum_{k=0}^{n-1}(22k+15)(-4)^{n-1-k}\beta_kT_k(46,1)\in\Z^+,
\end{equation}
and this number is odd if and only if $n$ is a power of two.

{\rm (ii)} Let $p$ be an odd prime. Then
\begin{equation}\label{beta6-p}
\sum_{k=0}^{p-1}\f{22k+15}{(-4)^k}\beta_kT_k(46,1)
\eq\f p4\l(357-297\l(\f{33}p\r)\r)\pmod{p^2}.
\end{equation}
If $(\f{33}p)=1$, then
\begin{equation}\label{beta6-pn}
\f1{(pn)^2}\(\sum_{k=0}^{pn-1}\f{22k+15}{(-4)^k}\beta_kT_k(46,1)
-p\sum_{k=0}^{n-1}\f{22k+15}{(-4)^k}\beta_kT_k(46,1)\)\in\Z_p
\end{equation}
for all $n\in\Z^+$.

{\rm (iii)} Let $p>3$ be a prime. Then
\begin{equation}\label{beta6-q}
\begin{aligned}&\sum_{k=0}^{p-1}\f{\beta_kT_k(46,1)}{(-4)^k}
\\\eq&\begin{cases}x^2-2p\pmod{p^2}&\t{if}\ (\f p{11})=1\ \&\ 4p=x^2+11y^2\ (x,y\in\Z),
\\0\pmod{p^2}&\t{if}\ (\f p{11})=-1.\end{cases}
\end{aligned}
\end{equation}
\end{conjecture}

\begin{conjecture} \label{conj-beta7} {\rm (i)} For any $n\in\Z^+$, we have
\begin{equation}\label{beta7-n}\f1n\sum_{k=0}^{n-1}(190k+91)(-60)^{n-1-k}\beta_kT_k(82,1)\in\Z^+,
\end{equation}
and this number is odd if and only if $n$ is a power of two.

{\rm (ii)} Let $p>5$ be a prime. Then
\begin{equation}\label{beta7-p}
\sum_{k=0}^{p-1}\f{190k+91}{(-60)^k}\beta_kT_k(82,1)
\eq\f p4\l(111+253\l(\f{-15}p\r)\r)\pmod{p^2}.
\end{equation}
If $(\f{-15}p)=1$, then
\begin{equation}\label{beta7-pn}
\f1{(pn)^2}\(\sum_{k=0}^{pn-1}\f{190k+91}{(-60)^k}\beta_kT_k(82,1)
-p\sum_{k=0}^{n-1}\f{190k+91}{(-60)^k}\beta_kT_k(82,1)\)\in\Z_p
\end{equation}
for all $n\in\Z^+$.

{\rm (iii)} For any prime $p>7$, we have
\begin{equation}\label{beta7-q}
\begin{aligned}&\l(\f p3\r)\sum_{k=0}^{p-1}\f{\beta_kT_k(82,1)}{(-60)^k}
\\\eq&\begin{cases}x^2-2p\pmod{p^2}&\t{if}\ (\f p{5})=(\f p7)=1\ \&\ 4p=x^2+35y^2\ (x,y\in\Z),
\\2p-5x^2\pmod{p^2}&\t{if}\ (\f p{5})=(\f p7)=-1\ \&\ 4p=5x^2+7y^2\ (x,y\in\Z),
\\0\pmod{p^2}&\t{if}\ (\f {-35}p)=-1.\end{cases}
\end{aligned}
\end{equation}
\end{conjecture}

\section{Series and congruences involving
 $w_n=\sum_{k=0}^{\lfloor n/3\rfloor}(-1)^k3^{n-3k}\bi n{3k}\bi{3k}k\bi{2k}k$
 and $T_n(b,c)$}
 \setcounter{equation}{0}

The numbers
$$w_n:=\sum_{k=0}^{\lfloor n/3\rfloor}(-1)^k3^{n-3k}\bi n{3k}\bi{3k}k\bi{2k}k\ \ \ (n=0,1,2,\ldots)$$
were first introduced by Zagier \cite{Zag} during his study of Ap\'ery-like integer sequences, who noted the recurrence
$$(n+1)^2w_{n+1}=(9n(n+1)+3)w_n-27n^2w_{n-1}\ (n=1,2,3,\ldots).$$

\begin{lemma}\label{Lem-w} Let $p>3$ be a prime. Then
$$w_k\eq\l(\f{-3}p\r)27^kw_{p-1-k}\pmod p\quad  \t{for all}\ k=0,\ldots,p-1.$$
\end{lemma}
\Proof. Note that
\begin{align*} w_{p-1}=&\sum_{k=0}^{\lfloor(p-1)/3\rfloor}(-1)^k3^{p-1-3k}\bi{p-1}{3k}\bi{3k}k\bi{2k}k
\\\eq&\sum_{k=0}^{p-1}\f{\bi{2k}k\bi{3k}k}{27^k}\eq\l(\f p3\r)\pmod p
\end{align*}
with the help of the known congruence $\sum_{k=0}^{p-1}\bi{2k}k\bi{3k}k/27^k\eq(\f p3)\pmod{p^2}$
conjectured by F. Rodriguez-Villegas \cite{RV} and proved by E. Mortenson \cite{Mo}.
Similarly,
\begin{align*}w_{p-2}=&\sum_{k=0}^{\lfloor(p-2)/3\rfloor}(-1)^k3^{p-2-3k}\bi{p-2}{3k}\bi{3k}k\bi{2k}k
\\=&\sum_{k=0}^{\lfloor(p-2)/3\rfloor}(-1)^k3^{p-2-3k}\f{3k+1}{p-1}\bi{p-1}{3k+1}\bi{3k}k\bi{2k}k
\\\eq&\f19\sum_{k=0}^{p-1}(9k+3)\f{\bi{2k}k\bi{3k}k}{27^k}
\eq\f19\l(\f p3\r)+\f19\sum_{k=0}^{p-1}(9k+2)\f{\bi{2k}k\bi{3k}k}{27^k}\pmod p.
\end{align*}
By induction,
$$\sum_{k=0}^n(9k+2)\f{\bi{2k}k\bi{3k}k}{27^k}=(3n+1)(3n+2)\f{\bi{2n}n\bi{3n}n}{27^n}$$
for all $n\in\N$. In particular,
$$\sum_{k=0}^{p-1}(9k+2)\f{\bi{2k}k\bi{3k}k}{27^k}=\f{(3p-2)(3p-1)}{27^{p-1}}pC_{p-1}\bi{3p-3}{p-1}\eq0\pmod p.$$
So we have $w_k\eq(\f{-3}p)27^kw_{p-1-k}\pmod p$ for $k=0,1$. (Note that $w_0=1$ and $w_1=3$.)

Now let $k\in\{1,\ldots,p-2\}$ and assume that
$$w_j\eq\l(\f{-3}p\r)27^jw_{p-1-j}\quad\text{for all}\ j=0,\ldots,k.$$
Then
\begin{align*}&(k+1)^2w_{k+1}=(9k(k+1)+3)w_k-27k^2w_{k-1}
\\\eq&(9(p-k)(p-k-1)+3)\l(\f{-3}p\r)27^kw_{p-1-k}
-27(p-k)^2\l(\f{-3}p\r)27^{k-1}w_{p-1-(k-1)}
\\=&\l(\f{-3}p\r)27^k\times 27(p-k-1)^2w_{p-k-2}\pmod p
\end{align*}
and hence
$$w_{k+1}\eq\l(\f{-3}p\r)27^{k+1}w_{p-1-(k+1)}\pmod p.$$

In view of the above, we have proved the desired result by induction. \qed

 For Lemma \ref{Lem-w} one may also consult \cite[Corollary 3.1]{SZH}. Let $p>3$ be a prime. In view of Lemma \ref{Lem-w} and Remark \ref{Rem-Dual}(ii), we have
\begin{align*}\sum_{k=0}^{p-1}\f{w_kT_k(b,c)}{m^k}\eq&
\l(\f{-3(b^2-4c)}p\r)\sum_{k=0}^{p-1}\l(\f{27(b^2-4c)}m\r)^kw_{p-1-k}T_{p-1-k}(b,c)
\\\eq&\l(\f{-3(b^2-4c)}p\r)\sum_{k=0}^{p-1}\f{w_kT_k(b,c)}{(27(b^2-4c)/m)^k}\pmod p
\end{align*}
for any $b,c,m\in\Z$ with $p\nmid (b^2-4c)m$.

Wan and Zudilin \cite{WZ} obtained the following irrational series for $1/\pi$ involving the Legendre polynomials and the numbers $w_n$:
$$\sum_{k=0}^\infty(14k+7-\sqrt{21})w_kP_k\l(\f{\sqrt{21}}{5}\r)\l(\f{7\sqrt{21}-27}
{90}\r)^k
=\f{5\sqrt{7(7\sqrt{21}+27)}}{4\sqrt2\,\pi}.$$
Using our congruence approach (including Conjecture \ref{Conj-Dual}), we find five rational series for $1/\pi$
involving $T_n(b,c)$ and the numbers $w_n$; Theorem 1 of \cite{WZ} might be helpful to solve them.

\begin{conjecture} \label{conj-w} We have
\begin{align}
\label{w1}\sum_{k=0}^\infty\f{13k+3}{100^k}w_kT_k(14,-1)&=\f{30\sqrt2}{\pi},
\\\label{w2}\sum_{k=0}^\infty\f{14k+5}{108^k}w_kT_k(18,1)&=\f{27\sqrt3}{\pi},
\\\label{w3}\sum_{k=0}^\infty\f{19k+2}{486^k}w_kT_k(44,-2)&=\f{81\sqrt3}{4\pi},
\\\label{w4}\sum_{k=0}^\infty\f{91k+32}{(-675)^k}w_kT_k(52,1)&=\f{45\sqrt3}{2\pi},
\\\label{w5}\sum_{k=0}^\infty\f{182k+37}{756^k}w_kT_k(110,1)&=\f{315\sqrt3}{\pi}.
\end{align}
\end{conjecture}

Below we present our conjectures on congruences related to the identities
\eqref{w2} and \eqref{w5}.

\begin{conjecture}\label{conj-w2} {\rm (i)} For any $n\in\Z^+$, we have
\begin{equation}\label{w3-n}\f1n\sum_{k=0}^{n-1}(14k+5)108^{n-1-k}w_kT_k(18,1)\in\Z^+,
\end{equation}
and this number is odd if and only if $n\in\{2^a:\ a\in\N\}$.

{\rm (ii)} Let $p>3$ be a prime. Then
\begin{equation}\label{w3-p}
\sum_{k=0}^{p-1}\f{14k+5}{108^k}w_kT_k(18,1)\eq\f p4\l(27\l(\f {-3}p\r)-7\l(\f {21}p\r)\r)\pmod{p^2}.
\end{equation}
If $(\f p7)=1\ ($i.e., $p\eq1,2,4\pmod7)$, then
\begin{equation}\label{w3-pn}\f1{(pn)^2}\(\sum_{k=0}^{pn-1}\f{14k+5}{108^k}w_kT_k(18,1)-\l(\f p3\r)\sum_{k=0}^{n-1}\f{14k+5}{108^k}w_kT_k(18,1)\)\in\Z_p
\end{equation}
for all $n\in\Z^+$.

{\rm (iii)} For any prime $p>7$, we have
\begin{equation}\begin{aligned}&\l(\f p3\r)\sum_{k=0}^{p-1}\f{w_kT_k(18,1)}{108^k}
\\\eq&\begin{cases}x^2-2p\pmod{p^2}&\t{if}\ (\f p5)=(\f p7)=1\ \&\ 4p=x^2+35y^2\ (x,y\in\Z),
\\5x^2-2p\pmod{p^2}&\t{if}\ (\f p5)=(\f p7)=-1\ \&\ 4p=5x^2+7y^2\ (x,y\in\Z),
\\0\pmod{p^2}&\t{if}\ (\f{-35}p)=-1.
\end{cases}\end{aligned}\end{equation}
\end{conjecture}

\begin{conjecture}\label{conj-w5} {\rm (i)} For any $n\in\Z^+$, we have
\begin{equation}\label{w6-n}\f1n\sum_{k=0}^{n-1}(182k+37)756^{n-1-k}w_kT_k(110,1)\in\Z^+,
\end{equation}
and this number is odd if and only if $n\in\{2^a:\ a\in\N\}$.

{\rm (ii)} Let $p>3$ be a prime with $p\not=7$. Then
\begin{equation}\label{w6-p}
\sum_{k=0}^{p-1}\f{182k+37}{756^k}w_kT_k(110,1)\eq\f p4\l(265\l(\f {-3}p\r)-117\l(\f {21}p\r)\r)\pmod{p^2}.
\end{equation}
If $(\f p7)=1\ ($i.e., $p\eq1,2,4\pmod7)$, then
\begin{equation}\label{w6-pn}\f1{(pn)^2}\(\sum_{k=0}^{pn-1}\f{182k+37}{756^k}w_kT_k(110,1)-\l(\f p3\r)\sum_{k=0}^{n-1}\f{182k+37}{756^k}w_kT_k(110,1)\)\in\Z_p
\end{equation}
for all $n\in\Z^+$.

{\rm (iii)} For any prime $p>3$ with $p\not=7,13$, we have
\begin{equation}\begin{aligned}&\l(\f p3\r)\sum_{k=0}^{p-1}\f{w_kT_k(110,1)}{756^k}
\\\eq&\begin{cases}x^2-2p\pmod{p^2}&\t{if}\ (\f p7)=(\f p{13})=1\ \&\ 4p=x^2+91y^2\ (x,y\in\Z),
\\7x^2-2p\pmod{p^2}&\t{if}\ (\f p7)=(\f p{13})=-1\ \&\ 4p=7x^2+13y^2\ (x,y\in\Z),
\\0\pmod{p^2}&\t{if}\ (\f{-91}p)=-1.
\end{cases}\end{aligned}\end{equation}
\end{conjecture}

Now we give one more conjecture in this section.

\begin{conjecture}\label{conj-w1} {\rm (i)} For any integer $n>1$, we have
\begin{equation}\label{w1-n}\f1{3n2^{\lfloor(n+1)/2\rfloor}}
\sum_{k=0}^{n-1}(2k+1)54^{n-1-k}w_kT_k(10,-2)\in\Z^+.
\end{equation}

{\rm (ii)} Let $p>3$ be a prime. Then
\begin{equation}\label{w1-p}
\sum_{k=0}^{p-1}\f{2k+1}{54^k}w_kT_k(10,-2)\eq p\l(\f p3\r)+\f{p}2(2^{p-1}-1)\l(5\l(\f p3\r)+3\l(\f 3p\r)\r)\pmod{p^3}.
\end{equation}
If $p\eq1\pmod4$, then
\begin{equation}\label{w1-pn}\f1{(pn)^2}\(\sum_{k=0}^{pn-1}\f{2k+1}{54^k}w_kT_k(10,-2)-\l(\f p3\r)\sum_{k=0}^{n-1}\f{2k+1}{54^k}w_kT_k(10,-2)\)\in\Z_p
\end{equation}
for all $n\in\Z^+$.

{\rm (iii)} For any prime $p>3$, we have
\begin{equation}\begin{aligned}&\l(\f p3\r)\sum_{k=0}^{p-1}\f{w_kT_k(10,-2)}{54^k}
\\\eq&\begin{cases}4x^2-2p\pmod{p^2}&\t{if}\ 4\mid p-1\ \&\ p=x^2+4y^2\ (x,y\in\Z),
\\0\pmod{p^2}&\t{if}\ p\eq3\pmod 4.
\end{cases}\end{aligned}\end{equation}
\end{conjecture}
\begin{remark}\label{Rem-w1} For primes $p>3$ with $p\eq3\pmod4$, in general the congruence \eqref{w1-pn}
is not always valid for all $n\in\Z^+$. This does not violate Conjecture \ref{Conj-pn} since
$\lim_{k\to+\infty}|w_kT_k(10,-2)|^{1/k}=\sqrt{27}\times\sqrt{10^2-4(-2)}=54$.
If the series $\sum_{k=0}^\infty\f{2k+1}{54^k}w_kT_k(10,-2)$ converges, its value times $\pi/\sqrt3$
should be a rational number.

\end{remark}

\section{Series for $\pi$
 involving $T_n$ and related congruences}
 \setcounter{equation}{0}

 Let $p$ be an odd prime and let $a,b,c,d,m\in\Z$ with $m(b^2-4c)\not\eq0\pmod p$.
 Then
 \begin{align*}\sum_{k=1}^{p-1}\f{a+dk}{m^k}\bi{2k}k^2T_k(b,c)
 \eq&\sum_{k=1}^{(p-1)/2}\f{a+dk}{k^2m^k}\l(k\bi{2k}k\r)^2T_k(b,c)
 \\\eq&\sum_{k=1}^{(p-1)/2}\f{a+dk}{k^2m^k}\l(-\f{2p}{\bi{2(p-k)}{p-k}}\r)^2T_k(b,c)\pmod p
 \end{align*}
 with the aid of \cite[Lemma 2.1]{S11e}. Thus
 \begin{align*}&\sum_{k=1}^{p-1}\f{a+dk}{m^k}\bi{2k}k^2T_k(b,c)
 \\\eq&4p^2\sum_{k=1}^{(p-1)/2}\f{a+dk}{k^2m^k}\times\f{T_k(b,c)}{\bi{2(p-k)}{p-k}^2}
 \\\eq&4p^2\sum_{p/2<k<p}\f{a+d(p-k)}{(p-k)^2m^{p-k}}\times \f{T_{p-k}(b,c)}{\bi{2k}k^2}
 \\\eq&4p^2\sum_{k=1}^{p-1}\f{(a-dk)m^{k-1}}{k^2\bi{2k}k^2}\l(\f{b^2-4c}p\r)(b^2-4c)^{p-k}T_{p-1-(p-k)}(b,c)
 \\\eq&\l(\f{b^2-4c}p\r)4p^2\sum_{k=1}^{p-1}\f{(a-dk)T_{k-1}(b,c)}{k^2\bi{2k}k^2}\l(\f{m}{b^2-4c}\r)^{k-1}
 \pmod p
 \end{align*}
 in view of Remark \ref{Rem-Dual}(ii).

 Let $p>3$ be a prime. By the above, the author's conjectural congruence  (cf. \cite[Conjecture 1.3]{S13d})
 $$\sum_{k=0}^{p-1}(105k+44)(-1)^k\bi{2k}k^2T_k\eq p\l(20+24\l(\f p3\r)(2-3^{p-1})\r)\pmod{p^3}$$
 implies that
 $$p^2\sum_{k=1}^{p-1}\f{(105k-44)T_{k-1}}{k^2\bi{2k}k^23^{k-1}}\eq 11\l(\f p3\r)\pmod p.$$
 Motivated by this, we pose the following curious conjecture.

\begin{conjecture}\label{T} We have the following identities:
\begin{align}\sum_{k=1}^\infty\f{(105k-44)T_{k-1}}{k^2\bi{2k}k^23^{k-1}}=&\f{5\pi}{\sqrt3}+6\log3,
\\\sum_{k=2}^\infty\f{(5k-2)T_{k-1}}{(k-1)k^2\bi{2k}k^23^{k-1}}=&\f{21-2\sqrt3\,\pi-9\log3}{12}.
\end{align}
\end{conjecture}
\begin{remark} The two identities were conjectured by the author on Dec. 7, 2019. One can easily check
them numerically via {\tt Mathematica} as the two series converge fast.
\end{remark}

Now we state our related conjectures on congruences.

\begin{conjecture}
For any prime $p>3$, we have
\begin{equation}p^2\sum_{k=1}^{p-1}\f{(105k-44)T_{k-1}}{k^2\bi{2k}k^23^{k-1}}
\eq 11\l(\f p3\r)+\f p2\l(13-35\l(\f p3\r)\r)\pmod{p^2}
\end{equation}
and
\begin{equation}p^2\sum_{k=2}^{p-1}\f{(5k-2)T_{k-1}}{(k-1)k^2\bi{2k}k^23^{k-1}}
\eq-\f12\l(\f p3\r)-\f p8\l(7+\l(\f p3\r)\r)\pmod{p^2}.
\end{equation}
\end{conjecture}

\begin{conjecture} {\rm (i)}
We have
$$\f1{n\bi{2n}n}\sum_{k=0}^{n-1}(-1)^{n-1-k}(5k+2)\bi{2k}kC_kT_k\in\Z^+$$
for all $n\in\Z^+$, and also
$$\sum_{k=0}^{p-1}(-1)^k(5k+2)\bi{2k}kC_kT_k\eq 2p\l(1-\l(\f p3\r)(3^p-3)\r)\pmod{p^3}$$
for each prime $p>3$.

{\rm (ii)} For any prime $p\eq1\pmod3$ and $n\in\Z^+$, we have
\begin{equation}\label{ctpn}\begin{aligned}&\f{\sum_{k=0}^{pn-1}(-1)^k(5k+2)\bi{2k}kC_kT_k-p\sum_{k=0}^{n-1}(-1)^k(5k+2)\bi{2k}kC_kT_k}
{(pn)^2\bi{2n}n^2}
\\\quad\qquad&\eq\l(\f p3\r)\f{3^p-3}{2p}(-1)^nT_{n-1}\pmod{p}.
\end{aligned}\end{equation}
\end{conjecture}
\begin{remark} See also \cite[Conjecture 67]{S19} for a similar conjecture.
\end{remark}

Let $p$ be an odd prime. We conjecture that
\begin{equation}\label{T3-4}\sum_{k=0}^{p-1}\f{8k+3}{(-16)^k}\bi{2k}k^2T_k(3,-4)\eq p\l(1+2\l(\f{-1}p\r)\r)\pmod{p^2}
\end{equation}
and
\begin{equation}\label{T8-2}\sum_{k=0}^{p-1}\f{33k+14}{4^k}\bi{2k}k^2T_k(8,-2)\eq p\l(6\l(\f{-1}p\r)+8\l(\f{2}p\r)\r)\pmod{p^2}.
\end{equation}
Though \eqref{T3-4} implies the congruence
$$p^2\sum_{k=1}^{p-1}\f{(8k-3)T_{k-1}(3,-4)}{k^2\bi{2k}k^2}\l(-\f{16}{25}\r)^{k-1}\eq\f 34\pmod p,$$
and \eqref{T8-2} with $p>3$ implies the congruence
$$p^2\sum_{k=1}^{p-1}\f{(33k-14)T_{k-1}(8,-2)}{k^2\bi{2k}k^218^{k-1}}\eq\f 7{2}\l(\f 2p\r)\pmod p,$$
we are unable to find the exact values of the two converging series
$$\sum_{k=1}^\infty\f{(8k-3)T_{k-1}(3,-4)}{k^2\bi{2k}k^2}\l(-\f{16}{25}\r)^{k-1}
\ \ \text{and}\ \ \sum_{k=1}^{\infty}\f{(33k-14)T_{k-1}(8,-2)}{k^2\bi{2k}k^218^{k-1}}.$$

\section*{Acknowledgment} The author would like to thank Prof. Qing-Hu Hou at Tianjin Univ. for his helpful
comments on the proof of Lemma \ref{Lem2.3}.

\medskip
Received for publication January 2020.
\medskip

\end{document}